\documentclass[a4paper,11pt]{amsart}
\usepackage[latin1]{inputenc}
\usepackage{amsmath, amsthm, amssymb, amsfonts, amscd}
\usepackage[english]{babel}
\usepackage[all]{xy}
\usepackage[linktocpage=true]{hyperref}
\usepackage{color}
\usepackage{pinlabel}
\usepackage{tikz-cd}
\usepackage{enumitem}
\newtheorem{thm}{Theorem}[section]
\newtheorem*{thm*}{Theorem}
\newtheorem{lemma}[thm]{Lemma}
\newtheorem{prop}[thm]{Proposition}
\newtheorem{cor}[thm]{Corollary}
\newtheorem{defi}[thm]{Definition}

\newtheorem{ques}[thm]{Question}

\theoremstyle{remark}

\newtheorem{rmk}[thm]{Remark}

\numberwithin{figure}{section}
\numberwithin{equation}{section}
\newcommand\restr[2]{{
  \left.\kern-\nulldelimiterspace 
  #1 
  \vphantom{\big|} 
  \right|_{#2} 
  }}

\title{Double branched covers of knotoids}
\author[A. Barbensi, D. Buck, H. A. Harrington, M. Lackenby]{Agnese Barbensi, Dorothy Buck, \\ Heather A. Harrington, Marc Lackenby}
\address{AB, HAH, ML: Mathematical Institute, University of Oxford, Oxford, UK.}
\address{DB: Department of Mathematical Sciences, University of Bath, Bath, UK and Mathematics/Biology, Trinity College of Arts \& Sciences, Duke University, Durham, NC, USA.}
\begin{document}

\begin{abstract}
By using double branched covers, we prove that there is a $1$-$1$ correspondence between the set of knotoids in $S^2$, up to orientation reversion and rotation, and knots with a strong inversion, up to conjugacy. This correspondence allows us to study knotoids through tools and invariants coming from knot theory. In particular, concepts from geometrisation generalise to knotoids, allowing us to characterise reversibility and other properties in the hyperbolic case. Moreover, with our construction we are able to detect both the trivial knotoid in $S^2$ and the trivial knotoid in $D^2$. 
 
\end{abstract}

\maketitle
\section{Introduction}
Knotoids were recently defined by V.Turaev \cite{turaev} as a generalisation of knots in $S^3$. More precisely, knotoids are defined as equivalence classes of diagrams of oriented arcs in $S^2$ or in $D^2$ up to an appropriate set of moves and isotopies. Some examples of knotoids are shown in Figure \ref{fig:examplesofknotoids}
\begin{figure}[h]
\includegraphics[width=9cm]{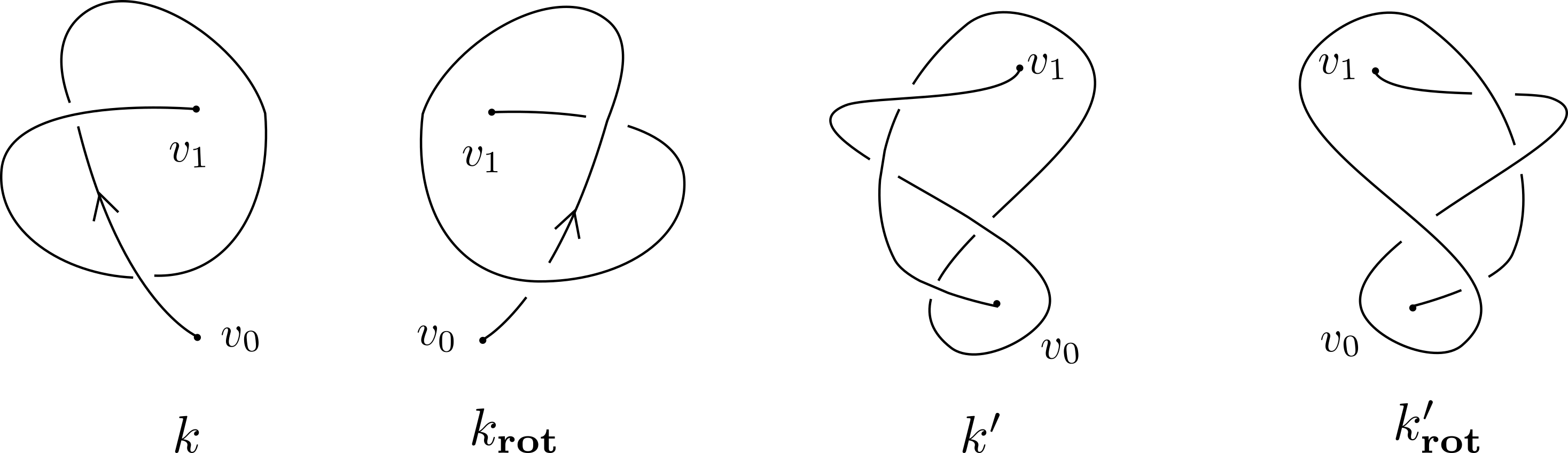}
\caption{Two knotoids $k$, $k'$ and their rotations $k_{\textbf{rot}}$, $k'_{\textbf{rot}}$. }
\label{fig:examplesofknotoids}
\end{figure}

Like knots, knotoids admit natural involutive operations such as mirror reflection and reversion. In addition, it is possible to define a further modification for knotoids called \emph{rotation}, formally defined in Section \ref{sec:Preliminares}. As an example, we show two knotoids $k$ and $k'$ with their corresponding rotations $k_{\textbf{rot}}$ and $k'_{\textbf{rot}}$ in Figure \ref{fig:examplesofknotoids}. Several invariants for knotoids have been adapted from classical knot theory, such as various versions of the bracket polynomial (see \emph{e.g} \cite{turaev}, \cite{knotoids}), and there are several well defined maps that associate a classical knot to a knotoid (see \emph{e.g} \cite{turaev}, \cite{knotoids}, \cite{knotoidstorus}). Often, non-equivalent knotoids share the same image under these maps, and it is possible to exhibit examples of non-equivalent knotoids with the same bracket polynomials. 

Here we present a complete invariant, that associates a knot to a knotoid through a double branched cover construction. Knotoids admit a $3$-dimensional interpretation as equivalence classes of embedded oriented arcs in $\mathbb{R}^3$ with endpoints lying on two fixed vertical lines. In this setting, given a knotoid $k$, its preimage in the double cover of $\mathbb{R}^3$ branched along these lines is a simple closed curve, that can be viewed as giving a classical knot $K$ in $S^3$. By construction, knots arising as pre-images of knotoids are \emph{strongly invertible}, that is, there exists an involution $\tau$ of $S^3$ mapping the knot to itself, preserving the orientation of $S^3$ and reversing the one of the knot. We exploit properties of strongly invertible knots to prove our main result in Section \ref{sec:strong}.

\begin{thm}\label{thm:completo}
There is a $1$-$1$ correspondence between unoriented knotoids, up to rotation, and knots $K$ with a strong inversion $\tau$, up to conjugacy. 
\end{thm}

As a corollary of Theorem \ref{thm:completo}, we have the following.

\begin{cor}\label{cor:completo2}
Given any torus knot $K_t$ there is exactly $1$ knotoid associated to it, up to rotation and reversion. Given any strongly invertible hyperbolic knot $K_h$ there either are either $1$ or $2$ knotoids associated to it up to reversion and rotation, depending on whether or not $K_h$ is periodic with period $2$. In general, given any strongly invertible knot $K$ there are only finitely many knotoids associated to it. 
\end{cor}
In particular, Corollary \ref{cor:completo2} implies that there are at most $4$ oriented knotoids associated to every torus knot, and at most $8$ associated to every hyperbolic knot. In fact, we will obtain more precise information about these numbers of knotoids in Corollary \ref{cor:completo3} below. We emphasise that Theorem \ref{thm:completo} and Corollary \ref{cor:completo2} provide a powerful link between knot theory and knotoids, allowing us to borrow all the sophisticated tools developed to distinguish knots to study knotoids. In particular, we can extend the concepts from geometrisation to knotoids: we will call \emph{hyperbolic} (respectively \emph{torus}) knotoids those lifting to hyperbolic (respectively torus) knots. Even if our invariant cannnot distinguish between a knotoid $k$, its reverse $-k$ and its rotation $k_{\textbf{rot}}$, in the case of hyperbolic knotoids these symmetries are completely characterised. A knotoid $k$ is called \emph{reversible} (respectively \emph{rotatable}) if it is equivalent to $-k$ (respectively $k_{\textbf{rot}}$).
 
\begin{thm}\label{llll}
A hyperbolic, oriented knotoid $k \in \mathbb{K}(S^2) $ is reversible if and only if its double branched cover has cyclic period $2$. Analogously, it is equivalent to the reverse of its rotation if and only if its double branched cover has free period $2$.
\end{thm}

Furthermore, hyperbolic knotoids are never rotatable.

\begin{thm}\label{marcrot}
A hyperbolic knotoid is never rotatable.
\end{thm}

As a consequence of these results we obtain the following.

\begin{cor}\label{cor:completo3}
Given any strongly invertible hyperbolic knot $K$ there are exactly $4$ oriented knotoids associated to it. Moreover, one of the following holds.
\begin{itemize}
\item If $K$ has cyclic period $2$, these are two inequivalent reversible knotoids $k^1$, $k^2$ and their rotations $k^1_{\textbf{rot}}$, $k^2_{\textbf{rot}}$;

 \item if $K$ has free period $2$, these are two inequivalent knotoids $k^1$, $k^2$ (each equivalent to the reverse of its rotation) and their reverses $-k^1$, $-k^2$;
 \item if $K$ does not have period $2$, these are a knotoid $k$, its reverse $-k$, its rotation $k_{\textbf{rot}}$ and its reverse rotation $-k_{\textbf{rot}}$. 
 \end{itemize}

\end{cor}
As a further corollary, we are able to count the number of inequivalent involutions in the symmetry group of some particular composite knot.

\begin{prop}\label{compositeknots}
Consider a knot $K$ isotopic to the connected sum of $\#_{i=1}^n K^i_h$, where $n \geq 2$ and every $K_h^i$ is a strongly invertible, hyperbolic knot. Suppose that these hyperbolic knots are pairwise distinct. Then, the number of non-equivalent strong involutions of $K$ is equal to $4^{n-1}(n!)$.
\end{prop}

As there is an algorithm to decide whether two hyperbolic knots are equivalent (see \cite{manning} and \cite{kupe}) and there is an algorithm to decide whether two involutions of a hyperbolic knot complement are conjugate (see \emph{e.g.} Theorems $8.2$ and $8.3$ of \cite{Lackenbynew}), Theorem \ref{thm:completo}  implies the following stronger result.

\begin{thm}\label{thm:algo}
Given two hyperbolic knotoids $k_1$ and $k_2$, there is an algorithm to determine whether $k_1$ and $k_2$ are equivalent as oriented knotoids.  
\end{thm}

Our construction enables us to distinguish knotoids sharing the same class in $S^2$ that are inequivalent as knotoids in $D^2$ (precise definitions are given in Section \ref{sec:Preliminares}). In particular, we can detect the trivial planar knotoid $k_0^{pl}$.

\begin{thm}\label{thm:banaleintro}
A planar knotoid $k $  lifts to a knot isotopic to the core of the solid torus if and only if $k$ is the trivial planar knotoid $k^{\textrm{pl}}_0$.
 
\end{thm}

\subsection*{Structure of the paper}
The paper is structured as follows. After recalling some basics on knotoids in Section \ref{sec:Preliminares}, we present the map defined by the double branched cover in Section \ref{sec:dbc}. We recall the $1$-$1$ correspondence between knotoids and isotopy classes of simple $\theta$-curves, following \cite{turaev}, in Section \ref{sec:dbc}. In Section \ref{sec:dbctheta}, results from \cite{theta} are translated in terms of knotoids. In particular, Theorem \ref{thm:banalegen} allows us to detect the trivial knotoid $k_0$ among all the other knotoids. In Section \ref{sec:trivialdetection} we prove a slightly more powerful version of Theorem \ref{thm:banalegen}, Theorem \ref{thm:banaleintro}, enabling the detection of the trivial planar knotoid $k^{\textrm{pl}}_0$ (see Section \ref{sec:Preliminares} for the precise definitions). Section \ref{sec:strong} is devoted to the proof of Theorems \ref{thm:completo} and \ref{thm:algo}. In Section \ref{sec:inam} we prove Theorem \ref{llll} and Theorem \ref{marcrot} together with Corollary \ref{cor:completo3} and Proposition \ref{compositeknots}. In Section \ref{sec:dimos} we show that our construction can be used to distinguish between planar knotoids that are equivalent in $S^2$.
Finally, in Section \ref{sec:gauss} we describe an algorithm that produces the Gauss code for the lift of a knotoid $k$ given the Gauss code for $k$.

All maps and manifolds are assumed to be smooth, and we use the following notation:
\begin{itemize}
 \item $\mathbb{K}(X)$ and $\mathbb{K}(X)/_{\sim}$ are the sets of oriented and unoriented knotoid diagrams, respectively, up to equivalence in $X$, where $X = S^2$ or $ \mathbb{R}^2$;
 \item $\mathbb{K}(X)/_{\approx}$ is the set of unoriented knotoid diagrams up to equivalence in $X$, where $X = S^2$ or $ \mathbb{R}^2$, and up to rotation;
 \item $\Theta^s$ is the set of simple labelled $\theta$-curves in $S^3$;
 \item $\Theta^s/_ \sim$ and $\Theta^s/_\approx$ are the sets of simple labelled $\theta$-curves in $S^3$ up to relabelling the vertices, and up to relabelling the vertices and the edges $e_-$ and $e_+$, respectively;
 \item $\mathcal{K}(Y)$ is the set of knots in $Y$, where $Y = S^3$ or $ S^1 \times D^2$;
 \item $\mathcal{K}_{S.I.}(S^3)$ is the subset of $\mathcal{K}(S^3)$ consisting of knots that admit a strong inversion;
 \item $\mathcal{K}SI(S^3)$ is the set of strongly invertible knots $(K,\tau)$ in $S^3$.
\end{itemize}

\vspace{0.8cm}
\textbf{Acknowledgements:}
A.B. would like to thank D. Celoria, M. Golla and M. Nagel for interesting conversations. A.B. is supported by the EPSRC grant ``Algebraic and topological approaches for genomic data in molecular biology'' EP/R005125/1. D.B. gratefully acknowledges support from the Leverhulme Trust, Grant RP2013-K-017. The authors are grateful to D. Goundaroulis for providing us with the examples of Section \ref{sec:dimos}. The authors would like to thank the referees, whose helpful suggestions led to considerable improvements to the paper.

\section{Preliminares}\label{sec:Preliminares}
A \emph{knotoid diagram} in $S^2$ is a generic immersion of the interval $[0,1]$ in $S^2$ with finitely many transverse double points endowed with over/under-crossing data. The images of the points $0$ and $1$ are distinct from the other points and from each other. The \emph{endpoints} of a knotoid diagram are called the \emph{tail} and the \emph{head} respectively, and denoted by $v_0$ and $v_1$. Knotoid diagrams are oriented from the tail to the head, see Figures \ref{fig:iotanonsurg} and \ref{fig:overclosure}. 

\begin{defi}
A \emph{knotoid} is an equivalence class of knotoid diagrams on the sphere considered up to isotopies of $S^2$ and the three classical Reidemeister moves (see Figure \ref{reidemoves}), performed away from the endpoints.  

\end{defi}

It is not permitted to pull the strand adjacent to an endpoint over/under a transversal strand as shown in Figure \ref{fig:forbidden}. Notice that allowing such moves produces a trivial theory: namely, any knotoid diagram can be transformed into a crossingless one by a finite sequence of forbidden moves. 

\begin{figure}[h]
\includegraphics[width=5cm]{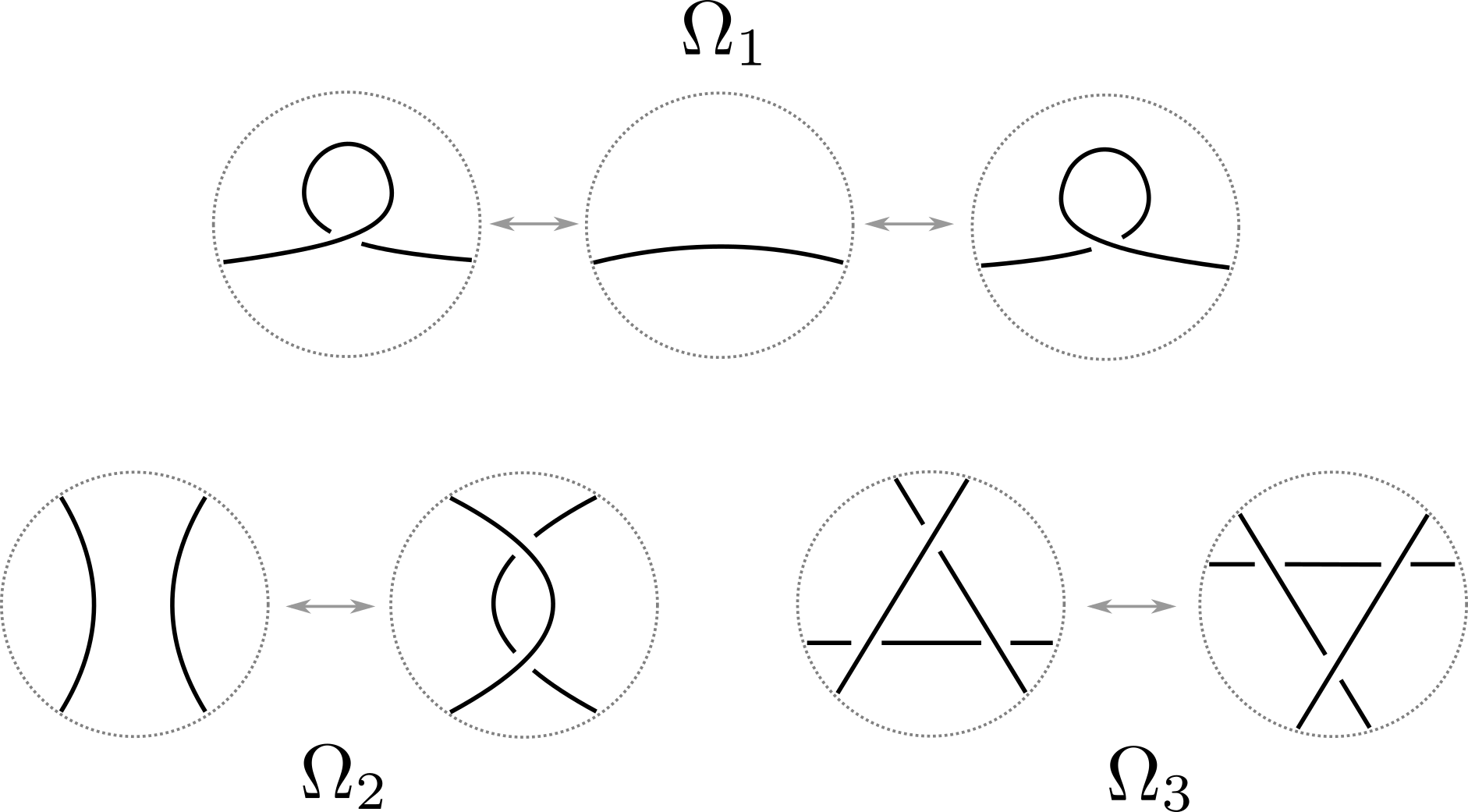}
\caption{The classical Reidemeister moves.}
\label{reidemoves}
\end{figure}

The \emph{trivial knotoid} $k_0$ is the equivalence class of the crossingless knotoid diagram.\\

\begin{figure}[h]
\includegraphics[width=7cm]{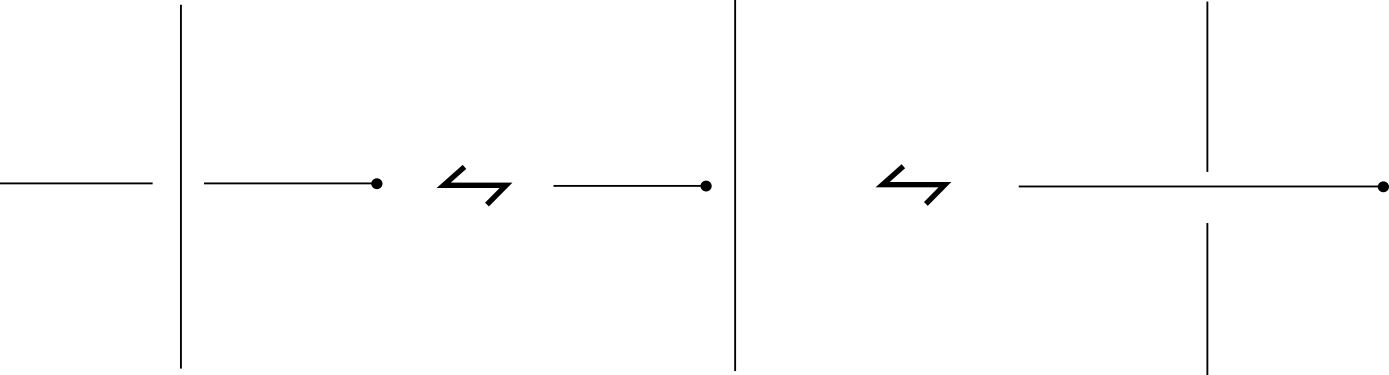}
\caption{Forbidden moves.}
\label{fig:forbidden}
\end{figure}

Any knotoid $k$ in $S^2$ can be represented by several knotoid diagrams in $\mathbb{R}^2$, by choosing different stereographic projections. Two knotoid diagrams in $\mathbb{R}^2$ are said to be \emph{equivalent} if they are related by planar isotopies and a finite sequence of Reidemeister moves, performed away from the endpoints. We will denote by $k^{\textrm{pl}}_0$ the equivalence class in $\mathbb{R}^2$ of the crossingless knotoid diagram.

Let us denote the set of oriented knotoids in the plane and in the sphere by $\mathbb{K}(\mathbb{R}^2)$ and $\mathbb{K}(S^2)$ respectively. We can define the map $$\iota: \mathbb{K}(\mathbb{R}^2) \longrightarrow \mathbb{K}(S^2) $$
induced by the inclusion $\mathbb{R}^2 \hookrightarrow S^2 = \mathbb{R}^2 \cup \infty$. The map $\iota$ is surjective but not injective (see Figure \ref{fig:iotanonsurg} for an example, or \cite{knotoids} for more details).

\begin{figure}[h]
\includegraphics[width=5cm]{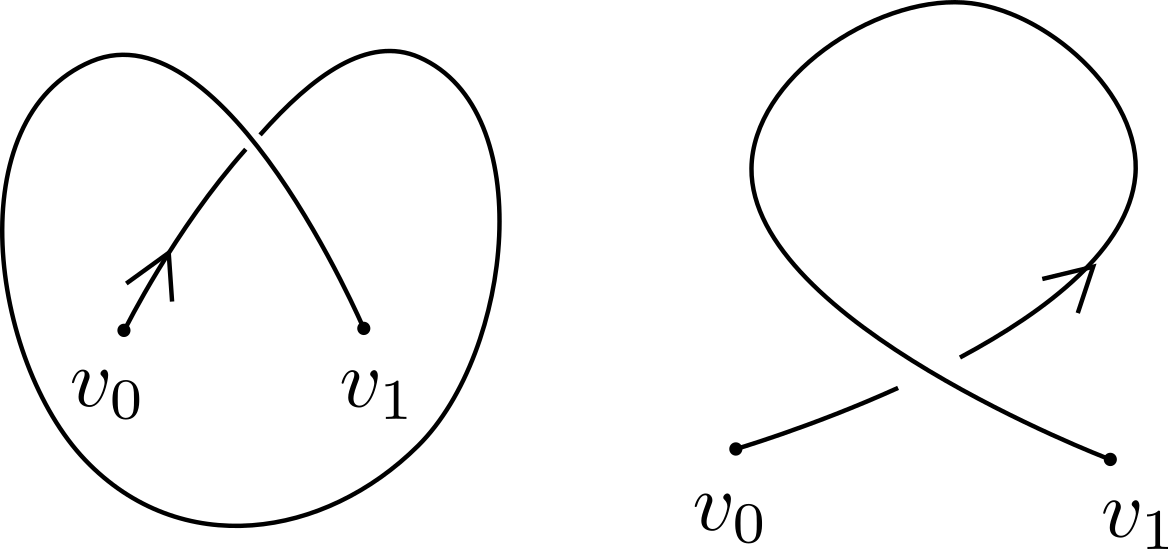}
\caption{The diagrams in the picture represent inequivalent knotoids $k_1, k_2$ in $\mathbb{K}(\mathbb{R}^2)$, but they are both sent into the trivial knotoid in $\mathbb{K}(S^2)$ under the map $\iota$.}
\label{fig:iotanonsurg}
\end{figure}

There is a natural way to associate two different knots to any knotoid through the \emph{underpass closure map} and the \emph{overpass closure map}, defined in \cite{turaev} and \cite{knotoids} and denoted by $\omega_-$ and $\omega_+$ respectively. Given a diagram representing a knotoid $k$, a diagram representing $\omega_-(k)$ (respectively $\omega_+ (k)$)  is obtained by connecting the two endpoints by an arc embedded in $S^2$ which is declared to go under (respectively over) each strand it meets during the connection.

\begin{rmk}\label{notenough}
Different knotoids may have the same image under $\omega_+$ and $\omega_-$, see for example the knotoids in Figure \ref{fig:overclosure}. In Section \ref{sec:dbc} we are going to present a more subtle way to associate a knot to a knotoid, which allows for a finer classification.

\end{rmk}

\begin{figure}[h]
\includegraphics[width=5cm]{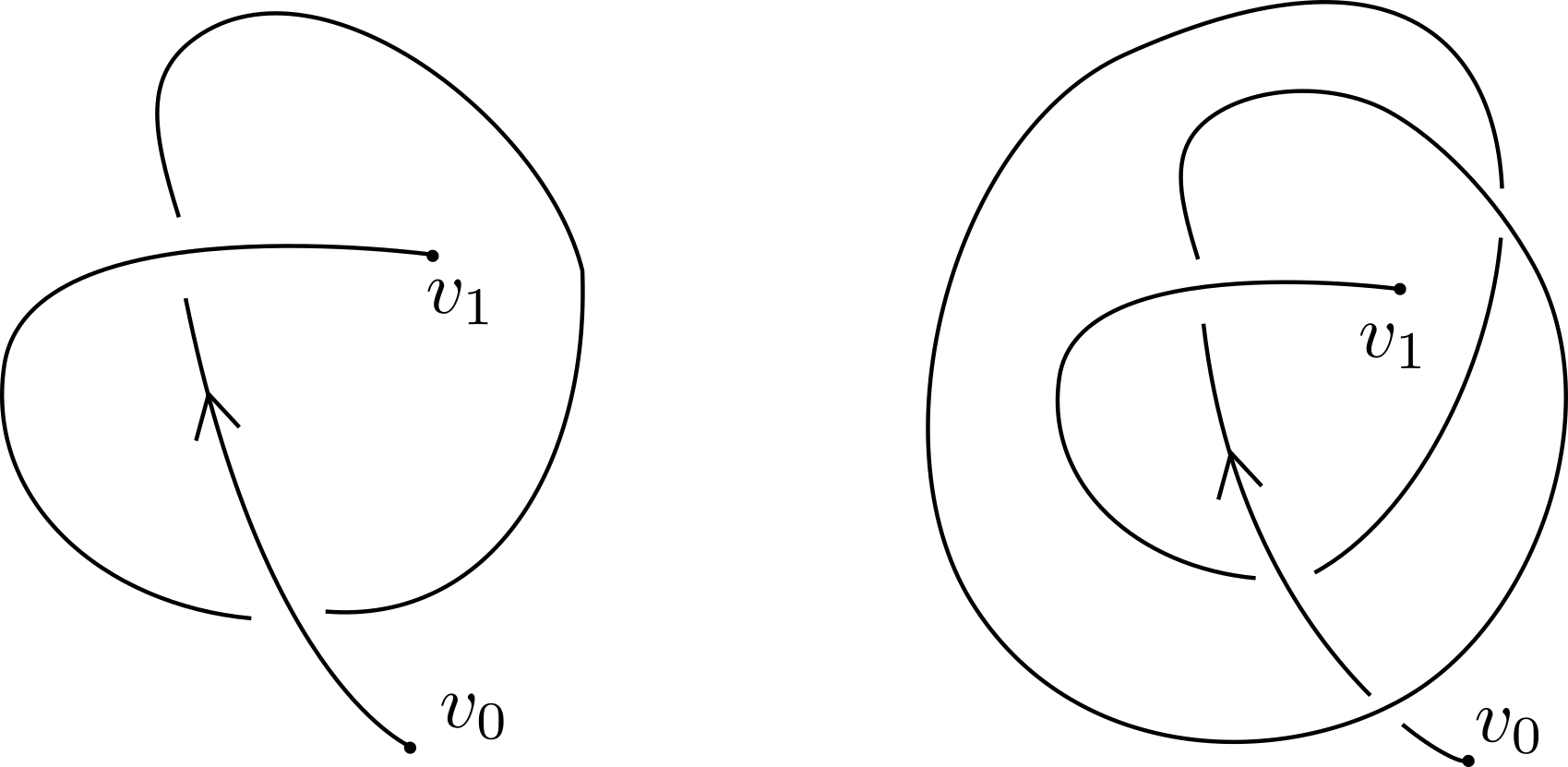}
\caption{The images under $\omega_-$ of the two knotoids in the figure are both the trefoil knot, and the images under $\omega_+$ are both the trivial knot, but the knotoids are not equivalent.}
\label{fig:overclosure}
\end{figure}

Moreover, every knot $ K \subset  S^3$ arises as the image under $\omega_{\pm} $ of a knotoid diagram. Indeed, choose any diagram representing $K$, and cut out an arc that does not contain any crossings, or that contains only crossings which are overcrossings (respectively undercrossings). This results in creating a knotoid diagram, whose image under $\omega_+$ (respectively $\omega_-$) is the starting knot $K$. It is important to notice that different choices of arcs in $K$ may result in non-equivalent knotoid diagrams. However, choosing the arc to be crossingless induces a well defined injective map $\alpha$ from the set of knots in $S^3$ to $\mathbb{K}(S^2)$ (see \cite{turaev}, \cite{knotoids} for more details).  

\begin{defi}
Knotoids in $\mathbb{K}(S^2)$ that are contained in the image of $\alpha$ are called \emph{knot-type knotoids}. Equivalently, a knotoid is a knot-type knotoid if and only if it admits a diagram in which the endpoints lie in the same region (see Figure \ref{fig:iotanonsurg}). 
The other knotoids are called proper knotoids (see Figure \ref{fig:examplesofknotoids}).
\end{defi}

There is a $1$-$1$ correspondence between knot-type knotoids and classical knots: knot-type knotoids may be thought as long knots, and (see \emph{e.g.} \cite{long}) closing the endpoints of a long knot produces a classical knot carrying the same knotting information. Thus, we can conclude that a knot-type knotoid can be considered the same as the knot it represents.

\subsection{Involutions}\label{sec:symmetries}
Turaev \cite{turaev} defines three commuting involutive operations on knotoids in $\mathbb{K}(S^2)$. These operations are called \emph{reversion}, \emph{mirror reflection} and \emph{symmetry}. The first two operations are borrowed from knot theory: reversion has the effect of changing the orientation of a knotoid, or, in other words, of exchanging the labels of the endpoints, and mirror reflection transforms a knotoid into a knotoid represented by the same diagrams with all the crossings changed. 
The third involution is defined by the extension to $S^2$ of the reflection of the plane $\mathbb{R}^2$ along the horizontal line passing through the endpoints, see Figure \ref{fig:symmetries}.

\begin{figure}[h]
\includegraphics[width=12cm]{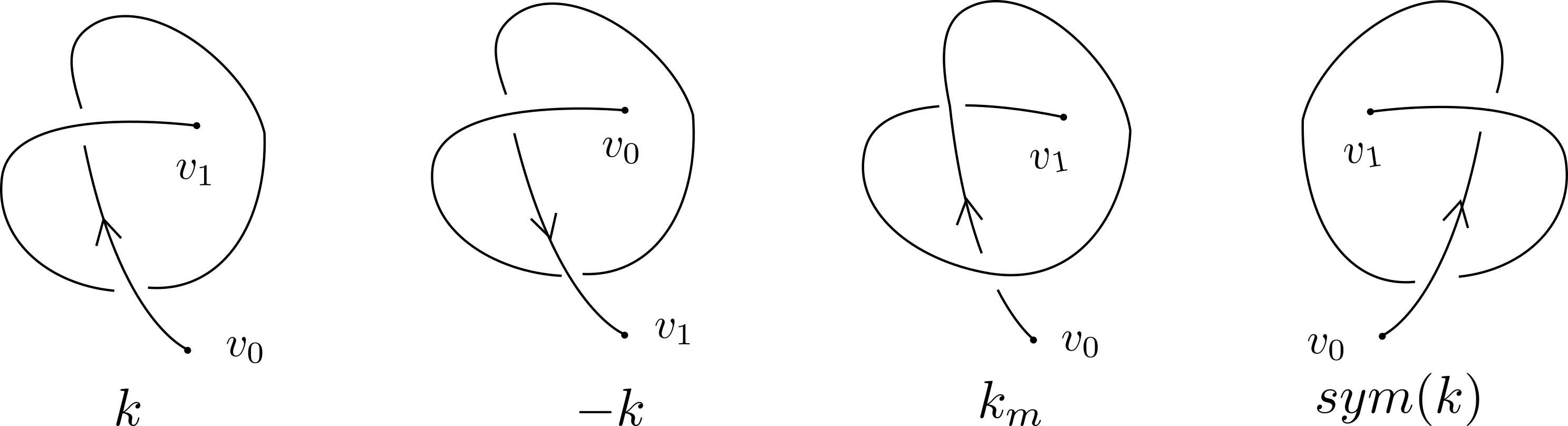}
\caption{From left to right, a knotoid $k$, its reverse $-k$, its mirror image $k_m$ and its symmetric $sym(k)$.}
\label{fig:symmetries}
\end{figure}

We will find it useful to define the involution obtained as the composition of symmetry and mirror reflection.

\begin{defi}\label{rotatable}
Two knotoids $k$ and $k_{\textbf{rot}} \in \mathbb{K}(S^2)$ differ by a \emph{rotation} if they have diagrams that are obtained from each other by reflecting the sphere of the diagram in a line through the endpoints of the knotoids, followed by a mirror reflection, see Figure \ref{fig:examplesofknotoids}. If a knotoid is unchanged by a rotation, we say that it is \emph{rotatable}.
\end{defi}

\begin{rmk}\label{noinvoknots}
Note that for knot-type knotoids symmetry and mirror reflection coincide (see \cite{turaev}). In particular, every knot-type knotoid is rotatable.
 
\end{rmk}

\subsection{Multiplication of Knotoids}\label{sub:molt}
In \cite{turaev} an analogue for the connected sum of knots is defined: the multiplication of knotoids. Note that each endpoint of a knotoid diagram $k$ in $S^2$ admits a neighbourhood $D$ such that $k$ intersects it in exactly one arc (a radius) of $D$. Such a neighbourhood is called a \emph{regular neighbourhood} of the endpoint. Given two diagrams in $S^2$ representing the knotoids $k_1$ and $k_2$, equipped with a regular neighbourhood $D_1$ for the head of $k_1$ and $D_2$ for the tail of $k_2$, the \emph{product knotoid} $k = k_1 \cdot k_2$ is defined as the equivalence class in $\mathbb{K}(S^2)$ of the diagram obtained by gluing $S^2 \setminus \text{int}(D_1)$ to $S^2 \setminus \text{int}(D_2)$ through an orientation-reversing homeomorphism $\partial D_1 \longrightarrow \partial D_2$ mapping the only point in $\partial D_1 \cap k_1$ to the only point in $\partial D_2 \cap k_2$. Note that this operation is not commutative (see \cite{turaev}, Section $4$).

\begin{defi}
 A knotoid $k$ in $\mathbb{K}(S^2)$ is called \emph{prime} if it is not the trivial knotoid and $k = k_1 \cdot k_2$ implies that either $k_1$ or $k_2$ is the trivial knotoid.
 
\end{defi}

This multiplication operation has been extensively studied in \cite{turaev}, where the following result on prime decomposition is proven.

\begin{thm}[Theorem $4.2$, \cite{turaev}]\label{unicadecomp}
Every knotoid $k$ in $\mathbb{K}(S^2)$ expands as a product of prime knotoids. 
\end{thm}

Moreover, the expansion as a product is unique up to the identity $$k \cdot k' = k' \cdot k$$ where $k'$ is a knot-type knotoid, and the multiplication operation turns $\mathbb{K}(S^2)$ into a semigroup. 

\begin{rmk}\label{rmk:prodpiano}
Since the surface in which the diagram of $k = k_1 \cdot k_2$ lies is the $2$-sphere obtained as the connected sum between the $2$-spheres containing the diagrams of $k_1$ and $k_2$, the operation of multiplication is well defined only in $\mathbb{K}(S^2)$. A diagram in the plane for $k$ can be obtained by drawing the tail of $k_2$ in the \emph{external} region of the diagram, as shown in Figure \ref{fig:prodotto}.
 
\end{rmk}

\begin{figure}[h]
\includegraphics[width=4cm]{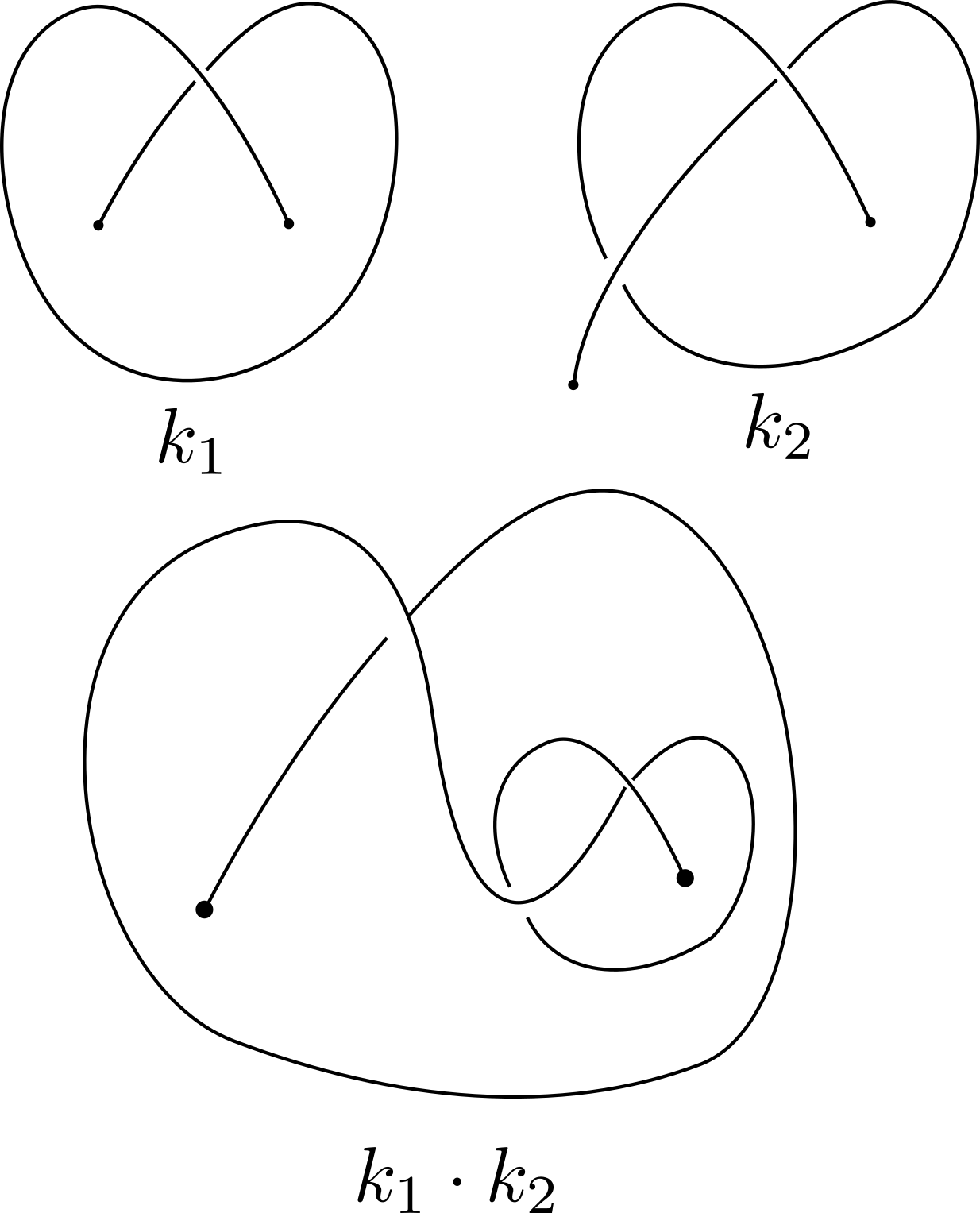}
\caption{On the bottom line, a diagram representing the product $k_1 \cdot k_2$ of the knotoids in the upper line.}
\label{fig:prodotto}
\end{figure}

Note that the orientation is required in order to define the multiplication operation. In particular, given a knotoid $k \in \mathbb{K}(S^2)$,  call $-k$ the knotoid represented by the same diagrams as $k$, but with opposite orientation. Then, the following relations hold. 

\begin{itemize}
 \item $k_1 \cdot k_2 = - (-k_2 \cdot -k_1) $
 \item $-k_1 \cdot k_2 = - (-k_2 \cdot k_1) $
 \item $k_1 \cdot -k_2 = - (k_2 \cdot -k_1) $
 \item $k_2 \cdot k_1 = - (-k_1 \cdot -k_2) $
\end{itemize}

We will sometimes find it useful to suppress the orientation. To this end, we will call $\mathbb{K}(S^2)/_{\sim}$ and $\mathbb{K}(\mathbb{R}^2)/_{\sim}$ the sets of unoriented knotoids in the sphere and in the plane, respectively. Note that, in general, the products $k_1 \cdot k_2$, $-k_1 \cdot k_2$, $k_1 \cdot -k_2$ and $k_2 \cdot k_1$ represent non-equivalent classes of unoriented knotoids.

\subsection{Bracket polynomial}

The bracket polynomial of oriented knotoids in $\mathbb{K}(S^2)$ or in $\mathbb{K}(\mathbb{R})^2$ was defined in \cite{turaev}, by extending the state expansion of the bracket polynomial of knots. The definition can be given in terms of a skein relation, with the appropriate normalisations, as for the bracket polynomial of knots. A normalisation of the bracket polynomial of knotoids gives rise to a knotoid invariant generalising the Jones polynomial of knots (after a change of variable). A version of the bracket polynomial (extended bracket polynomial, defined in \cite{turaev}) is used in \cite{tableknotoids} to distinguish knotoids taken from a list containing diagrams with up to $5$ crossings. \\

Although bracket polynomials are useful invariants, it is fairly simple to produce examples of oriented knotoids that cannot be distinguished by them. One way to construct such examples is by using the concept of \emph{mutation}.
Recall that the mutation of an oriented knot $K$ can be described as follows. 
Consider a diagram for $K$, and a $2$-tangle $R$ as in Figure \ref{fig:tangle}. New knots $K'_i$ can be formed by replacing the tangle $R$ with the tangle $R' = \rho_i(R)$ given by rotating $R$ by $\pi$ in one of three ways described on the right side of Figure \ref{fig:tangle}. Each of these three knots is called a mutant of $K$.

\begin{figure}[h]
\includegraphics[width=7cm]{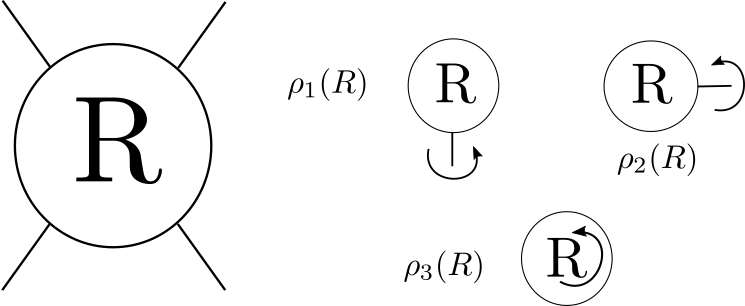}
\caption{A portion of a knot diagram for $K$ contained in the $2$-tangle $R$. By rotating $R$ as in the right side of the picture, we obtain new knots $K'_1$, $K'_2$ and $K'_3$. Each of these knots is called a mutant of $K$.}
\label{fig:tangle}
\end{figure}
The probably best known example of non-equivalent mutant knots is the Conway and Kinoshita-Teresaka pair shown in Figure \ref{fig:mutanti1}. 

\begin{figure}[h]
\includegraphics[width=6cm]{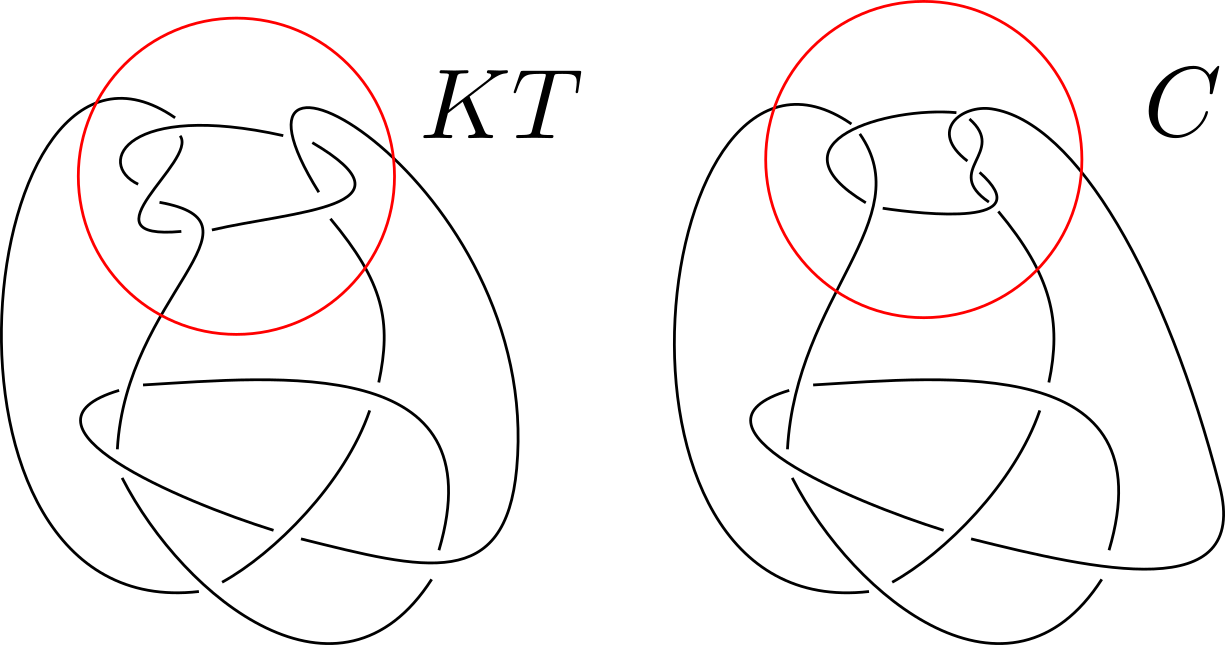}
\caption{The Kinoshita-Teresaka knot $KT$ (on the left) and the Conway knot $C$ (on the right). It was shown by Gabai (\cite{gabai}) that the genus of $KT$ is $2$ while $C$ has genus $3$, thus, they are inequivalent.}
\label{fig:mutanti1}
\end{figure}

Mutation can be generalised to knotoids, by requiring that both the endpoints of a knotoid $k$ lie outside of the tangle $R$ that is rotated.

\begin{rmk}\label{rmk:poliuguale}
It is well known that mutant knots share the same bracket and Jones polynomial. The same result also holds for knotoids, and a proof can be produced in exactly the same way as for knots by using Conway's linear skein theory (see \emph{e.g.} \cite{mutation1} and \cite{mutation2}).
\end{rmk}

Consider knot-tpe knotoids associated to the knots $KT$ and $C$ shown in Figure \ref{fig:mutanti1}; by construction they are non-equivalent, and Remark \ref{rmk:poliuguale} implies that they share the same bracket polynomials. By taking the products $k_1 = KT \cdot k $, $k_2 = C \cdot k $, where $k$ is any proper knotoid (see Figure \ref{fig:mutanti2}) we obtain two proper knotoids\footnote{As a consequence of \cite{turaev}, Theorem $4.2$, the product of two knotoids $k_1$ and $k_2$ is a knot-type knotoid if and only if both $k_1$ and $k_2$ are knot-type knotoids.} with the same bracket polynomials. We will prove in Section \ref{sec:dbctheta} that $k_1$ and $k_2$ are non-equivalent.

\begin{figure}[h]
\includegraphics[width=5.5cm]{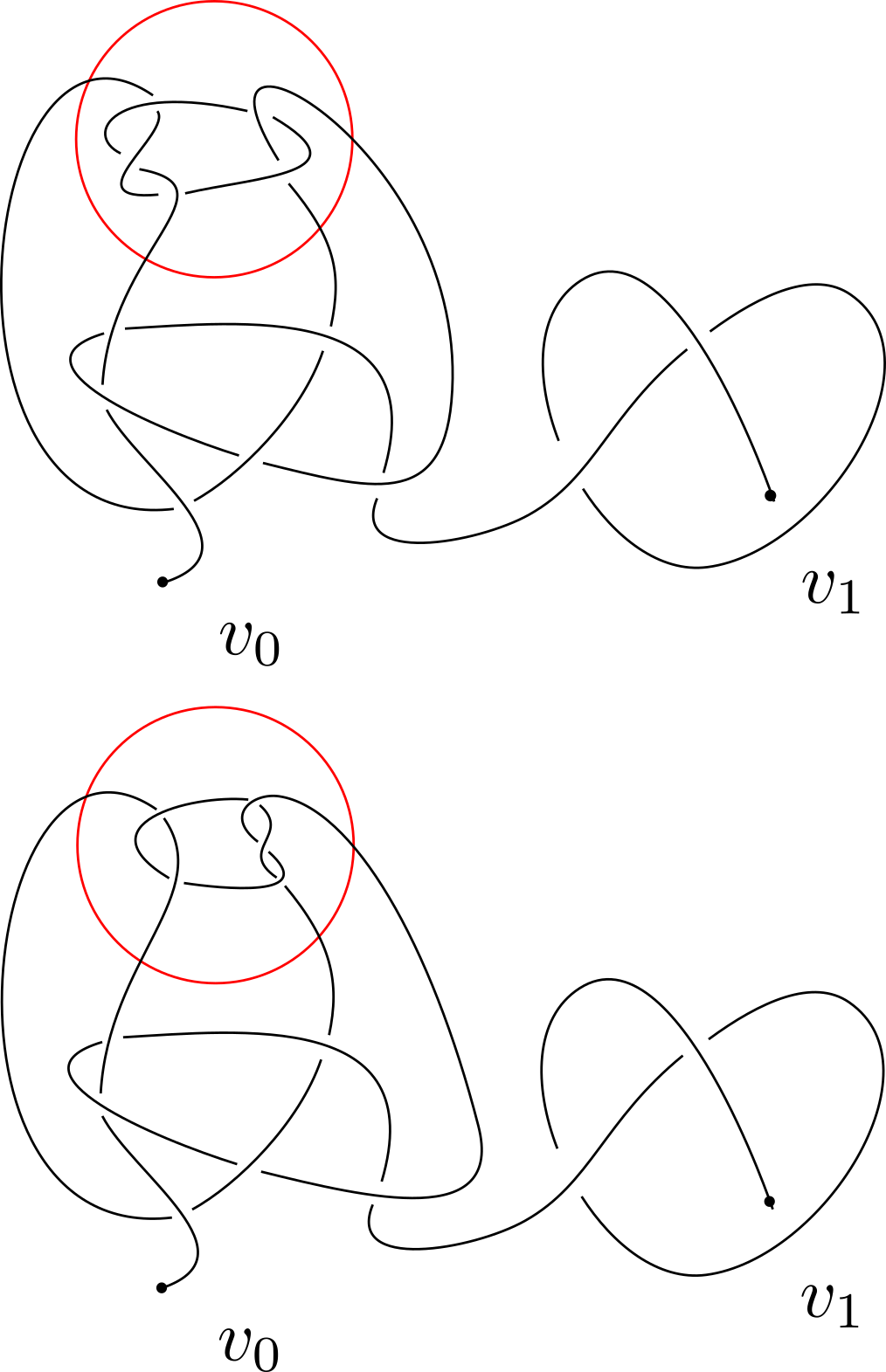}
\caption{The mutant knotoids $k_1 = KT \cdot k $ and $k_2 = C \cdot k $ share the same bracket polynomials.}
\label{fig:mutanti2}
\end{figure}

\section{Double Branched Covers}\label{sec:dbc}
As described in \cite{knotoids}, it is possible to give a $3$-dimensional definition of knotoids, as embedded arcs in $\mathbb{R}^3$, up to a particular isotopy notion. 
\subsection{Knotoids as embedded arcs}\label{subsec:embedded}
Consider a knotoid diagram $k$ in $\mathbb{R}^2$, and identify the plane of the diagram with $\mathbb{R}^2 \times \{0\} \subset \mathbb{R}^3$. We can embed $k$ in $\mathbb{R}^3$ by pushing the overpasses of the diagram into the upper half-space, and the underpasses into the lower one. The endpoints $v_0$ and $v_1$ of $k$ are attached to two lines $t\times \mathbb{R}, h \times \mathbb{R}$ perpendicular to the $xy$ plane.

Two embedded arcs in $\mathbb{R}^3$ with endpoints lying on these two lines are said to be \emph{line isotopic} if there is a smooth ambient isotopy of the pair
($\mathbb{R}^3 ,t\times \mathbb{R} \cup h \times \mathbb{R} $) taking one curve to the other, endpoints to endpoints, and leaving each one of the special lines invariant. Conversely, an embedded curve in $\mathbb{R}^3$ whose projection on the $xy$-plane is generic (plus the additional data of over and under passings) defines a knotoid diagram (see Figure \ref{fig:differentproj}).

\begin{figure}[h]
\includegraphics[width=5cm]{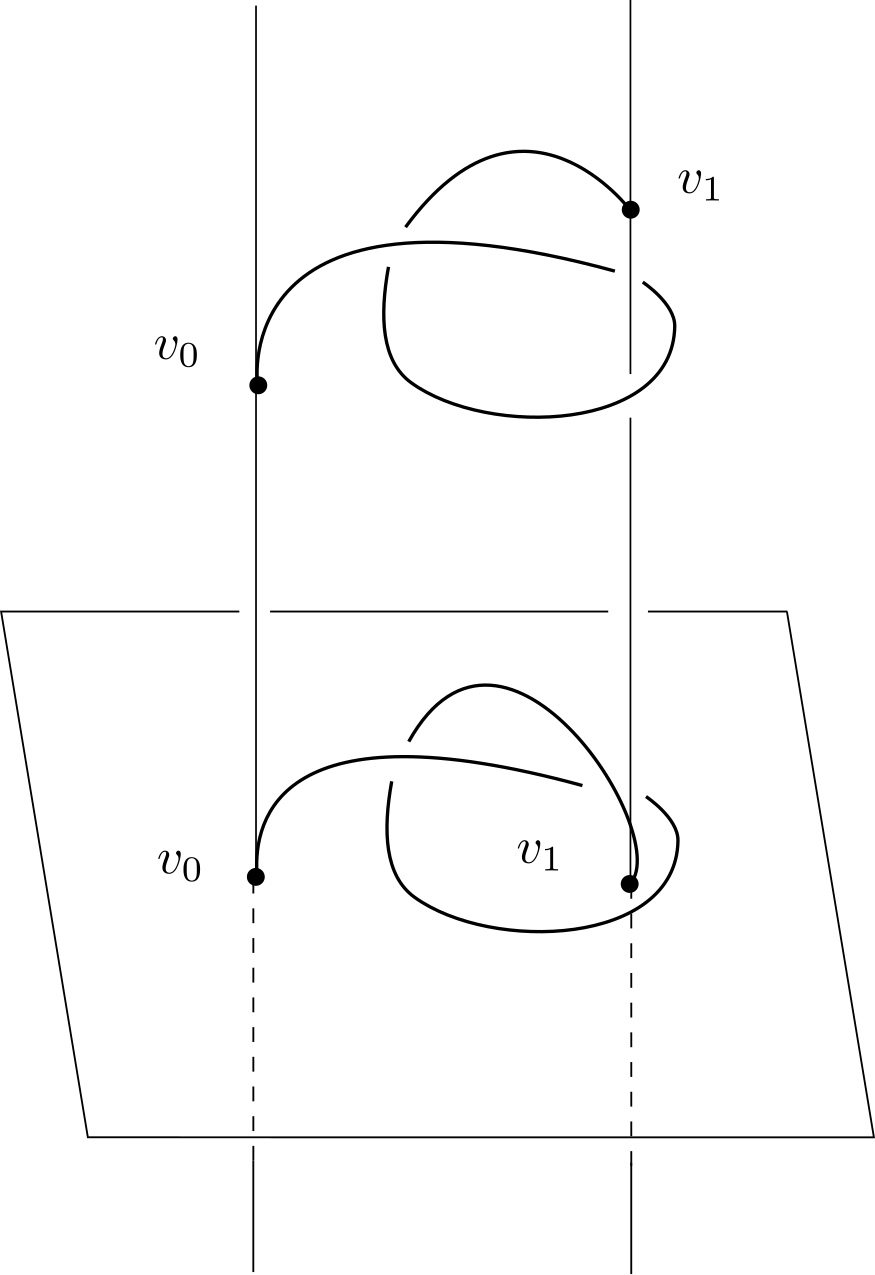}
\caption{On the top, the curve in $\mathbb{R}^3$ obtained from the knotoid diagram on the bottom.}
\label{fig:differentproj}
\end{figure}

There is a $1$-$1$ correspondence (see \cite{knotoids}, Theorem $2.1$ and Corollary $2.2$) between the set of oriented knotoids in $\mathbb{R}^2$ and the set of line-isotopy classes of smooth oriented arcs in $\mathbb{R}^3$, with endpoints attached to two lines perpendicular to the $xy$ plane.

Similarly, given a knotoid $k$ in $\mathbb{K}(S^2)$ we can construct an embedded arc in $S^2 \times I$ with the same procedure. Now the endpoints are attached to two lines perpendicular to the sphere $S^2 \times \{pt\}$. Theorem $2.1$ and Corollary $2.2$ in \cite{knotoids} extend naturally to this setting. 

\begin{rmk}
The notion of line isotopy between open arcs is explored also in \cite{railknotoids}, where \emph{rail knotoids} are introduced. As for line isotopy classes of embedded arcs, rail knotoids corresponds bijectively with planar knotoids. Moreover, the rail knotoid approach is related to the study of genus $2$ handlebodies.  
\end{rmk}

\begin{rmk}\label{rotation3d}
There is an easy way to visualise rotation of knotoids in this setting. Indeed, consider a knotoid $k$ and its rotation $k_{\textbf{rot}}$. If we view the knotoids as embedded arcs in $\mathbb{R}^3$ with endpoints in $t\times \mathbb{R}, h \times \mathbb{R}$, then they differ from each other by applying a rotation through an angle $\pi$ along a horizontal line going through $t\times \mathbb{R}$ and $h \times \mathbb{R}$.
\end{rmk}

\subsection{Knotoids and $\theta$-curves}\label{miserve2}
Consider a knotoid as an embedded curve in $S^2 \times I$, with endpoints attached to the two special lines. We can compactify the manifold by collapsing $S^2 \times \partial I$ to two points, obtaining an embedded curve in $S^3$ with endpoints lying on an unknotted circle, as in Figure \ref{fig:compattifica}.

\begin{figure}[h]
\includegraphics[width=9cm]{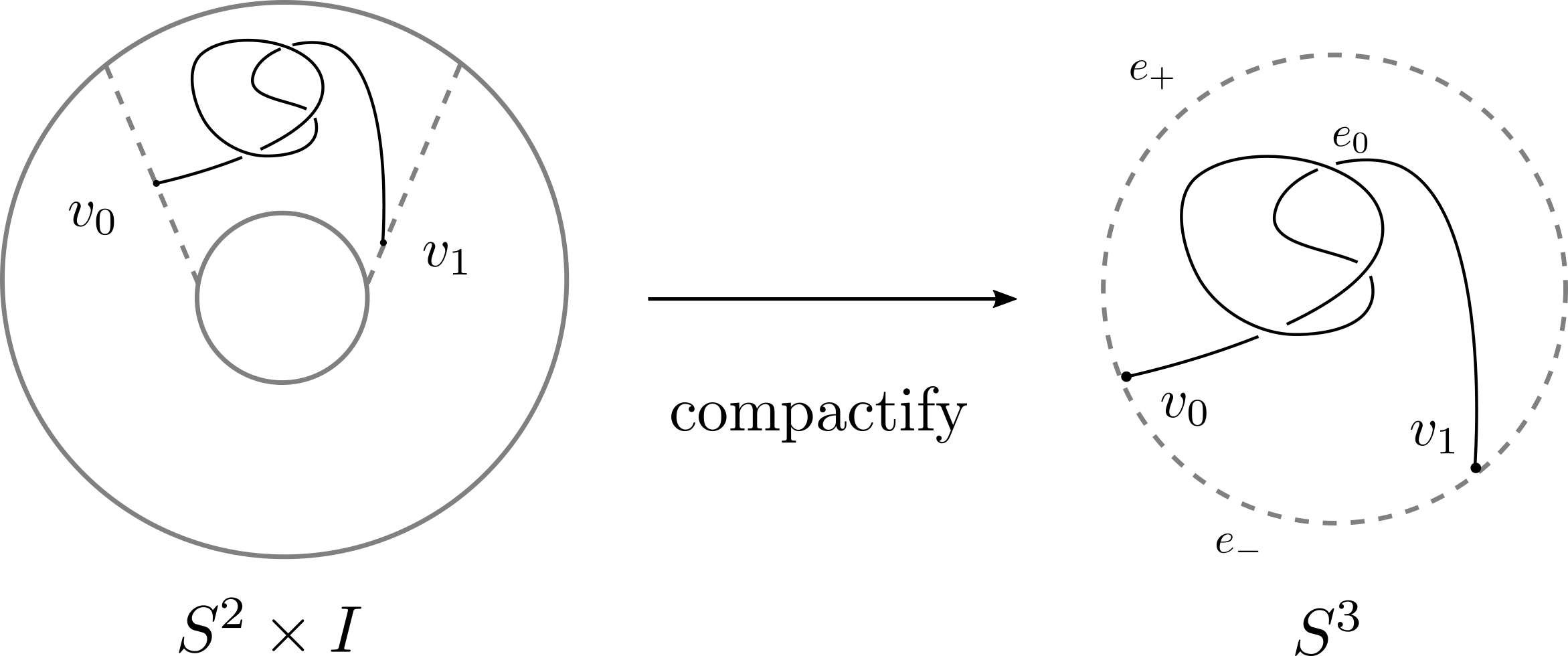}
\caption{On the left, a knotoid seen as an embedded curve in $S^2 \times I$, with endpoints lying on the dotted lines. By collapsing  $S^2 \times \partial I$ to two points, we obtain an embedded arc in $S^3$, with endpoints lying on a dotted circle (the projection of the dotted lines). }
\label{fig:compattifica}
\end{figure}

The union of the embedded curve with this unknotted circle is a $\theta$-curve. 

\begin{defi}
A \emph{labelled} $\theta$\emph{-curve} is a graph embedded in $S^3$ with $2$ vertices, $v_0$ and $v_1$, and $3$ edges, $e_+, e_-$ and $e_0$, each of which joins $v_0$ to $v_1$. The curves $e_0 \cup e_-$, $e_- \cup e_+$ and $e_0 \cup e_+$ are called the \emph{constituent knots} of the $\theta$-curve. We will call two labelled $\theta$-curves \emph{isotopic} if they are related by an ambient isotopy preserving the labels of the vertices and the edges. A $\theta$-curve is called \emph{simple} if its constituent knot $e_- \cup e_+$ is the trivial knot.
 
\end{defi}

Thus, we can associate a simple labelled $\theta$-curve to a knotoid $k \in \mathbb{K}(S^2)$, whose vertices are the endpoints of $k$ and with $e_0 =k$. We label the remaining edges of the $\theta$-curve in the following way. The edge containing the image of $S^2 \times \{ 1\}$ under the collapsing map is labelled $e_+$. The edge containing the image of $S^2 \times \{ -1 \}$ is labelled $e_-$. We will call the unknotted circle $e_- \cup e_+$ the \emph{preferred constituent unknot} of the $\theta$-curve.
It is shown in \cite{turaev} that this construction induces a well defined map $t$ between the set of oriented knotoids $\mathbb{K}(S^2)$ and the set $\Theta^s$ of isotopy classes of simple labelled $\theta$-curves. Moreover, $\Theta^s$ endowed with the vertex-multiplication operation (for a definition of vertex-multiplication see \emph{e.g. \cite{theta2}}) is a semigroup, and the following theorem holds.

\begin{thm}[Theorem $6.2$ in \cite{turaev}]\label{thm:semigroupsiso}
The map $t: \mathbb{K}(S^2) \longrightarrow \Theta^s$ is a semigroup isomorphism.
\end{thm}

The inverse map $t^{-1}$ associates a knotoid in $\mathbb{K}(S^2)$ to a labelled simple $\theta$-curve in the following way. Any element $\theta$ of $\Theta^s$ can be isotoped to lie in $\mathbb{R}^3 \subset S^3 $, with the edge $e_+$ contained in the upper half-space, and $e_-$ in the lower one, in such a way that they both project to a same arc $a$ in $\mathbb{R}^2$ connecting $v_0$ to $v_1$. We say that the $\theta$-curve is in \emph{standard} position. The projection of the edge $e_0$ to $\mathbb{R}^2$ defines the associated knotoid (see \cite{turaev} for more details).

It should be clear that the $\theta$-curves associated to a knotoid and its reverse differ by exchanging the labels of the vertices. Consider now a knotoid $k$ and its rotation $k_{\textbf{rot}}$. Their $\theta$-curves differ from each other simply by swapping the labels on the $e_-$ and $e_+$ edges, and leaving all other labels unchanged. To see this, arrange the $\theta$-curve $t(k)$ in standard position. Suppose that we swap the labels $e_-$ and $e_+$. Then, we can isotope the $\theta$-curve in a way that reinstates $e_+$ as lying above the horizontal plane and $e_-$ as lying below it. After projecting, we get the knotoid $k_{\textbf{rot}}$. 
Thus, swapping the labels of the edges $e_-$ and $e_+$ takes the $\theta$-curve $t(k)$ to $t(k_{\textbf{rot}})$, see Figure \ref{fig:thetainvertite2}.

\begin{figure}[h]
\includegraphics[width=9cm]{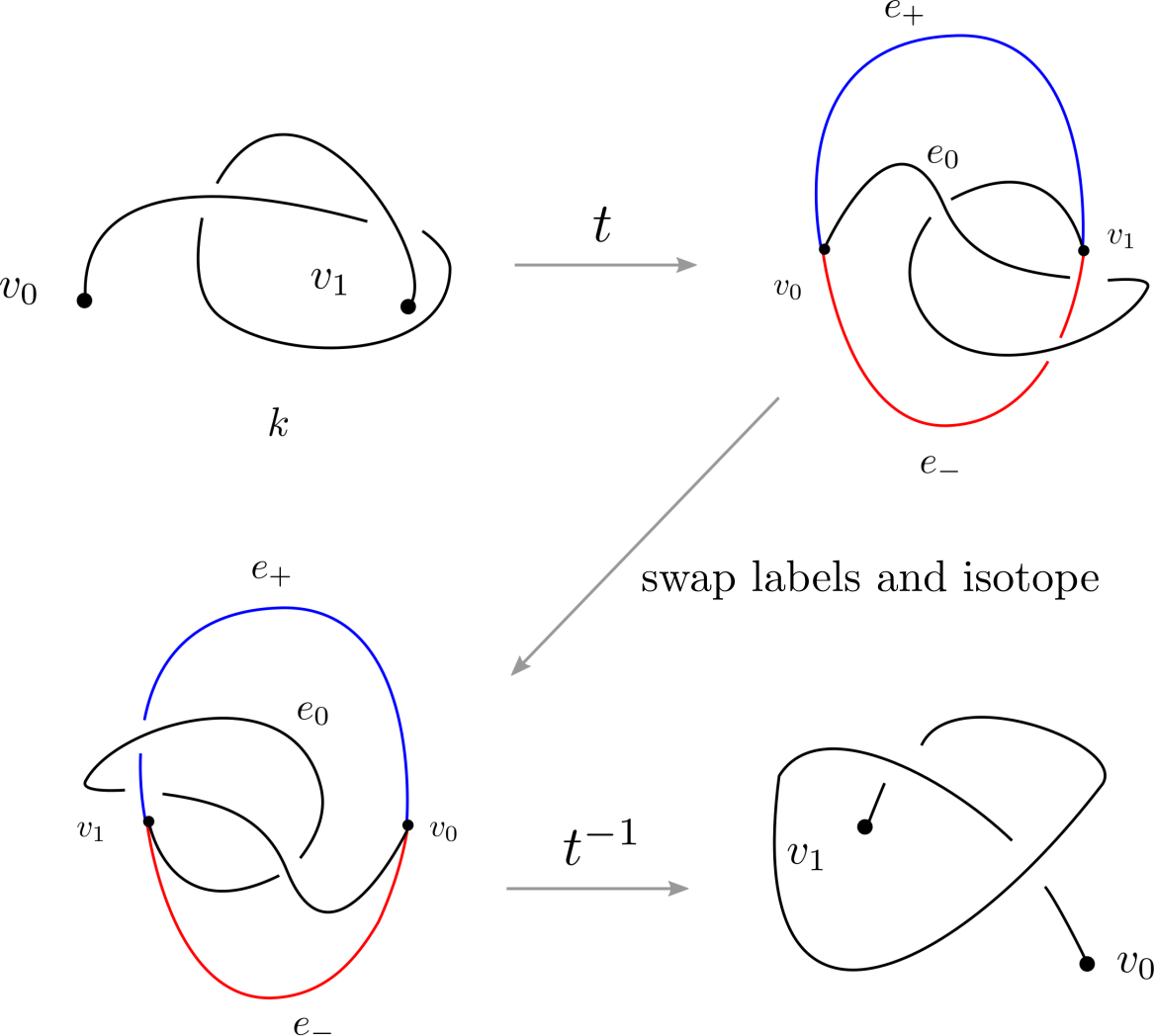}
\caption{A knotoid and its rotation are associated with $\theta$-curves differing from each other by swapping the $e_-$ and $e_+$ labels.}
\label{fig:thetainvertite2}
\end{figure}

Call $ \Theta^s /_\sim $ the set of simple $\theta$-curves up to relabelling the two vertices. The isomorphism $t$ of Theorem \ref{thm:semigroupsiso} gives a bijection $$t_{\sim} : \mathbb{K}(S^2)/_{\sim} \longrightarrow \Theta^s /_\sim $$ between unoriented knotoids and elements of $ \Theta^s /_\sim $. Furthermore, $t$ also induces a bijection $$t_{\approx}:\mathbb{K}(S^2)/_{\approx} \longrightarrow  \Theta^s/_{\approx}$$ from the set of unoriented knotoids up to rotation and the set of $\theta$-curves up to relabelling the edges $e_-$ and $e_+$ and the vertices. This latter bijection will be the key element in proving Theorem \ref{thm:completo}.

\subsection{Double branched covers}\label{miserve}
Consider a planar knotoid $k$, thought of as an embedded arc in the cylinder $D^2 \times I$. The double cover of $D^2 \times I$ branched along the special vertical arcs where the endpoints lie is the solid torus $S^1 \times D^2$. Denote the branched covering map by $$p:S^1 \times D^2 \longrightarrow D^2 \times I $$ 

Figure \ref{rivestimento} shows how to construct this double branched cover by ``cuts''. More precisely, we can cut the cylinder $D^2 \times I$ along two disks (these are the two dashed arcs times $I$, shown in Figure \ref{rivestimento}), take two copies of the obtained object, and then glue them as shown in the picture. For more details on how to construct branched covers see \emph{e.g.}~\cite[Chapter $10.B$]{rolfsen}.
The pre-image $p^{-1}(k)$ of the knotoid in the double branched cover is a knot inside the solid torus $S^1 \times D^2$. The knot type of this branched cover is a knotoid invariant; in particular by composing the branched covering construction with any invariant of knots in the solid torus (see \emph{e.g.}~\cite{sophia}, \cite{maciei} and \cite{hoste}) we obtain a new knotoid invariant. Note that by definition, the lifts of line-isotopic embedded arcs are ambient isotopic knots, since isotopies of $k$ preserving the branching set lift to equivariant isotopies.

\begin{rmk}\label{rmk:diagrammariv}
From the knotoid diagram obtained by projecting $k$, it is possible to construct a diagram in the annulus $S^1 \times I$ for $p^{-1}(k)$ by taking the double cover of the disk $D^2 \times \{pt\}$ branched over the endpoints of the diagram, as shown in Figure \ref{rivestimento}.
 
\end{rmk}

Similarly, given a knotoid $k \in \mathbb{K}(S^2)$ and the associated $\theta$-curve $t(k)$ in $S^3$, the pre-image of $k$ under the double cover of $S^3$ branched along the preferred constituent unknot of $t(k)$ is a knot in $S^3$.  Double branched covers of simple $\theta$-curves have been extensively studied in \cite{theta}, whose main results are discussed and used in Section \ref{sec:dbctheta}.
\begin{rmk}\label{rmk:diagrammarivsf}
Consider a diagram representing $k \in \mathbb{K}(S^2)$: we can obtain a diagram for the lift of $k$ by taking the double cover of $S^2$ branched along the endpoints.
\end{rmk}
\begin{figure}[h] 
\includegraphics[width=8cm]{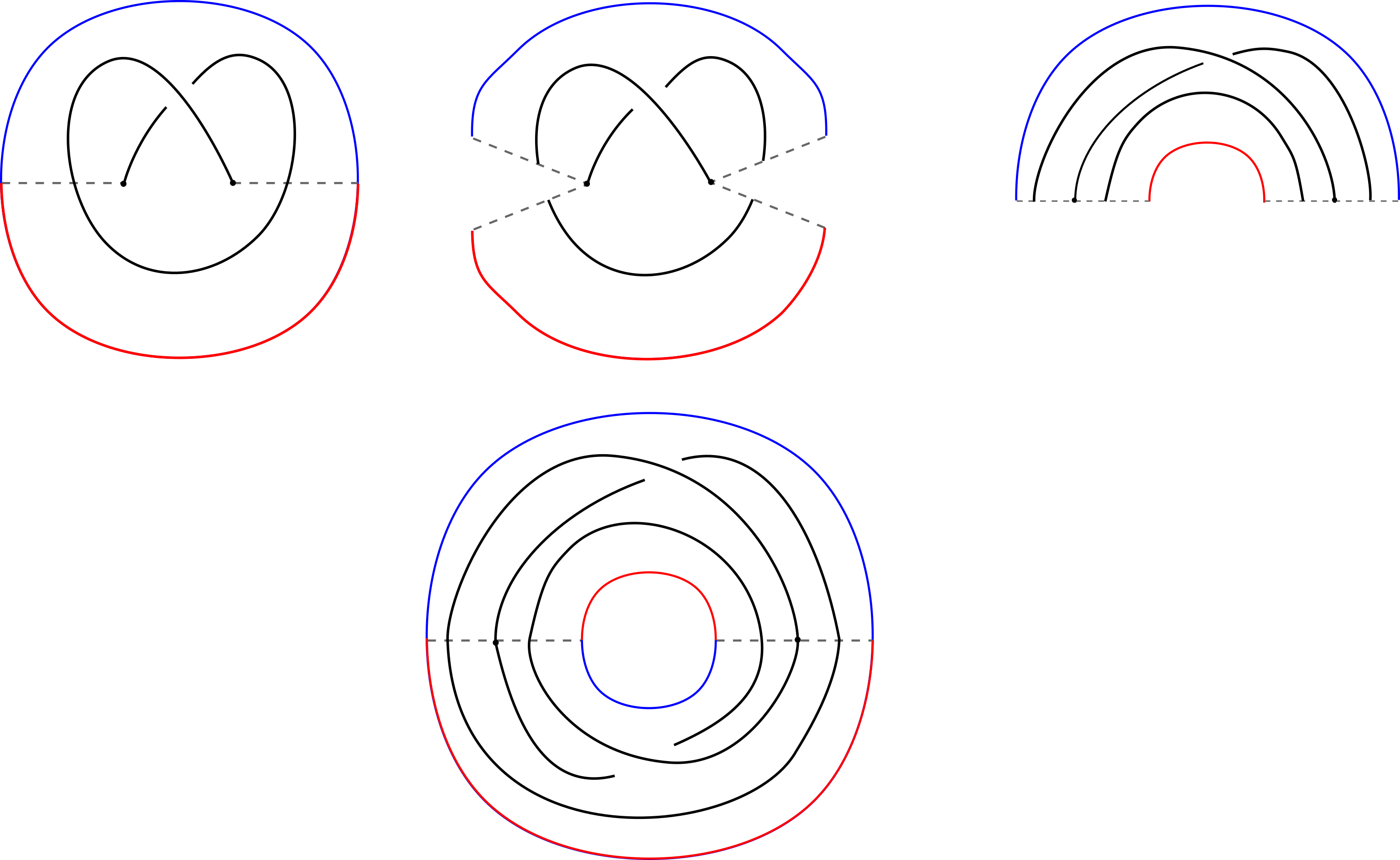}
\caption{The two-fold branched cover of $D^2 \times I$ by $S^1 \times D^2$ can be described by ``cuts'' (see \emph{e.g.} \cite[Chapter $10.B$]{rolfsen} for a reference). The picture shows the product $D^2 \times I$ seen from above. The red and blue circle in $D^2 \times \{1\}$ (the boundary of $D^2 \times \{1\}$) lifts to two parallel longitudes of the solid torus.}
\label{rivestimento}
\end{figure}

Call $\mathcal{K}(S^1 \times D^2)$ and $\mathcal{K}(S^3)$ the sets of knots in the solid torus and in $S^3$ respectively, taken up to the appropriate ambient isotopies. Thus, we have the following maps induced by the double branched covers: $$ \gamma_T : \mathbb{K}(\mathbb{R}^2) \longrightarrow \mathcal{K}(S^1 \times D^2) $$ $$ \gamma_S : \mathbb{K}(S^2) \longrightarrow \mathcal{K}(S^3) $$

\begin{rmk}\label{rmk:samelift}
Recall that two knotoids $k$ and $k_{\textbf{rot}}$ that differ by a rotation lift to $\theta$-curves differing from each other simply by swapping the labels on the $e_-$ and $e_+$ edges (see the discussion in Section \ref{miserve2}). The double branched covers of such $\theta$-curves produce isotopic knots. Thus, $k$ and $k_{\textbf{rot}}$ have the same image under the map $\gamma_S$. The same is true for a knotoid $k$ and its reverse $-k$.
\end{rmk}

Consider a circle in the boundary of $D^2 \times \{pt\}$, as the red and blue one in Figure \ref{rivestimento}. This lifts to two parallel longitudes of the solid torus. We can then define a natural embedding $e$ of the solid torus in $S^3$ by sending any of these longitudes to the preferred longitude of the solid torus in $S^3$ arising as the neighbourhood of the standard unknot. By composing $\gamma_T$ with $e$ we can associate to a knotoid in $\mathbb{K}(\mathbb{R}^2)$ a knot in $S^3$. 

\begin{prop}\label{prop:s2}
Given a knotoid $k$ in $\mathbb{K}(\mathbb{R}^2)$, $e(\gamma_T(k)) = \gamma_S(\iota(k))$. Similarly, given $k \in \mathbb{K}(S^2)$ take any planar representative $k^{pl}$ of $k$. Then, the knot type of $e(\gamma_T(k^{pl}))$ does not depend on the particular choice of $k^{pl}$.

\end{prop}
In other words, the knot type in $S^3$ of the lift a planar knotoid $k$ depends only on its class $\iota(k) \in  \mathbb{K}(S^2)$.
\begin{proof}
Consider the diagram for $k$ arising from the projection onto $D^2 \times \{pt\}$. The $2$-fold cover of the disk branched along the endpoints can be viewed as the restriction of the $2$-fold cover of a $2$-sphere branched along the endpoints, see Figure \ref{rivsfera}.
Thus, isotopies on the sphere below translate into isotopies on the sphere for the lifted diagram.

\end{proof}

\begin{rmk}\label{differenza}
Note that, as shown in Figure \ref{fig:differenza}, non-equivalent knotoids in  $\mathbb{K}(\mathbb{R}^2) $ that are equivalent in  $\mathbb{K}(S^2) $ might lift to different knots in the solid torus. 
 
\end{rmk}

\begin{figure}[h]
\includegraphics[width=7cm]{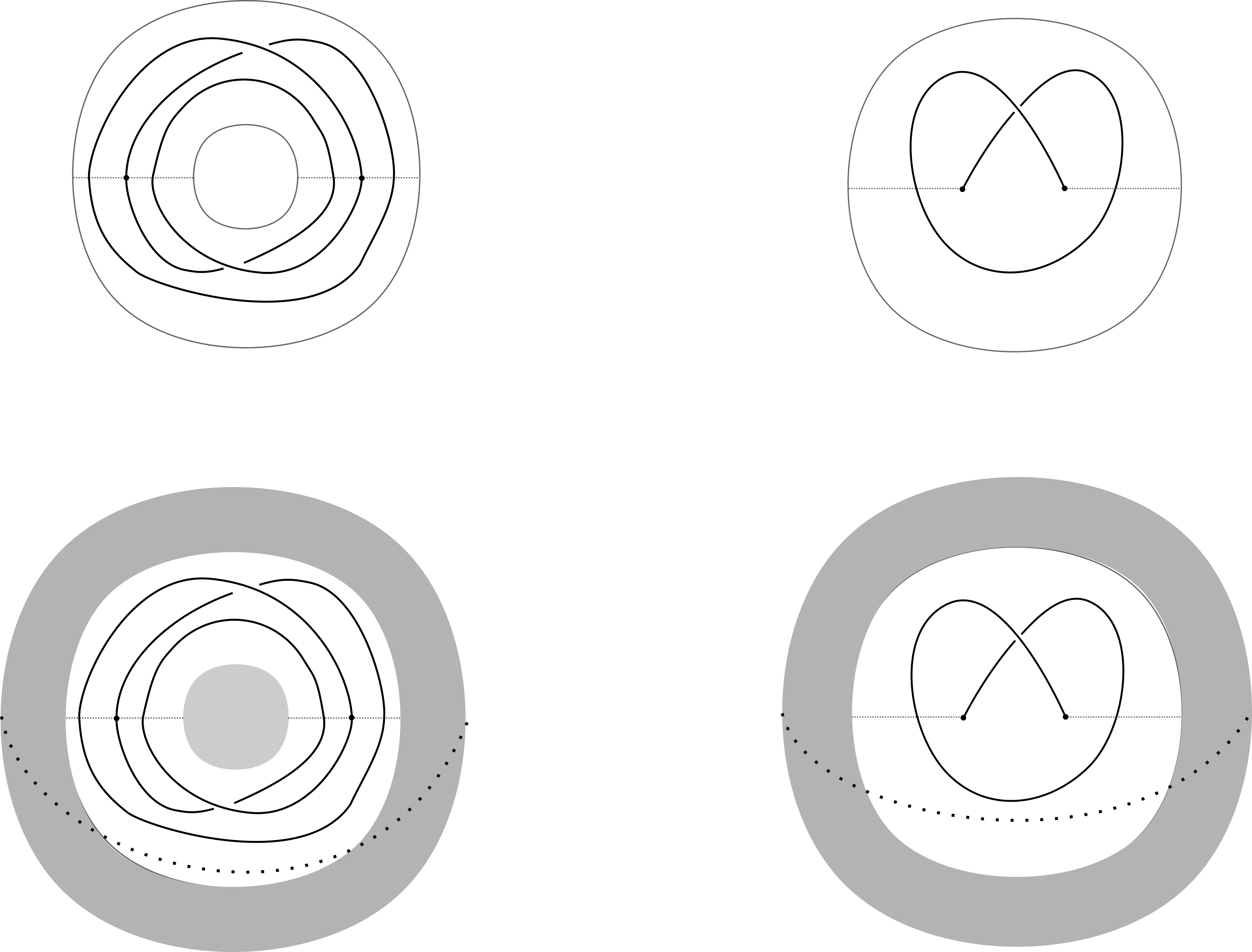}
\caption{On the top, the annulus as double branched cover of the disk. On the bottom, the extension of that cover to a double branched cover of the $2$-sphere over the $2$-sphere. Isotopies on the sphere on the left-side translate into isotopies on the sphere for the lifted diagrams.}
\label{rivsfera}
\end{figure}

\begin{figure}[h]
\includegraphics[width=9cm]{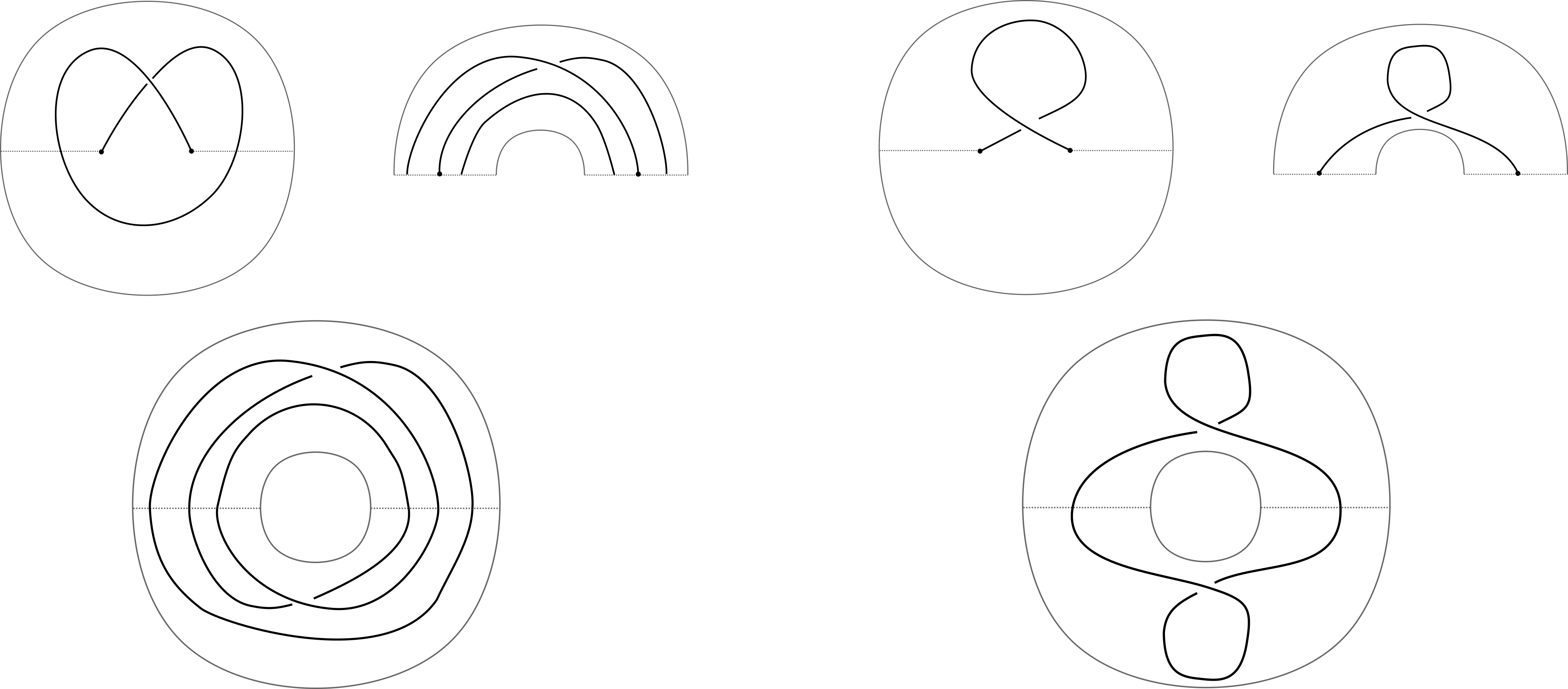}
\caption{Two knotoid diagrams $k_1$ and $k_2$ that represent different classes in $\mathbb{K}(\mathbb{R}^2) $ such that $\iota(k_1) =\iota(k_2)$ lift to different knots in the solid torus. $\gamma_T(k_2)$ is the core of the solid torus, while $\gamma_T(k_1)$ is the $2_3$ knot in Gabrovsek and Mroczkowski's table for knots in the solid torus (see \cite{maciei}), with winding number equal to $3$. However, $e(\gamma_T(k_1)) = e(\gamma_T(k_2)) = \bigcirc$. }
\label{fig:differenza}
\end{figure}

For knot type knotoids the behaviour under the maps $\gamma_S$ and $\gamma_T$ is unsurprisingly trivial.

\begin{prop}\label{liftknottype}
Consider an oriented knot-type knotoid $K$. The lift $\gamma_S(K)$ is the connected sum $K' \# rK'$, where $K'$ is the knot naturally associated to $K$ (with orientation induced by $K$) and $rK'$ is its inverse. 
\end{prop}

\begin{proof}
Thanks to Proposition \ref{prop:s2} we can choose a planar diagram for $k$ in which the endpoints lie in the \emph{external region} of the disk, so that there are no intersections (apart from the endpoints themselves) between the diagram and the arc which define the cuts, and the statement is trivially true.  
\end{proof}

\subsection{Knotoids and strongly invertible knots}\label{sec:st}
Consider a knotoid $k \in \mathbb{K}(S^2)$ and its lift $\gamma_S(k)$ in $S^3$. The fact that $S^3$ is the double cover of itself branched along the preferred constituent unknot $U$ of $t(k)$ defines an orientation preserving involution $\tau$ of $S^3$, whose fixed point set is precisely the unknot $U$. The involution $\tau$ reverses the orientation of $\gamma_S(k)$, and the fixed point set intersects $\gamma_S(k)$ in exactly two points (the lifts of the endpoints of $k$). A knot with this property is called a \emph{strongly invertible knot} (a precise definition will be given in Section \ref{sec:strong}).

Since not every knot in $S^3$ is strongly invertible, this in particular implies that the maps $\gamma_S$ and $\gamma_T$ are not surjective.

\begin{rmk}
We could have inferred the non-surjectivity of the map $\gamma_T$ from the following observation. 

The winding number $[\gamma_T(k)] \in H_1(S^1\times D^2; \mathbb{Z})$ of the lift $\gamma_T(k)$ of a knotoid $k$ is always odd. This is true since by construction the lifted knot intersects the meridian disk containing the lifted branching points an odd number of times.
 
\end{rmk}

In Section \ref{sec:strong} we will use classical results on symmetry groups of knots to better understand the map $\gamma_S$ and to prove the $1$-$1$ correspondence of Theorem \ref{thm:completo}.

\subsection{Behaviour under forbidden moves}

A band surgery is an operation which deforms a link into another link.
\begin{defi}
Let $L$ be a link and $b: I \times I \longrightarrow S^3$ an embedding such that $L \cap b(I \times I) = b(I \times \partial I)$. The link $L_1 = (L \setminus b(I \times \partial I) ) \cup  b(\partial I \times I)$ is said to be obtained from $L$ by a \emph{band surgery} along the band $B=  b(I \times I)$, see Figure \ref{fig:bandsurg}
 
\end{defi}

Performing a band surgery on a link $L$ may change its number of components; band surgeries which leave unchanged the number of components are called $H(2)$-moves (see \emph{e.g.} \cite{band1}, \cite{band2}). 

\begin{figure}[h]
\includegraphics[width=5cm]{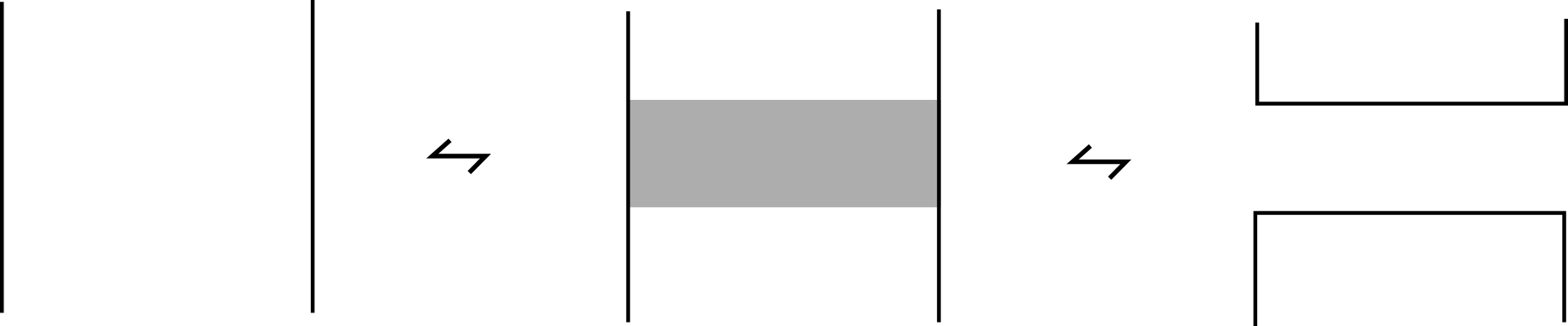}
\caption{Band surgery.}
\label{fig:bandsurg}
\end{figure}

An $H(2)$-move is an \emph{unknotting} operation, that is, any knot may be transformed into the trivial knot by a finite sequence of $H(2)$-moves.
Consider two knotoids that differ by a forbidden move, as on the top of Figure \ref{fig:bsesempio}: it is easy to see that their lifts are related by a single $H(2)$-move (see the bottom part of Figure \ref{fig:bsesempio}).

\begin{figure}[h]
\includegraphics[width=5cm]{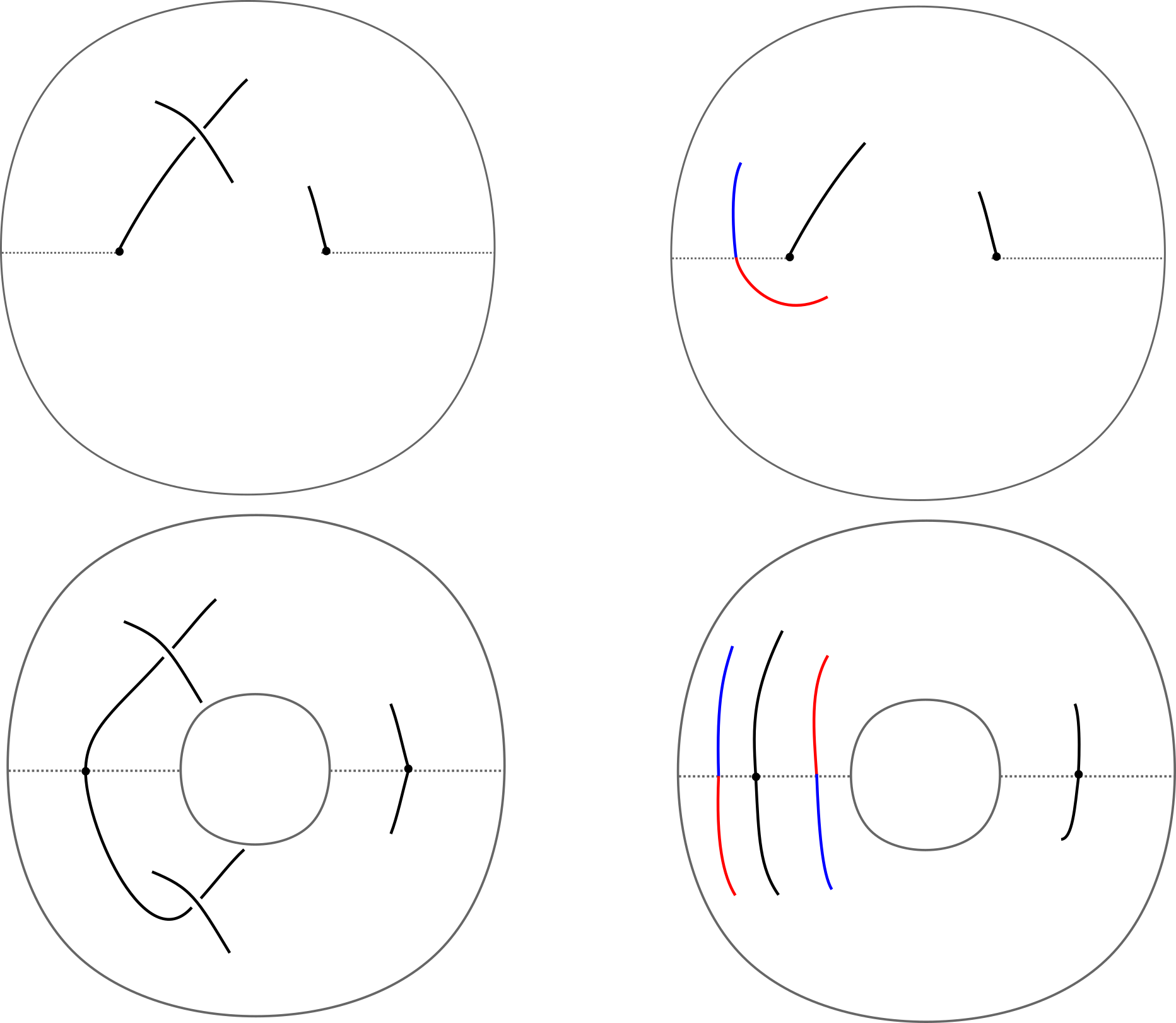}
\caption{Two knotoids that differ by a forbidden move have lifts related by a single band surgery.}
\label{fig:bsesempio}
\end{figure}

\section{Multiplication and trivial knotoid detection}\label{sec:dbctheta}

In this section we will first discuss the behaviour of $\gamma_S$ under multiplication of knotoids. We will then prove two different results on the detection of the trivial knotoid.  

\subsection{Behaviour under multiplication}\label{sec:dbctheta2}

Double branched covers of simple $\theta$-curves have been extensively studied in \cite{theta}. 

\begin{defi}\label{def:thetaprime}
A $\theta$-curve is said to be \emph{prime} if:
\begin{itemize}
 \item it is non-trivial;
 \item it is not the connected sum of a non trivial knot and a $\theta$-curve (see the top part of Figure \ref{fig:thetamulti});
 \item it is not the result of a vertex-multiplication (for a definition of the vertex-multiplication operation see \emph{e.g.} \cite{turaev}, Section $5$) of two non-trivial $\theta$-curves (see the bottom part of Figure \ref{fig:thetamulti}). 
\end{itemize}

\end{defi}

\begin{figure}[h]
\includegraphics[width=9cm]{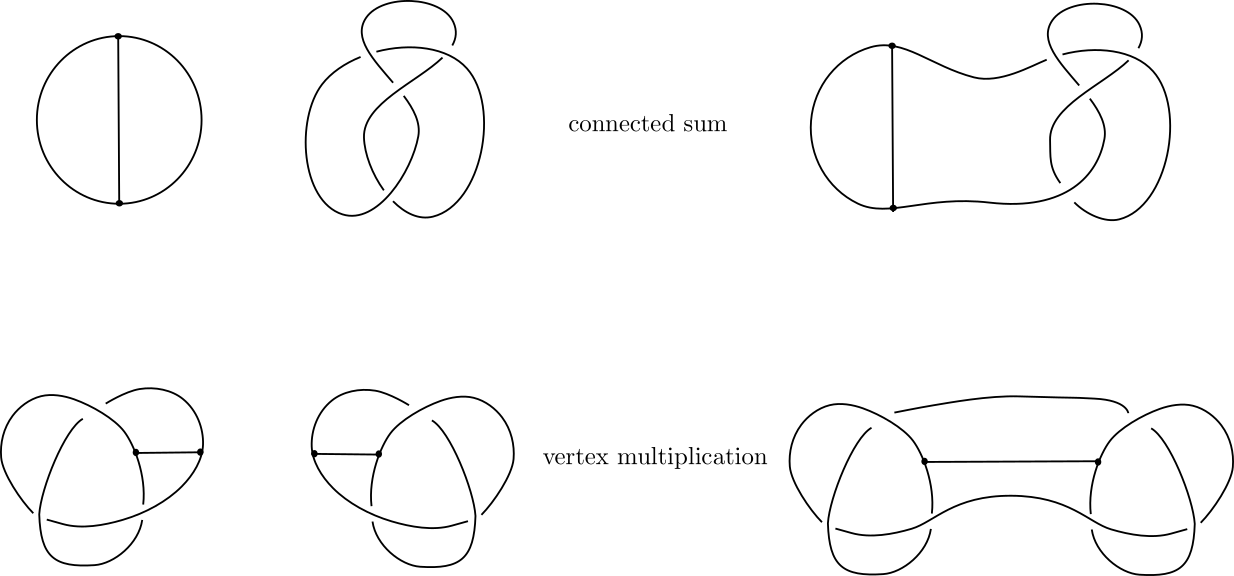}
\caption{On the top, the result of a connected sum between a knot and a $\theta$-curve. On the bottom, the result of a vertex-multiplication of two $\theta$-curves.}
\label{fig:thetamulti}
\end{figure}

According to Definition \ref{def:thetaprime}, if $K$ is a knot-type knotoid, then $t(K)$ is the vertex multiplication of a non trivial knot and a $\theta$-curve, thus, it is not prime. The following result is attributed to Thurston by Moriuchi (\cite{moriuchi}, Theorem $4.1$), and it has been proven in \cite{theta}.

\begin{thm}[Main Theorem in \cite{theta}]\label{thm:main}

Consider a simple $\theta$-curve $a$, with unknotted constituent knot $a_1$, and let $K$ be the closure of the pre-image of $a\setminus a_1$ under the double cover of $S^3$ branched along $a_1$. Then $K$ is prime if and only if $a$ is prime. 

\end{thm}

Theorem \ref{thm:main} together with Theorem \ref{thm:semigroupsiso} directly imply the following result on knotoids.

\begin{thm}\label{thm:multandcon}
The lift $\gamma_S(k)$ of a proper knotoid $k$ is prime if and only if $k$ is prime. In particular $\gamma_S(k_1 \cdot k_2) = \gamma_S(k_1) \# \gamma_S(k_2)$.

\end{thm}
Note that even if the products $k_1 \cdot k_2$ and $k_2 \cdot k_1$ are in general distinct both as oriented and unoriented knotoids (see the relations described in Section \ref{miserve}), their lift are equivalent as knots in $S^3$. This seems to imply that $\gamma_S$ cannnot tell apart $k_1 \cdot k_2$ and $k_2 \cdot k_1$. Indeed, to distinguish them it is necessary to use the information on the involution defined by the double branched cover construction, as we will see in Section \ref{sec:connectedsums}.

Consider the mutant knotoids $k_1$ and $k_2$ of Figure \ref{fig:mutanti2}; Proposition \ref{liftknottype} and Theorem \ref{thm:multandcon} imply that: $$\gamma_S(k_1) = KT \hspace{2mm} \# \hspace{2mm} KT \hspace{2mm} \# \hspace{2mm} \gamma_S(k)$$ $$\gamma_S(k_2) = C  \hspace{2mm} \# \hspace{2mm} C \hspace{2mm} \# \hspace{2mm} \gamma_S(k)$$ Since $\gamma_S(k)$ is isotopic to the trefoil knot $3_1$ (see \emph{e.g.} Figure \ref{fig:esempiogc}), and since the genus of a knot is additive under connected sum, it follows that: $$g(\gamma_S(k_1)) = 2+2+1 =5 \hspace{5mm} g(\gamma_S(k_2)) = 3+3+1 =7$$ Thus, $\gamma_S(k_1)$ and $\gamma_S(k_2)$ are different knots. Moreover, by letting $k$ vary in the set of proper knotoids, we obtain an infinite family of pairs of knotoids sharing the same polynomial invariants whose images under $\gamma_S$ are different. 

\subsection{Trivial knotoid detection}\label{sec:trivialdetection}

The double branched cover of knotoids provides a way to detect the trivial knotoid, thanks to the following result.

\begin{thm}[Lemma $2.3$ in \cite{theta}]\label{thm:banalegen}
A knotoid $k \in \mathbb{K}(S^2)$  lifts to the trivial knot in $S^3$ if and only if $k$ is the trivial knotoid $k_0$ in $\mathbb{K}(S^2)$.
 
\end{thm}
Theorem \ref{thm:banalegen} is proven for $\theta$-curves. In the setting of knotoids, a slightly more powerful version of this result holds, allowing for the detection of the trivial planar knotoid $k^{\textrm{pl}}_0 \in \mathbb{K}(\mathbb{R}^2)$ as well.

\begin{thm}\label{thm:banale}
A knotoid $k \in \mathbb{K}(\mathbb{R}^2)$  lifts to a knot isotopic to the core of the solid torus if and only if $k=k^{\textrm{pl}}_0$ in $\mathbb{K}(\mathbb{R}^2)$.
 
\end{thm}

\begin{proof}
If $k$ is the trivial knotoid, then its lift is a knot isotopic to the core of the solid torus (see \emph{e.g.} the right side of Figure \ref{fig:differenza}). Conversely, suppose that $\gamma_T(k)$ is isotopic to the core $C$ of the solid torus $S^1 \times D^2$. Then, its complement in the solid torus is homeomorphic to the product $T^2 \times I$. Since $T^2 \times I$ arises as a double branched cover, there is an involution $\tau$ of $T^2 \times I$ with $4$ disjoint arcs as fixed set (see Figure \ref{fig:involutionexterior}).

\begin{figure}[h]
\includegraphics[width=6cm]{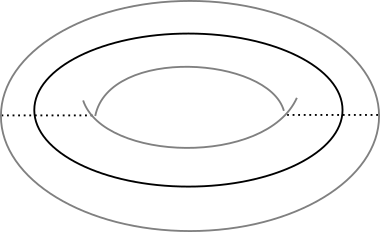}
\caption{$T^2 \times I$ admits an involution with fixed set the union of $4$ arcs. These arcs are the intersection between the lines defining the cover and the complement of a tubular neighbourhood of the lifted knot (the core of the solid torus). }
\label{fig:involutionexterior}
\end{figure}

Thanks to the following result we know that the involution defined by the double branched cover respects the product structure on $T^2 \times I$.

\begin{thm}[Theorem $A$ of \cite{inv2}]
Let $h$ be a PL involution of $F \times I$, where $F$ is a compact surface, such that $h(F \times \partial I)= F \times \partial I$. Then, there exist an involution $g$ of $F$ such that $h$ is equivalent (up to conjugation with homeomorphisms) to the involution of $F \times I$ defined by $(x,t) \mapsto (g(x), \lambda(t))$ for $(x,t) \in F\times I$, and where $\lambda: I \longrightarrow I$ is either the identity or $t \mapsto 1-t$.

\end{thm}

The intersection between the fixed set $Fix(\tau)$ and every torus $T^2 \times \{pt\}$ consists of $4$ isolated points, as highlighted in Figure \ref{fig:involutionexterior}. Involutions of closed surfaces are completely classified; the following result is well known, and it probably should be attributed to \cite{inv3}, but we refer to \cite{inv1} for a more modern and complete survey.

\begin{thm}[Theorem $1.11$ of \cite{inv1}]
There is only one involution $\bar{\tau}$, up to conjugation with homeomorphisms, for the torus $S^1 \times S^1$ with $4$ isolated fixed points. This involution is shown in Figure \ref{fig:onlyinvo}; it is orientation preserving and it is induced by a rotation of $\pi$ about the dotted line indicated in the picture.  
 
\end{thm}

\begin{figure}[h]
\includegraphics[width=6cm]{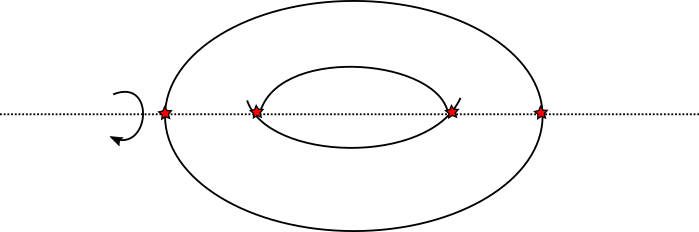}
\caption{The involution of the torus with $4$ fixed points, indicated with the red stars.}
\label{fig:onlyinvo}
\end{figure}

With an abuse of notation, call $\bar{\tau}$ the involution of $T^2 \times I$ obtained as the product $\bar{\tau} \times Id_I$. Since conjugated involutions produce homeomorphic quotient spaces, thanks to the previous two results we can say that the complement of the trivial knot in the solid torus projects to a homeomorphic copy of the complement of the trivial knotoid in the three-ball. In other words, our quotient space $T^2 \times I/ \tau$ is homeomorphic to $T^2 \times I/ \bar{\tau}$, the last one being precisely the  complement of the trivial knotoid, as in Figure \ref{fig:quotientspace}.

\begin{figure}[h]
\includegraphics[width=4cm]{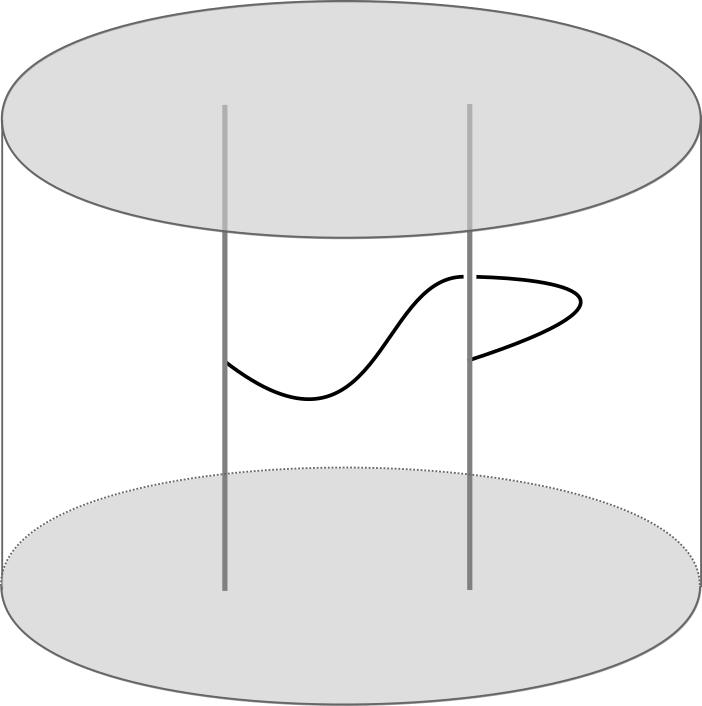}
\caption{The quotient space under the involution is homeomorphic to the complement of the trivial knotoid in the cylinder.}
\label{fig:quotientspace}
\end{figure}

We will be done once we prove that the line isotopy class of the curve in Figure \ref{fig:quotientspace} is not affected by the action of homeomorphisms; this is a consequence of the following proposition. Let $Y$ be the cylinder $D^2 \times I$, and call $MCG(Y; p,q)$ the group of isotopy-classes of automorphisms of $Y$ that leave $p \times I$ and $q \times I$ invariant, where $p,q$ are points in the interior of $D^2 \times \{pt\}$.

\begin{prop}\label{prop:mcg}
$MCG(Y; p,q)$ is isomorphic to $\mathbb{Z}$, and it is generated by a Dehn-twist along the blue rectangle in Figure \ref{fig:quotientspacedehn}.

\end{prop}

\begin{figure}[h]
\includegraphics[width=4cm]{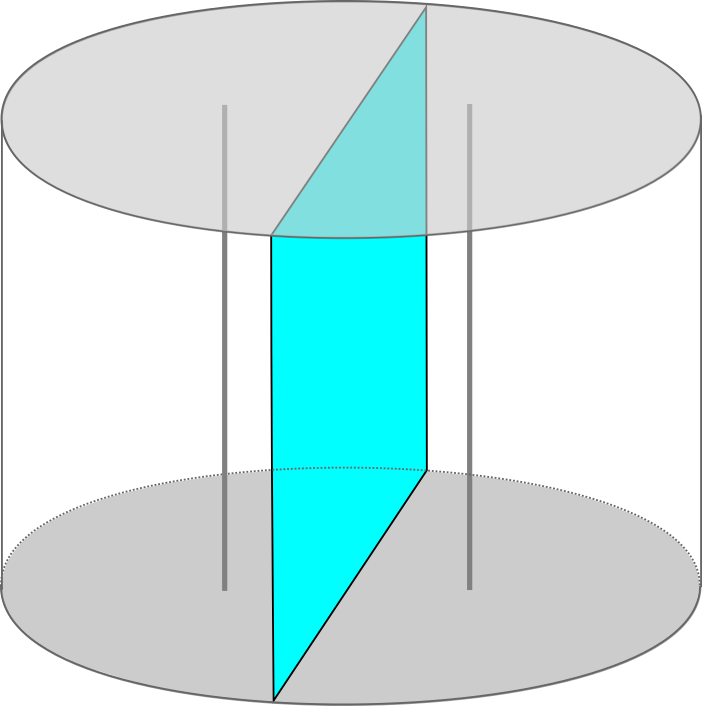}
\caption{$MCG(Y; p,q)$ is generated by a Dehn twist along the boundary of the blue rectangle.}
\label{fig:quotientspacedehn}
\end{figure}

The proof of Proposition \ref{prop:mcg} requires a couple of preliminary results. First, note that removing the two lines yields  a $3$-dimensional genus $2$ handlebody $ H $. The homeomorphisms of a handlebody are determined by their behaviour on the boundary; more precisely, the mapping class group of a handlebody can be identified with the subgroup of the mapping class group of its boundary, consisting of homeomorphisms that can be extended to the handlebody due to the following lemma.

\begin{lemma}

Let $H$ be a genus $2$ handlebody. Any homeomorphism $\phi: H \longrightarrow H$ such that $\restr{\phi}{\partial H}$ is isotopic to $Id_{\partial H}$ is isotopic to $Id_{H}$.
 
\end{lemma}
The previous lemma is well known, and a proof may be found \emph{e.g.} in Chapter $3$ of \cite{fomenko}. 
\begin{rmk}\label{rmk:estendere}
Recall that a self homeomorphism of the boundary of a handlebody can be extended to the handlebody if and only if the image of the boundary of every meridian disc is contractible in the handlebody. In particular, Dehn twists along the blue curves in Figure \ref{fig:cutsurface} do not extend to the handlebody.
 
\end{rmk}

Now, cutting the boundary of the handlebody $H$ along the blue curves, as in Figure \ref{fig:cutsurface}, produces a sphere with $4$ holes $S$.

\begin{figure}[h]
\includegraphics[width=9cm]{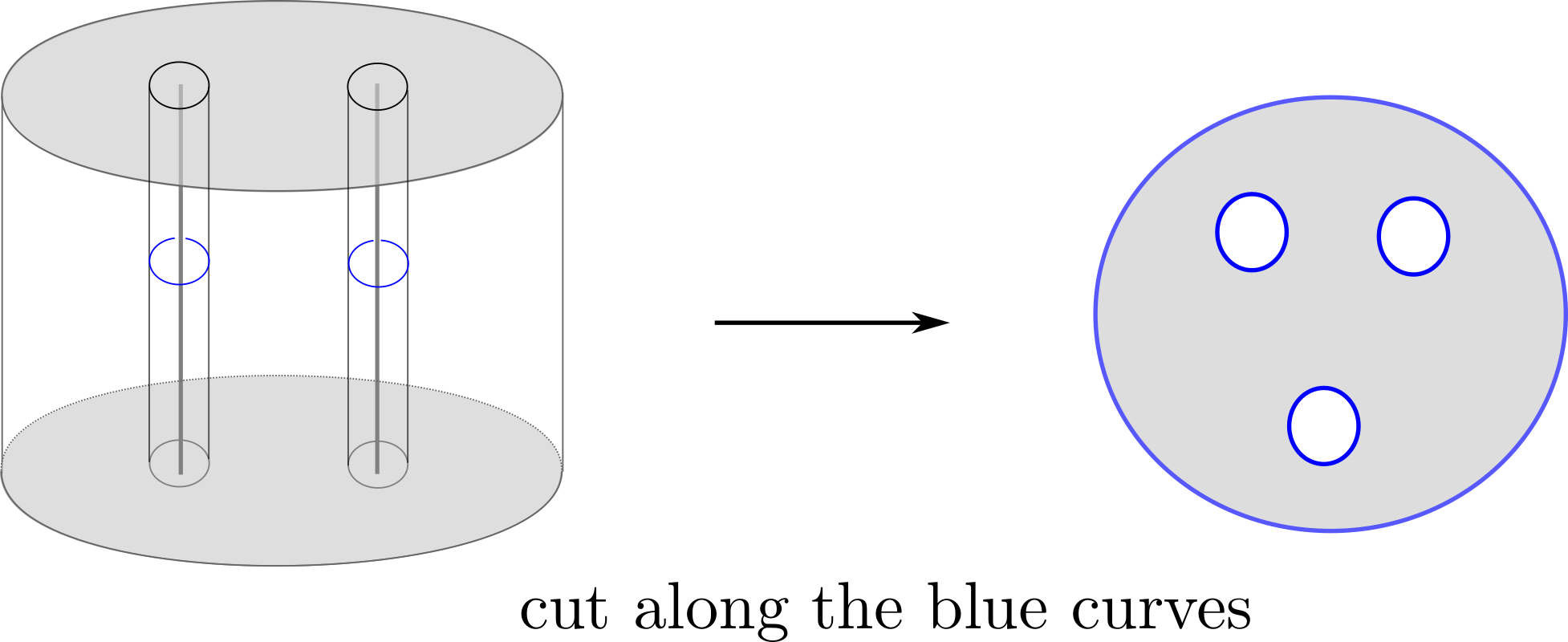}
\caption{Cutting the boundary of the handlebody along the blue curves gives back the sphere with $4$ holes $S$.}
\label{fig:cutsurface}
\end{figure}

A proof for the following lemma can be deduced from \emph{e.g.} the proof of Proposition $2.7$, Chapter $2$, \cite{farb}. Given a surface $S$ with boundary, denote by $MCG(S, \partial S)$ the group of isotopy classes of orientation-preserving homeomorphisms of $S$ that leave each boundary component invariant.

\begin{lemma}\label{lemma:pshere}
Let $S$ be the sphere with $4$ holes. Then, $MCG(S, \partial S)$ is isomorphic to a subgroup of $MCG(T^2) /-Id  \cong PSL(2, \mathbb{Z})$.
 
\end{lemma}

\begin{figure}[h]
\includegraphics[width=3.5cm]{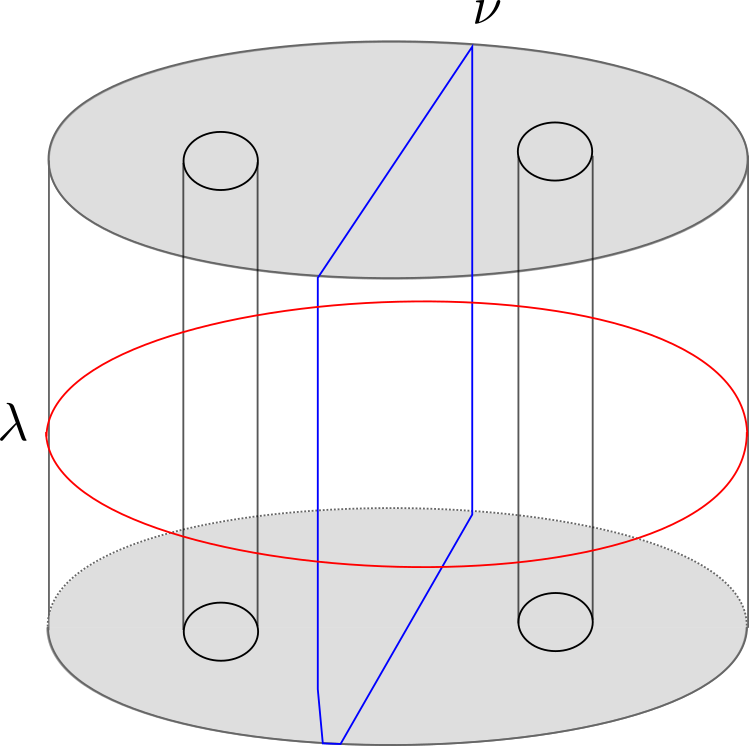}
\caption{$\phi: S \longrightarrow S$ in $MCG(S, \partial S)$ is completely determined by the images of the curves $\nu$ and $\lambda$.}
\label{fig:mcgfinale}
\end{figure}

Thus, we defined a homomorphism $MCG(Y; p,q) \longrightarrow MCG(S, \partial S)$, and an injective homomorphism $ MCG(S, \partial S) \longrightarrow PSL(2, \mathbb{Z})$. An element in the kernel of the composition of these two homomorphisms is then an automorphism of $\partial H$ that leaves the two blue curves in the left side of Figure \ref{fig:cutsurface} invariant and that is isotopic to the identity on $S$. Such an element is a product of Dehn twists about the two blue curves of Figure \ref{fig:cutsurface}, but thanks to Remark \ref{rmk:estendere}, the only element in $MCG(Y;p,q)$ of that form is the trivial element. Moreover, Lemma \ref{lemma:pshere} is proven by exhibiting a bijection between homotopy classes of essential closed curves in $T^2$ and in $S$, and this in particular implies that any homeomorphism $\phi: S \longrightarrow S$ leaving each component of $\partial S$ invariant is completely determined by the images of the curves $\nu$ and $\lambda$ in Figure \ref{fig:mcgfinale}. Now, $\phi(\nu) = \nu$, since $\nu$ is the only essential closed curve in $S$ which is trivial in $H_1(H)$; on the other hand, Remark \ref{rmk:estendere} implies that $\phi(\lambda)$ is the curve that results from $\lambda$ by applying a Dehn twist along $\nu$.
Putting all together, we obtain a proof for Proposition \ref{prop:mcg} and Theorem \ref{thm:banale}, as wanted.
 
\end{proof}

\section{Knotoids and strongly invertible knots}\label{sec:strong}

\subsection{Proof of the main theorem}\label{sec:proofcorrespondence}
This section is devoted to the proof of the main result, Theorem \ref{thm:completo}. We should point out that the correspondence between knotoids and strongly invertible knots is partially inspired by the construction in \cite{liam}, Section $2.2$. We begin by giving a precise definition of what a strongly invertible knot is. Recall that $Sym(S^3,K)$ denotes the symmetry group of a knot $K$, that is, the group of diffeomorphisms of the pair $(S^3, K)$ modulo isotopies, and $Sym^+(S^3,K)$ is the subgroup of $Sym(S^3,K)$ of diffeomorphisms preserving the orientation of $S^3$.

\begin{defi}\label{defi:strongknots}
A \emph{strongly invertible knot} is a pair $(K, \tau)$, where $\tau \in Sym(S^3,K)$ is called a \emph{strong inversion}, and it is an orientation preserving involution of $S^3$ that reverses the orientation of $K$, taken up to conjugacy in $Sym^+(S^3,K)$.
Thus, two strongly invertible knots $(K_1, \tau_1)$ and $(K_2, \tau_2)$ are equivalent if there is an orientation preserving homeomorphism $f:S^3 \longrightarrow S^3$ satisfying $f(K_1)=K_2$ and $f \tau_1 f^{-1}=\tau_2$.
 
\end{defi}

Call $\mathcal{K}SI(S^3)$ the set of strongly invertible knots $(K,\tau)$ in $S^3$, up to equivalence, and $\mathcal{K}_{S.I.}(S^3)$ the subset of $\mathcal{K}(S^3)$ consisting of knots that admit a strong inversion. There is then a natural forgetful map $\mathcal{K}SI(S^3) \longrightarrow \mathcal{K}_{S.I.} (S^3)$. 
As we saw in Section \ref{sec:st}, the lift of a knotoid through the double branched cover of $S^3$ is a strongly invertible knot, thus, $\gamma_S(\mathbb{K}(S^2)) \subset \mathcal{K}_{S.I.}(S^3)$. More precisely, the branching set $e_- \cup e_+$ determines an involution $\tau$. Thus, we can promote $\gamma_S$ to a map $\gamma_S: \mathbb{K}(S^2) \longrightarrow \mathcal{K}SI(S^3)$. Further, a knotoid $k$, its reverse $-k$, its rotation $k_{\textbf{rot}}$ and its reverse rotation $-k_{\textbf{rot}}$ map to the same element in $ \mathcal{K}SI(S^3)$ (see Remark \ref{rmk:samelift}, and note that their associated $\theta$-curves share the same preferred constituent unknot $e_- \cup e_+$). Thus, $\gamma_S$ descends to a well defined map on the quotient $$ \gamma_S: \mathbb{K}(S^2)/_\approx \longrightarrow \mathcal{K}SI(S^3) $$ On the other hand, given a strongly invertible knot there are four oriented knotoids associated to it, given by the construction explained below. Consider a strongly invertible knot $(K, \tau) \in \mathcal{K}SI(S^3)$; the fixed point set of $\tau$ is an unknotted circle (thanks to the positive resolution of the Smith conjecture, \cite{wald}). Moreover, $\tau$ defines the projection $$ p: S^3 \longrightarrow S^3 / \tau \cong S^3 $$ that can be interpreted as the double cover of $S^3$ branched along $Fix(\tau)$.

\begin{figure}[h]
\includegraphics[width=5cm]{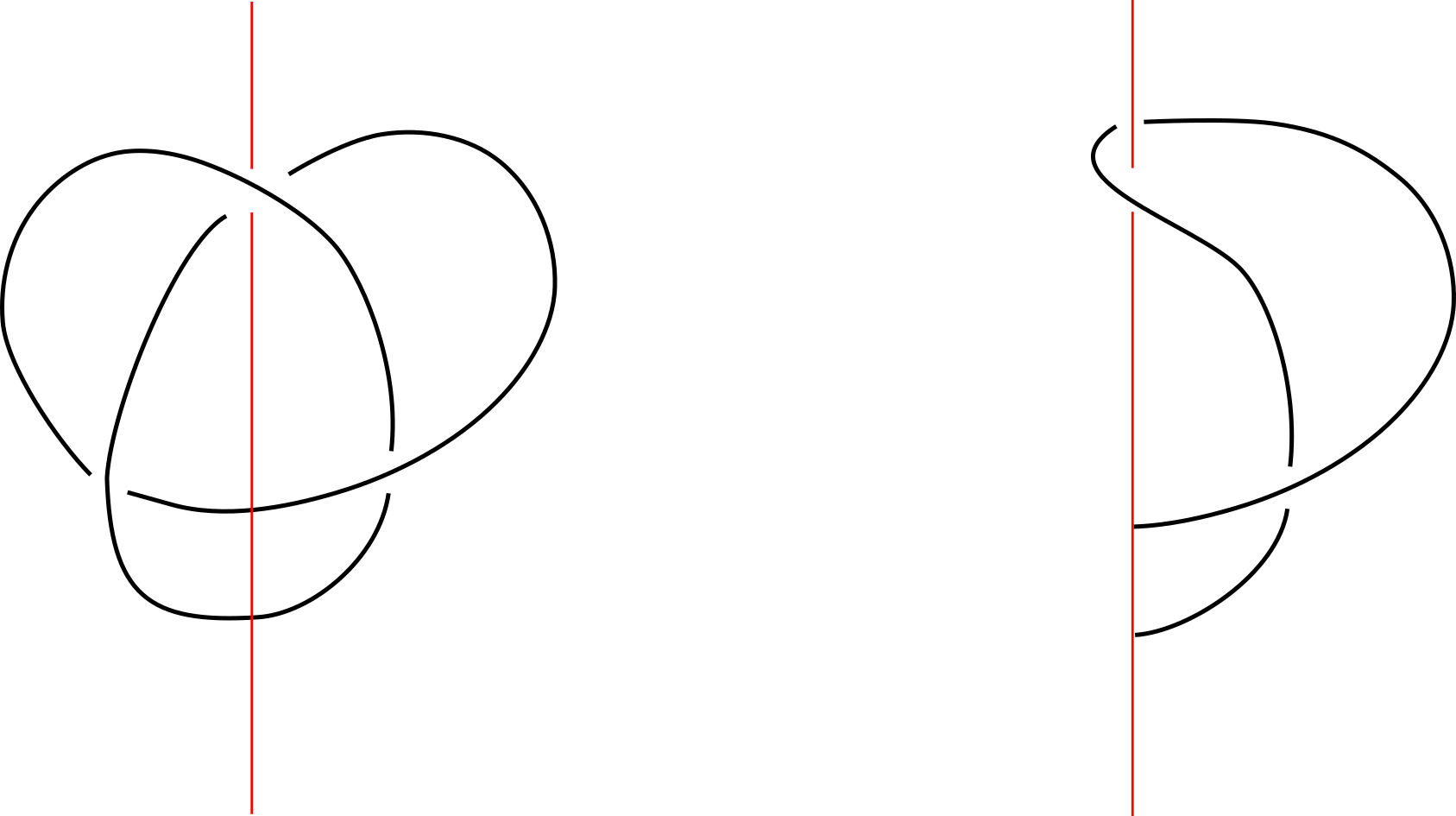}
\caption{The trefoil is a strongly invertible knot. Up to isotopy, we can represent the fixed point set as the $z$ axis in $\mathbb{R}^3$. On the right, the $\theta$-curve obtained from the projection. The unknotted component is again represented as the $z$ axis. }
\label{fig:trefoiltheta}
\end{figure}
From $(K, \tau)$ we can construct the $\theta$-curve $\theta(K, \tau)= p(Fix(\tau)) \cup p(K)$, where $p(K) = e_0$ and $p(Fix(\tau)) = e_- \cup e_+$, as explained in \cite{sakuma} and as shown in Figure \ref{fig:trefoiltheta}. Equivalent strongly invertible knots project to equivalent $\theta$-curves (as elements of $\Theta^s/_\approx$), thus, we have a well defined map $$\beta:\mathcal{K}SI(S^3) \longrightarrow \Theta^s/_\approx $$ The four labelled $\theta$-curves corresponding to the different choices of labelling the edges $e_-$ and $e_+$ and the vertices $v_0$ and $v_1$ are mapped by the isomorphism $t$ of Theorem \ref{thm:semigroupsiso} to knotoids $k$,$-k$, $k_{\textbf{rot}}$ and $-k_{\textbf{rot}}$ related by reversion and rotation, as discussed in Section \ref{miserve2}. Thus, we have a well defined map $$\Pi = t_{\approx}^{-1}  \circ \beta $$ from the set of strongly invertible knots to the set $\mathbb{K}(S^2)/_\approx$ of unoriented knotoids in $S^2$ up to rotation. Since the preferred constituent unknot of $t_{\approx}(t_{\approx}^{-1}(\theta(K,\tau))) = \theta(K,\tau)$ is clearly $p(Fix(\tau))$, $\Pi$ is the inverse of $\gamma_S$. From this and the discussion in Section \ref{sec:st} we obtain that $$ \gamma_S: \mathbb{K}(S^2)/_\approx \longrightarrow \mathcal{K}SI(S^3) $$ is a bijection, and Theorem \ref{thm:completo} is proven.

\subsection{Connected sums}\label{sec:connectedsums}
Call $k_1$ the knotoid on the left-side of Figure \ref{fig:symmetries}, and consider the product $k_1 \cdot k_1$. Its image under $\gamma_S$ is the composite knot $3_1 \# 3_1$ (see \emph{e.g.} Figure \ref{fig:esempiogc}, and recall Theorem \ref{thm:multandcon}). We know from Proposition \ref{liftknottype} that the knot-type knotoid $k$ associated to the trefoil knot lifts to $3_1 \# 3_1$ as well (the trefoil is invertible, thus $3_1 \sim -3_1$). Theorem \ref{thm:completo} implies that $3_1 \# 3_1$ admits at least two non-equivalent involutions, associated to the equivalence classes in $\mathbb{K}(S^2)/_\approx$ of the knotoids $k_1 \cdot k_1$ and $k$, respectively. These non-equivalent involutions are shown in Figure \ref{fig:3_1piu3_1}.

\begin{figure}[h]
\includegraphics[width=4cm]{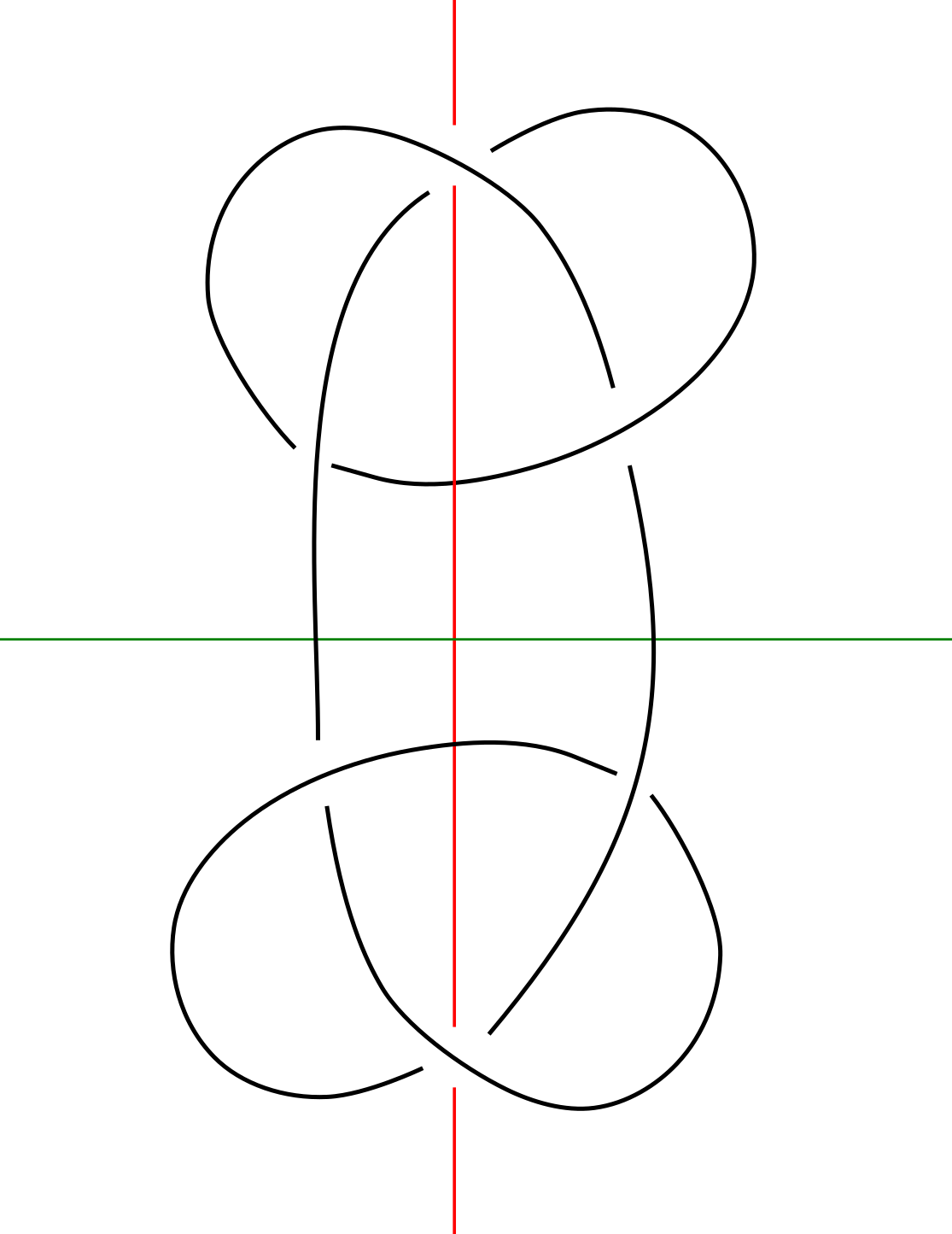}
\caption{The fixed point sets of two non-equivalent involutions are shown here. The one corresponding to the vertical line associates the knot $3_1 \# 3_1$ to $k_1 \cdot k_1$. The quotient under the involution corresponding to the horizontal line is the one associated to the knot-type knotoid $k$.}
\label{fig:3_1piu3_1}
\end{figure}

Similarly, Figure \ref{fig:differentinvo} shows two different involutions of the composite knot $3_1 \# 8_{20}$, defining different composite knotoids.

\begin{figure}[h]
\includegraphics[width=10cm]{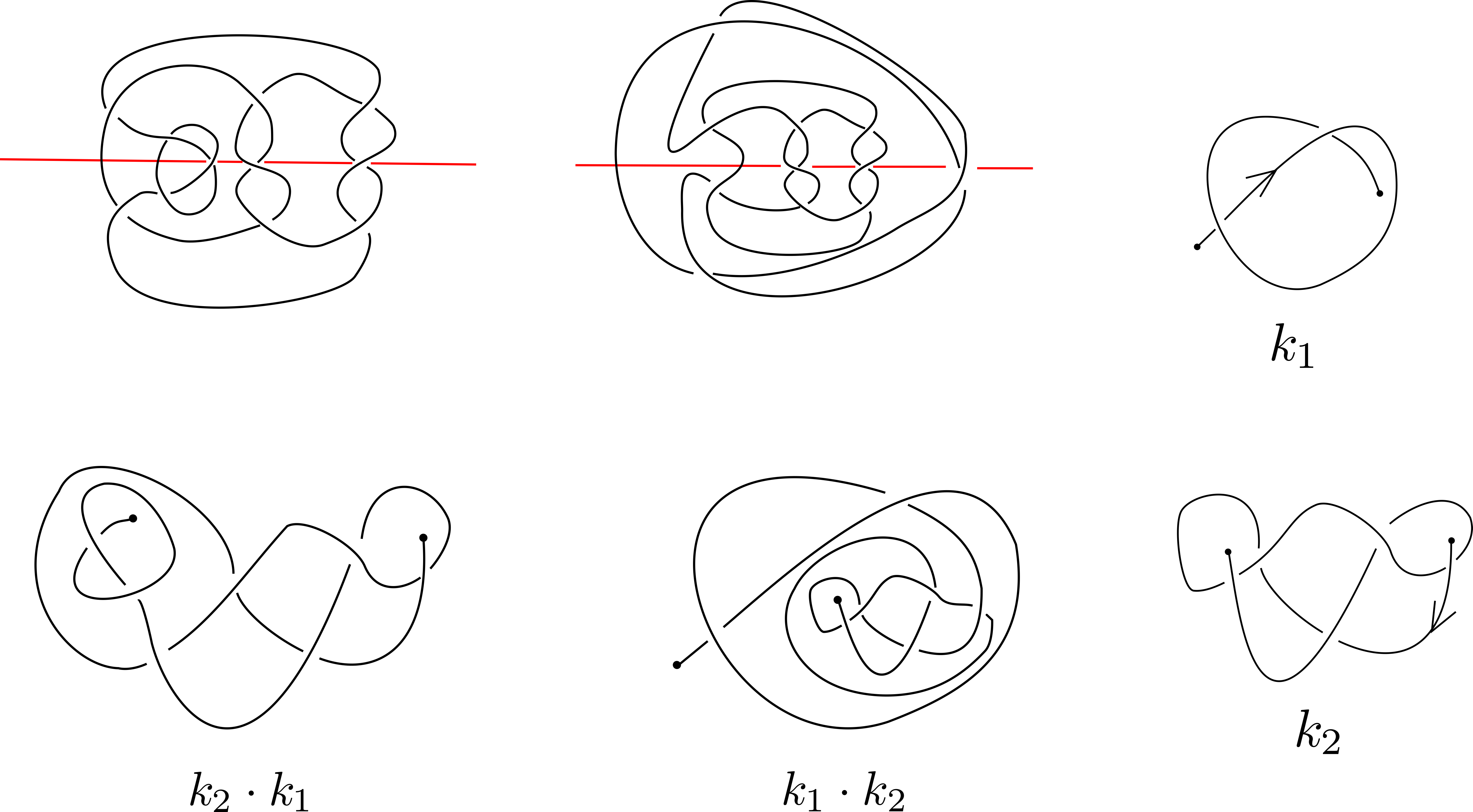}
\caption{Two different involutions of the composite knot $3_1 \# 8_{20}$, associated to the composite knotoids $k_1 \cdot k_2$ and $k_2 \cdot k_1$.}
\label{fig:differentinvo}
\end{figure}

\subsection{Strong inversions}\label{sec:stronginversion}
It is a classical result \cite{finite} that every knot admits a finite number of non equivalent strong inversions. For torus and hyperbolic knots a stronger result holds. Recall that we say that a knot $K$ \emph{admits period $2$} if it is fixed by an orientation preserving involution which also preserves the knot orientation. More precisely, $K$ has \emph{cyclic} (respectively \emph{free}) period $2$ if there exist a non-trivial $\phi \in Sym^+(S^3,K)$ with $\phi^2=id$, that preserves the orientation on $K$, with $fix(\phi)$ an unknot (respectively $fix(\phi) = \varnothing$).

\begin{thm}[Proposition $3.1$, \cite{sakuma}]\label{thm:hyper}
A torus knot admits exactly one strong inversion. If a hyperbolic knot is strongly invertible, then it admits either $1$ or $2$ non equivalent inversions, and it admits exactly $2$ if and only if it admits (cyclic or free) period $2$.
 
\end{thm}

The previous result together with Theorem \ref{thm:completo} proves Corollary \ref{cor:completo2}. Thus, to every torus knot there is a single knotoid associated up to reversion and rotation, and to every hyperbolic knot there are at most two. We give the following definition, borrowed from classical knot theory.

\begin{defi}
We will call \emph{torus knotoids} the knotoids whose lifts are torus knots. Similarly, we will call \emph{hyperbolic knotoids} those lifting to hyperbolic knots. 
 
\end{defi}
More generally, only finitely many knotoids are associated with a single knot type. Hence it is natural to ask the following.
\begin{ques}\label{algorithm}
Is there an algorithm to decide whether two knotoids $k_1$ and $k_2$ are equivalent?
\end{ques}

Since there are now several known ways to solve the knot recognition problem (for a survey, see \emph{e.g} \cite{Lackenby} and \cite{dinni}), the next step to answer Question \ref{algorithm} positively would be to decide whether two given involutions of a knot complement are conjugate homeomorphisms. As stated in the introduction, using the solution to the equivalence problem for hyperbolic knots (\cite{manning} and \cite{kupe}), since there is an algorithm to decide whether two involutions of a hyperbolic knot complement are conjugate \cite{Lackenbynew}, this can be done in the hyperbolic case. Thus, it is possible to tell if two hyperbolic knotoids $k_1$ and $k_2$ represent equivalent classes in $  \mathbb{K}(S^2)/_\approx$. This is enough to distinguish them as oriented knotoids.

\begin{thm}\label{thm:algo1}
Given two hyperbolic knotoids $k_1$ and $k_2 $, there is an algorithm to determine whether $k_1$ and $k_2$ are equivalent as oriented knotoids.  
\end{thm}

\begin{proof}
By the previous discussion, we can tell if $k_1$ and $k_2$ represent equivalent classes in $ \in \mathbb{K}(S^2)/_\approx$. Suppose they do, and note that since the mapping class group of a hyperbolic knot is computable (see \cite{henri}), we can tell whether their lift admits or not period $2$. If it does, then Corollary \ref{cor:completo3} (a proof of which is contained in Section \ref{sec:inam}) tells us that exactly one of the following holds:
\begin{itemize}
 \item $k_1$ and $k_2$ are isotopic as oriented knotoids;
 \item $k_1$ is isotopic to $k_{2\textbf{rot}}$.
\end{itemize}
We can then consider the diagrams of $k_1$ and $k_2$ given as an input, and any diagram of $k_{2\textbf{rot}}$. Now, if two diagrams represent equivalent knotoids, then there exists a finite sequence of Reidemeister moves and isotopies taking one to the other. Since we know that $k_1$ is equivalent to one between $k_2$ and $k_{2\textbf{rot}}$, after a finite number of Reidemeister moves and isotopies performed on $k_1$, this will transform into either $k_2$ or $k_{2\textbf{rot}}$. Similarly, if the lift of $k_1$ and $k_2$ does not admit period $2$, Corollary \ref{cor:completo3} assures that exactly one of the following holds: 
\begin{itemize}
 \item $k_1$ and $k_2$ are isotopic;
 \item $k_1$ is isotopic to $-k_{2}$;
 \item $k_1$ is isotopic to $k_{2\textbf{rot}}$;
 \item $k_1$ is isotopic to $-k_{2\textbf{rot}}$.
\end{itemize}
And we can distinguish between these possibilities exactly as before. 
\end{proof}

\begin{rmk}
Note that Question \ref{algorithm} can be answered positively using the correspondence between knotoids and $\theta$-curves (Theorem \ref{thm:semigroupsiso}). Indeed, given two $\theta$-curves, we can consider their complements in $S^3$, together with the data of the meridians of the three edges. We could then let Haken's algorithm (see \cite{haken}, \cite{abbi}) run to decide whether or not the obtained $3$-manifolds are equivalent. However, the algorithm of Theorem \ref{thm:algo1} appears to be practical, whereas Haken's algorithm is not. 

\end{rmk}

\subsection{An example: the $T_{2,2k+1}$-torus knotoids}\label{sec:torusknots}
Every $T_{2,2k+1}$-torus knot admits a diagram of the form shown in the upper part of Figure \ref{fig:torusinvertibile1}, where its unique involution $\tau$ is represented as a straight line.  

\begin{figure}[h]
\includegraphics[width=10cm]{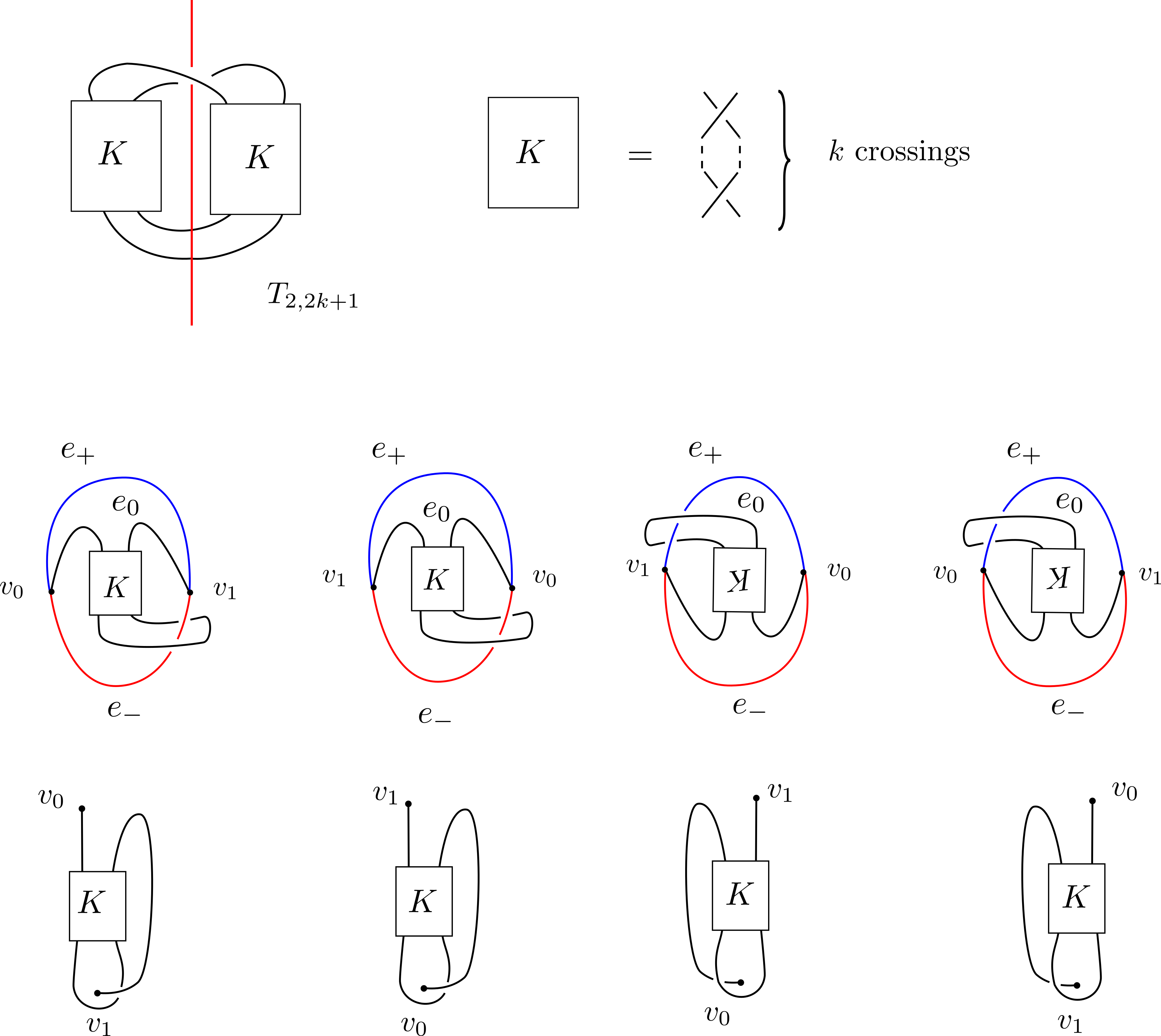}
\caption{On the top, a diagram representing the torus knot $T_{2,2k+1}$. The $K$ box contains $k$ consecutive crossings of the same sign, as shown on the top-right side of the picture. Note that rotating $K$ by $\pi$ does not change it. The unique involution $\tau$ of $T_{2,2k+1}$ gives $4$ labelled $\theta$-curves and their associated oriented knotoids. These are related to one another by reversion and rotation.}
\label{fig:torusinvertibile1}
\end{figure}
Recall that, as defined in Section \ref{sec:proofcorrespondence}, the inverse of $\gamma_S$ is given by $\Pi = t_{\approx}^{-1}  \circ \beta $, where $\beta$ is the map from the set of strongly invertible knots $\mathcal{K}SI(S^3)$ to the set $\Theta^s/_\approx$ of equivalence classes of simple and labeled $\theta$-curves, and $t_{\approx}$ is the bijection $t_{\approx}: \mathbb{K}(S^2)/_\approx \longrightarrow \Theta^s/_\approx$. In the equivalence class of $\beta(T_{2,2k+1})$ there are \emph{a priori} $4$ distinct labelled $\theta$-curves. These are shown in the middle of Figure \ref{fig:torusinvertibile1}, over their associated knotoids. Using these canonical representatives of $\Pi(T_{2,2k+1}, \tau)$ we can prove the following.  

\begin{prop}
Every $T_{2,2k+1}$-torus knotoid is reversible.
\end{prop}

\begin{proof}
First, observe that it is enough to prove that one of the knotoids in Figure \ref{fig:torusinvertibile1} is reversible, since rotation and reversion commute. The proof is contained in Figure \ref{fig:torusinvertibile}.

\end{proof}

\begin{figure}[h]
\includegraphics[width=3cm]{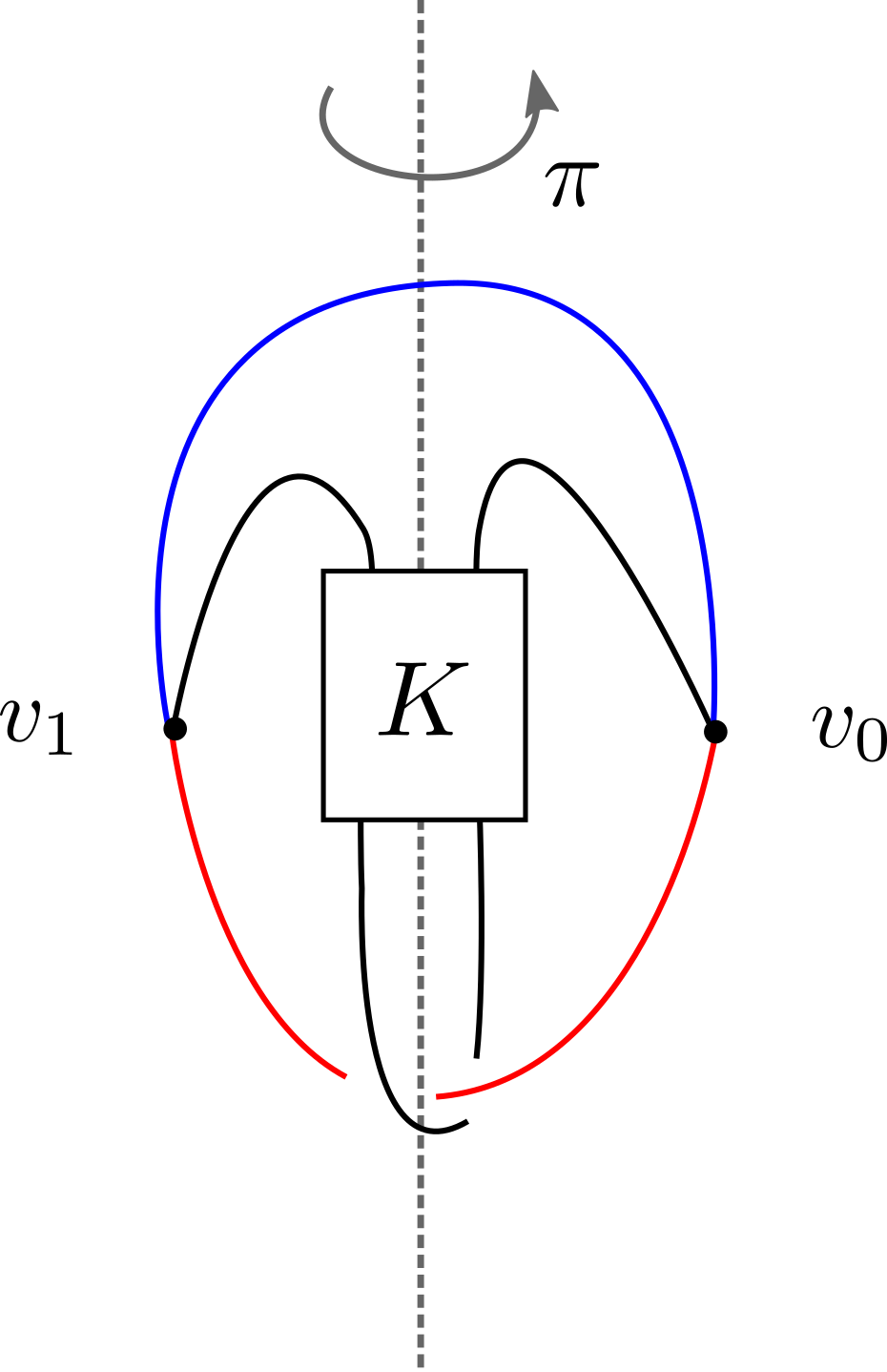}
\caption{A proof that the $T_{2,2k+1}$-torus knotoids are reversible. The rotation by $\pi$ sends the $\theta$-curve associated to a $T_{2,2k+1}$-torus knotoid into the one corresponding to its reverse.}
\label{fig:torusinvertibile}
\end{figure}

The previous proposition, together with Corollary \ref{cor:completo2} implies that there are at most $2$ oriented knotoids associated to every $T_{2,2k+1}$-torus knot. These knotoids are reversible and related by a rotation. 

\section{Amphichirality, reversibility and rotatability in hyperbolic knotoids}\label{sec:inam}
In this section we show how Corollary \ref{cor:completo2} can be improved in two different ways, by considering properties of the symmetry groups of hyperbolic knots.

\subsection{Orbifolds}

Much of this section uses the machinery of orbifolds. The formal definition of an orbifold is given in \cite{scott}[Section 2] and \cite{conemanifolds}[Section 2.1]. We omit it here because it is quite lengthy. Informally, an orbifold is a space where each point $x$ has an open neighbourhood of the form $\mathbb{R}^n / \Gamma$ for some finite subgroup $\Gamma$ of $O(n)$, and where $x$ is taken to the image of the origin. However, it is more than just a topological space, since the orbifold also records the specific group $\Gamma$ attached to $x$. This is called the \emph{local group} of $x$. The \emph{singular locus} of the orbifold is the set of points with non-trivial local group.

The most basic example of an orbifold is the quotient of a smooth manifold $M$ by a finite group of diffeomorphisms. In this paper, we will only consider the case of the group $\mathbb{Z}_2$ acting on a smooth manifold by an orientation-preserving involution $\tau$. In this case, the quotient $M / \tau$ is again a manifold and its singular locus is a properly embedded 1-manifold.

The most signifcant structural result about 3-orbifolds is the Orbifold Theorem \cite{conemanifolds, BoileauLeebPorti}, which gives a version of the Geometrisation Conjecture for orbifolds. Rather than state the precise theorem, we observe the following well-known consequence.

\begin{thm}
\label{Thm:ConjugateToIsometries}
Let $\Gamma$ be a finite group of diffeomorphisms of a finite-volume hyperbolic 3-manifold $M$. Then $\Gamma$ is conjuagate to a group of isometries.
\end{thm}

\begin{proof}
Mostow Rigidity implies that $\Gamma$ is homotopic to a group of isometries, but this is weaker than our desired conclusion. Instead, we use the Orbifold Theorem. This implies that the orbifold $M / \Gamma$ is hyperbolic. Hence, $M$ has a hyperbolic structure upon which $\Gamma$ acts by isometries. By Mostow Rigidity, this hyperbolic structure is isometric to our given one, via an isometry $h$. Thus, $h$ is our required conjugating diffeomorphism.
\end{proof}

There is a notion of the fundamental group $\pi_1(\mathcal{O})$ of an orbifold $\mathcal{O}$ and of a covering map between orbifolds. For a definition of these terms, see \emph{e.g.}~\cite{scott} or \cite{conemanifolds}.

\subsection{Reversible knotoids}\label{invertible}
Let's turn our attention back to oriented knotoids. Even if the maps $\gamma_S$ and $\gamma_T$ cannnot distinguish between two knotoids differing only in the orientation, using $\gamma_S$ it is possible to tell whether a hyperbolic knotoid is reversible or not, and whether it is rotatable or not. The following lemma follows from Theorem \ref{thm:semigroupsiso}.

\begin{lemma}\label{propp}
An oriented knotoid $k \in \mathbb{K}(S^2) $ is reversible (respectively rotatable) if and only if there is an isotopy taking the $\theta$-curve $t(k)$ back to itself, that preserves each edge but swaps the two vertices (respectively, that preserves $e_0$, and swaps $e_+$ and $e_-$ and the vertices). 
\end{lemma}

Furthermore, we can prove that the isotopies of the previous lemma have order $2$.

\begin{prop}\label{lemma:xxx}
A hyperbolic, oriented knotoid $k \in \mathbb{K}(S^2) $ is reversible (respectively, equivalent to the reverse of its rotation) if and only if there is an order two orientation preserving homeomorphism of $S^3$ taking the $\theta$-curve $t(k)$ back to itself, that perserves each edge but swaps the two vertices (respectively, that preserves $e_0$ and swaps the two vertices and $e_-$ and $e_+$).
\end{prop}

\begin{proof}
One direction is clear. So suppose that $k$ is a hyperbolic knotoid that is either reversible or equivalent to the reverse of its rotation. Let $t(k)$ be the corresponding $\theta$-curve, with edges labeled $e_-$, $e_+$ and $e_0$, and where $e_- \cup e_+$ is the preferred constituent unknot. By hypothesis, $\gamma_S(k)$ is a hyperbolic knot. The involution on the hyperbolic manifold $S^3 \setminus \text{int}(N(\gamma_S(k)))$ is realised by a hyperbolic isometry $\tau$ by Theorem \ref{Thm:ConjugateToIsometries}. The quotient $(S^3 \setminus \text{int}(N(\gamma_S(k)))/\tau$ is therefore a hyperbolic orbifold $\mathcal{O}$. Its underlying space is the $3$-ball $S^3 \setminus \text{int}(N(e_0))$, and its singular locus is the intersection with $e_- \cup e_+$. Now, we are assuming that $k$ is either reversible or equivalent to the reverse of its rotation. Hence, by Lemma \ref{propp}, there is a homeomorphism $\rho$ of $S^3$ taking $t(k)$ back to itself, that perserves each edge but swaps the two vertices, or that swaps both the vertices and the edges $e_-$ and $e_+$. In both cases, this therefore induces a homeomorphism of $\mathcal{O}$ that preserves its singular locus. By Mostow rigidity, this is homotopic to an isometry $\overline{\rho}$ of $\mathcal{O}$. In both cases, the action of $\overline{\rho}^2$ on $\partial \mathcal{O}$ is isotopic to the identity, via an isotopy that preserves the singular points throughout. Hence, because it is an isometry of a Euclidean pillowcase orbifold, $\overline{\rho}^2$ is the identity. Therefore, $\overline{\rho}^2$ is actually equal to the identity on $\mathcal{O}$. So, in both cases, $\overline \rho$ extends to the required order two homeomorphism of $S^3$, taking $t(k)$ back to itself, that perserves each edge but swaps the two vertices or that preserves $e_0$ and swaps the vertices and the two other edges. 
\end{proof}

Sakuma and Kodama \cite{kodama} proved that, given an invertible hyperbolic knot $K$ with a strong involution $\tau$, the existence of such symmetries for the $\theta$-curve $\theta(K,\tau)$ is completely determined by $Sym(S^3,K)$.

\begin{thm}[Proposition $1.2$, \cite{kodama}]\label{thm:thetaandinvertible}
Given an invertible hyperbolic knot $K$ with a strong inversion $\tau$, then $Sym(S^3,K)$ admits cyclic (respectively free) period $2$ if and only if there exist an orientation preserving involution of $S^3$ fixing setwise $\theta(K,\tau)$ that preserves each edge but swaps the two vertices (respectively, that preserves $e_0$, and swaps $e_+$ and $e_-$ and the vertices)\footnote{In \cite{kodama} Sakuma and Kodama call $\theta(K,\tau)$ \emph{strongly reversible} in the first case, and say that $\theta(K,\tau)$ has \emph{period $2$ centered in $e_0 = p(K)$} in the second case.}. 

\end{thm}

As an immediate corollary of Theorem \ref{thm:thetaandinvertible} and Lemma \ref{lemma:xxx} we have the following characterisation for hyperbolic knotoids, which is a restatement of Theorem \ref{llll}.

\begin{thm}\label{thm:detectinvert}
A hyperbolic, oriented knotoid $k \in \mathbb{K}(S^2) $ is reversible if and only if $\gamma_S(k)$ has cyclic period $2$. Analogously, it is equivalent to the reverse of its rotation if and only if $\gamma_S(k)$ has free period $2$.
\end{thm}

\subsection{Hyperbolic knotoids are not rotatable}
Recall that a knotoid $k \in \mathbb{K}(S^2)$ is called rotatable if it is isotopic to its rotation $k_{\textbf{rot}}$. For hyperbolic knotoids, this is never the case. 

\begin{thm}\label{thm:rotmarc}
A hyperbolic knotoid is never rotatable.
\end{thm}

\begin{proof}
Associated to the hyperbolic knotoid $k$, there is the $\theta$-curve $t(k)$, which has three edges $e_0$, $e_-$ and $e_+$, and where $e_- \cup e_+$ is the preferred constituent unknot. By hypothesis, $\gamma_S(k)$ is a hyperbolic knot. As in the proof of Proposition \ref{lemma:xxx}, the involution on the hyperbolic manifold $S^3 \setminus \text{int}(N(\gamma_S(k)))$ is realised by a hyperbolic isometry $\tau$. The quotient $(S^3 \setminus \text{int}(N(\gamma_S(k)))/\tau$ is a hyperbolic orbifold $\mathcal{O}$. Its underlying space is the $3$-ball $S^3 \setminus \text{int}(N(e_0))$, and its singular locus is the intersection with $e_- \cup e_+$.
Let $k_{\textbf{rot}}$ be the knotoid obtained by rotating $k$. If we suppose that $k_{\textbf{rot}}$ and $k$ are equivalent, then by Lemma \ref{lemma:xxx} there is a homeomorphism of $S^3$ taking $t(k)$ back to itself that swaps the edges labelled $e_-$ and $e_+$ and leaves all other labels unchanged. This therefore induces a homeomorphism $h \colon \mathcal{O} \rightarrow \mathcal{O}$ that preserves the singular locus. It is homotopic to an isometry $\overline h \colon \mathcal{O} \rightarrow \mathcal{O}$. 

We claim that $\overline{h}$ is not equal to the identity, but that $(\overline{h})^2$ is the identity. First observe that $h$ swaps the two components of the singular locus. These are distinct geodesics in $\mathcal{O}$. Thus, the isometry $\overline{h}$ also swaps these two geodesics and therefore is not the identity. On the other hand, $h^2$ acts as the identity on $\partial N(t(k))$. Hence, the restriction of $(\overline{h})^2$ to $\partial \mathcal{O}$ is isotopic to the identity, via an isotopy that preserves the singular points of $\partial \mathcal{O}$. But any isometry of a Euclidean pillowcase orbifold that is isotopic to the identity via an isotopy preserving the singular points must be equal to the identity. So the restriction of the isometry $(\overline{h})^2$ to $\partial \mathcal{O}$ is the identity and hence $(\overline{h})^2$ is the identity on $\mathcal{O}$. By the double branched cover construction, we obtain $p \colon S^3 - {\rm int}(N(\gamma_S(k))) \rightarrow \mathcal{O}$ that is an orbifold covering map. This induces a homomorphism $p_\ast \colon \pi_1(S^3 - {\rm int}(N(\gamma_S(k))) \rightarrow \pi_1(\mathcal{O})$. The image of this homomorphism is an index $2$ subgroup of $\pi_1(O)$. This subgroup consists of those loops in $S^3 - {\rm int}(N(t(k)))$ that have even linking number with the unknot $e_- \cup e_+$.

Now $\overline{h}$ lifts to an isometry $\phi \colon S^3 \setminus \mathrm{int}(N(\gamma_S(k))) \rightarrow S^3 \setminus \mathrm{int}(N(\gamma_S(k)))$. This is because $\overline{h}_\ast \colon \pi_1(\mathcal{O}) \rightarrow \pi_1(\mathcal{O})$ preserves the subgroup $p_\ast \pi_1(S^3 - {\rm int}(N(\gamma_S(k)))$. This lift $\phi$ swaps the arcs $p^{-1}(e_- \cap \mathcal{O})$ and $p^{-1}(e_+ \cap \mathcal{O})$. It preserves each of the meridians of $\gamma_S(k)$ at the endpoints of these arcs. To see this, pick a small arc $\alpha$ in $\partial \mathcal{O}$ near one of the endpoints of $k$, joining a point of $e_- \cap \partial \mathcal{O}$ to a point of $e_+ \cap \partial \mathcal{O}$. This is preserved by $h$ and hence the Euclidean geodesic representative for $\alpha$ is preserved by $\overline{h}$. The inverse image of this geodesic in $S^3 \setminus \mathrm{int}(N(\gamma_S(k)))$ is therefore a meridian of $\gamma_S(k)$ that is preserved by $\phi$. Hence, $\phi$ extends to an involution of $S^3$ that fixes $\gamma_S(k)$. But no such symmetry exists, by the solution to the Smith Conjecture, since $\gamma_S(k)$ is a non-trivial knot.
\end{proof}

\begin{rmk}\label{rmk:onlyoneinvo}
Consider a hyperbolic knotoid $k$, and suppose that its lift $\gamma_S(k)$ admits simultaneously free and cyclic period $2$. Then, by Theorem \ref{thm:detectinvert}, $k$ is equivalent both to $-k$ and to $-k_{\textbf{rot}}$. This imply that $k$ is equivalent to its rotation $k_{\textbf{rot}}$, contradicting Theorem \ref{thm:rotmarc}. Thus, Theorem \ref{thm:rotmarc} and Theorem \ref{thm:detectinvert} imply that a strongly invertible hyperbolic knot $K$ can not admit simultaneously free and cyclic period $2$. 
We believe that this statement holds for hyperbolic knots in general, but we weren't able to find a reference.
\end{rmk}

Thus, we obtain the following improvement of Corollary \ref{cor:completo2}, dealing with hyperbolic knotoids.

\begin{cor}\label{atmost2}
Given any strongly invertible hyperbolic knot $K$ there are exactly $4$ distinct oriented knotoids associated to it. Moreover, one of the following holds.
\begin{itemize}

 \item If $K$ has cyclic period $2$, these are two inequivalent reversible knotoids $k^1$, $k^2$ and their rotations $k^1_{\textbf{rot}}$, $k^2_{\textbf{rot}}$;

 \item if $K$ has free period $2$, these are two inequivalent knotoids $k^1$, $k^2$ (each equivalent to the reverse of its rotation) and their reverses $-k^1$, $-k^2$;
 \item if $K$ does not have period $2$, these are a knotoid $k$, its reverse $-k$, its rotation $k_{\textbf{rot}}$ and its reverse rotation $-k_{\textbf{rot}}$.
\end{itemize}
 
\end{cor}

\begin{proof}
By Theorem \ref{thm:hyper}, $K$ has either $1$ or $2$ strong inversions up to equivalence. It has exactly $2$ if and only if $K$ has period $2$.

Suppose first that $K$ does not have period $2$. Then, up to reversion and rotation, it is associated with just one knotoid $k$. Moreover, by Theorems \ref{thm:detectinvert} and \ref{thm:rotmarc}, $k$, $-k$, $k_{rot}$ and $-k_{rot}$ are all distinct.

Suppose that $K$ has cyclic period $2$. Then, it is associated with two knotoids $k^1$ and $k^2$ that are distinct, even up to rotation and reversion. By Theorem \ref{thm:detectinvert}, $k^1$ is equivalent to $-k^1$. Hence, $k^1_{rot}$ is equivalent to $-k^1_{rot}$. By Theorem \ref{thm:rotmarc}, $k^1$ and $-k^1$ are distinct from $k^1_{rot}$ and $-k^1_{rot}$. Similar statements are true for $k^2$.

Finally, suppose that $K$ has free period $2$. Then again it is associated with two knotoids $k^1$ and $k^2$ that are distinct up to reversion and rotation. By Theorem \ref{thm:detectinvert}, $k^1$ is equivalent to $-k^1_{rot}$. Hence, $-k^1$ is equivalent to $k^1_{rot}$. By Theorem \ref{thm:rotmarc}, $k^1$ and $-k^1_{rot}$ are distinct from $-k^1$ and $k^1_{rot}$. Again, similar statements are true for $k^2$.
\end{proof}

In particular, the knotoids in Figure \ref{fig:8_20} are all not reversible. In fact, their images under the double branched cover construction are the knot $8_{20}$, which is hyperbolic with symmetry group isomorphic to the dihedral group $D_1$ (thus, it does not admit period $2$).

\begin{figure}[h]
\includegraphics[width=12cm]{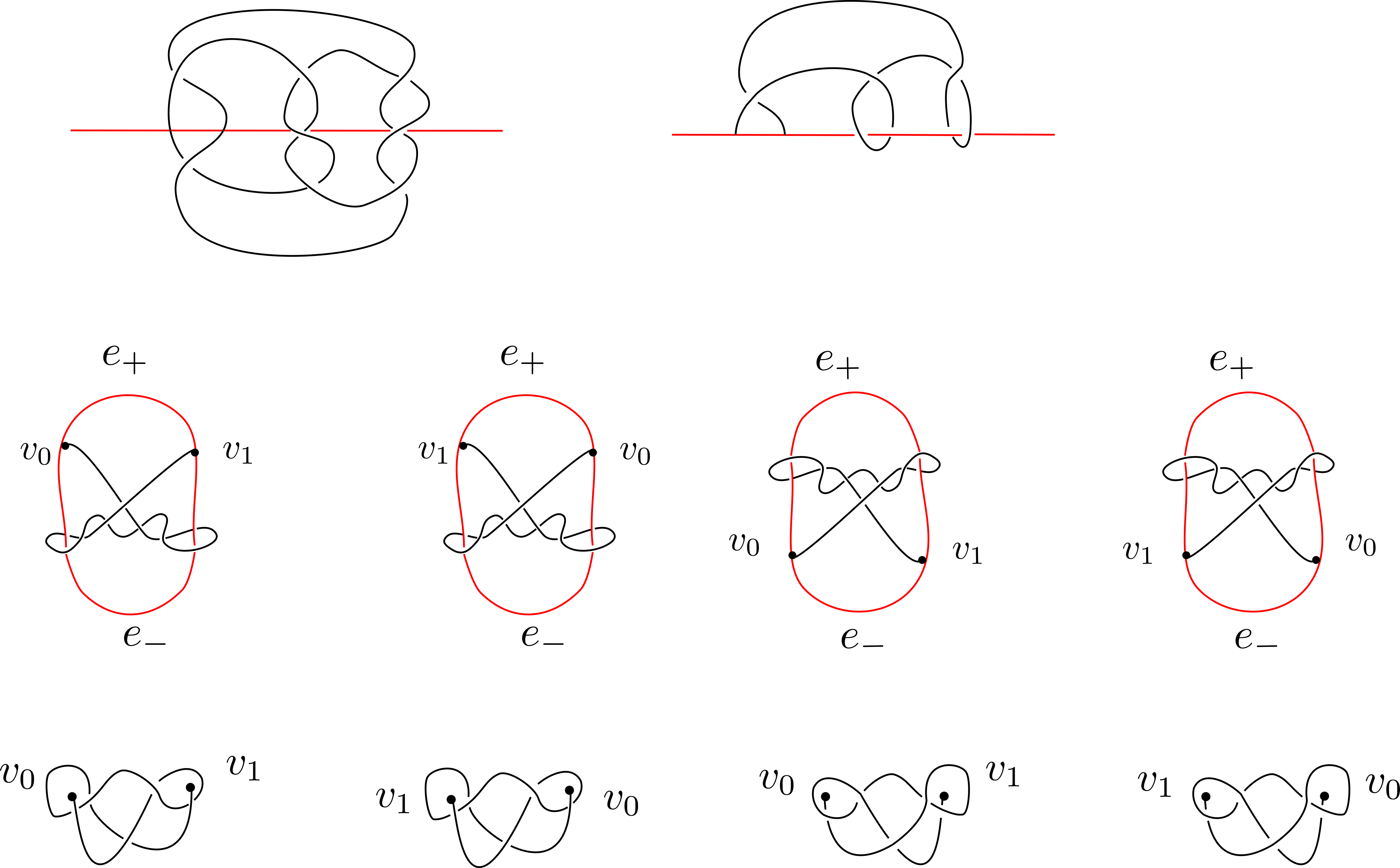}
\caption{There are $4$ oriented and not reversible knotoids associated to the knot $8_{20}$. These are related to each other by reversion and/or rotation.}
\label{fig:8_20}
\end{figure}

\subsection{Amphichiral strongly invertible knots}
It is possible to give a refinement of Corollary \ref{cor:completo2} in the case of amphichiral hyperbolic knots.

\begin{defi}\label{defi:amphi}
A knot $K$ is called \emph{amphichiral} if there exists an orientation-reversing homeomorphism of $S^3$ fixing the knot (setwise). Note that this implies that $K$ is equivalent to its mirror $K_m$.
\end{defi}

Consider an invertible, hyperbolic, amphichiral knot $K$, and suppose that it admits period $2$. From Theorem \ref{thm:hyper} it follows that $K$ admits two non-equivalent involutions $\tau_1$ and $\tau_2$. Let $\phi$ be the (isotopy class of) the orientation reversing homeomorphism of Definition \ref{defi:amphi}; from  [\cite{sakuma}, Proposition $3.4$], we know that $\tau_1$ and $\tau_2$ are conjugated through $\phi$ $$ \tau_2 = \phi \circ \tau_1 \circ \phi^{-1}$$

Thus,  $(K,\tau_1)$ is equivalent to $m(K,\tau_2)$, where $m(K,\tau_2)$ is the strongly invertible knot obtained from $(K,\tau_2)$ by reversing the orientation of $S^3$, and the following holds.

\begin{prop}\label{amphi}
Given an invertible, hyperbolic, amphichiral knot $K$ admitting period $2$, and let $\tau_1$ and $\tau_2$ be the two non-equivalent strong involutions of $K$. Then $\Pi(K,\tau_1)$ is the equivalence class in $\mathbb{K}(S^2)/_\approx$ containing the mirror images of the knotoids in $\Pi(K,\tau_2)$.
 
\end{prop}
Note that Theorem \ref{thm:detectinvert} tells us that the oriented knotoids in the equivalence classes of $\Pi(K,\tau_1)$ and $\Pi(K,\tau_2)$ are all either reversible or equivalent to the reverse of their rotations, and Theorem \ref{amphi} tells us that each class contains the mirror reflections of the other.

\subsubsection{Example: $4_1$}\label{41}
We work out the case of the figure eight knot ($4_1$ in the Rolfsen table) as an example of Proposition \ref{amphi}. The $4_1$ knot is known to be hyperbolic, invertible, amphichiral and it admits period $2$; thus, it admits two distinct inversions $\tau_1$ and $\tau_2$, shown in the upper part of Figure \ref{fig:4_1}.

\begin{figure}[h]
\includegraphics[width=6cm]{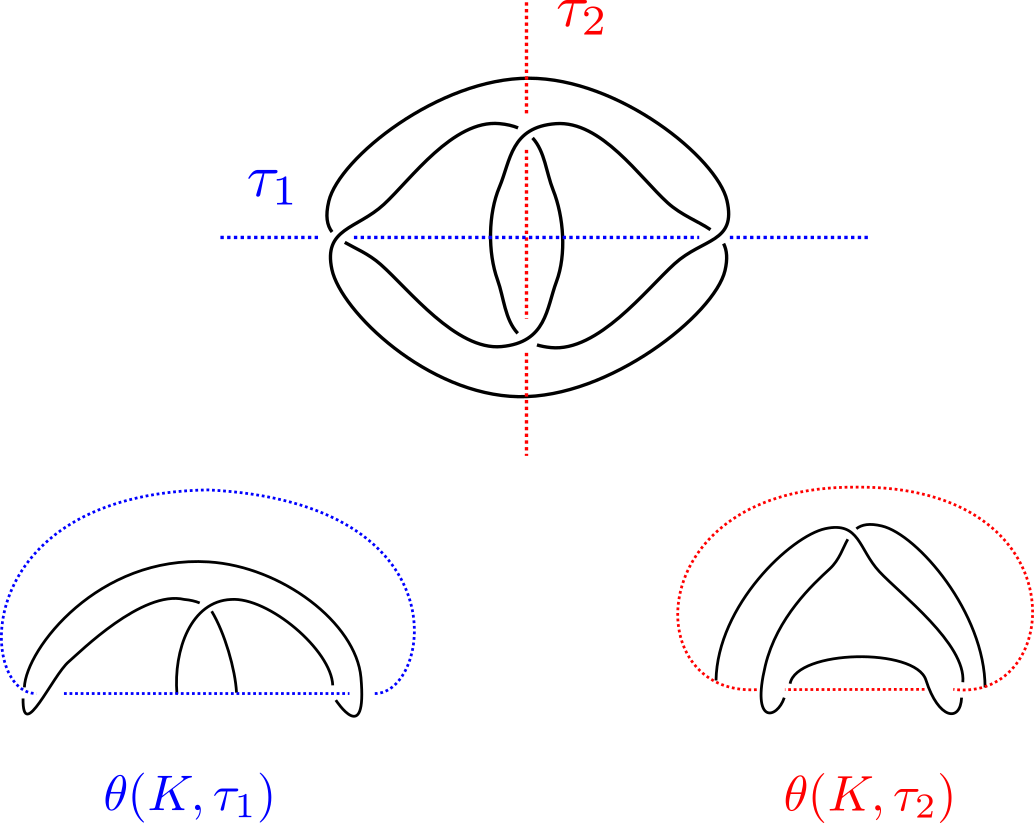}
\caption{On the top, a diagram for $4_1$ with the fixed point sets of $\tau_1$ and $\tau_2$ represented as straight lines. On the bottom, the $\theta$-curves $\theta(4_1,\tau_1)$ and $\theta(4_1,\tau_2)$.}
\label{fig:4_1}
\end{figure}

By considering the quotients under $\tau_1$ and $\tau_2$ we obtain two elements $\theta(4_1,\tau_1)$ and $\theta(4_1,\tau_2)$ of $\Theta^s/_{\approx}$, shown in the bottom of Figure \ref{fig:4_1}. Their constituent knots are two unknots and the torus knot $5_1$, and two unknots and the mirror image of $5_1$ respectively. Since it is well known that $5_1 \nsim m5_1$, it follows that $\theta(4_1,\tau_1) $ and $ \theta(4_1,\tau_2)$ represent different elements of $\Theta^s/_{\approx}$.

\begin{figure}[h]
\includegraphics[width=7cm]{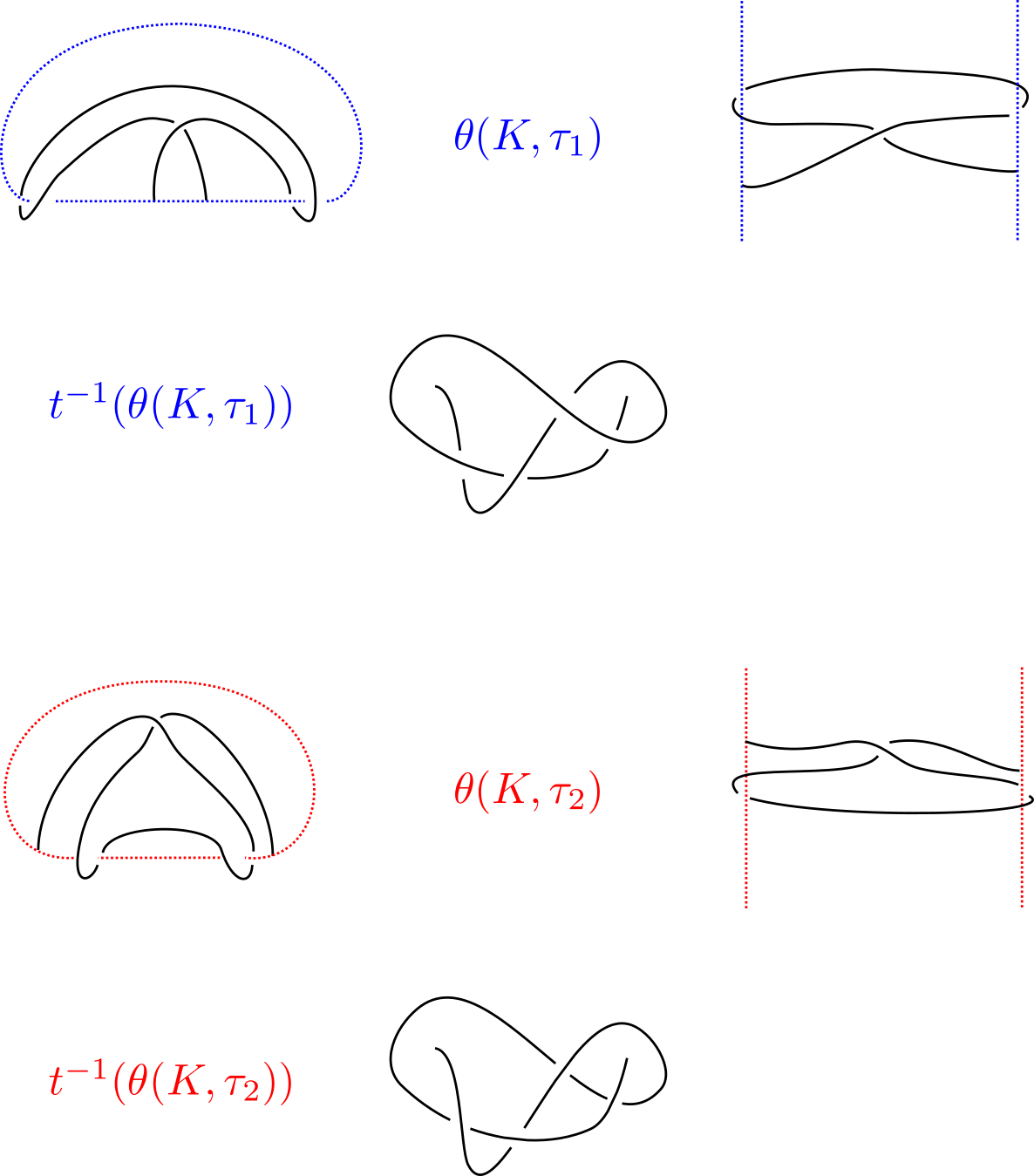}
\caption{On the top: from $\theta(4_1,\tau_1)$ to $\Pi(4_1,\tau_1)$. On the bottom: from $\theta(4_1,\tau_2)$ to $\Pi(4_1,\tau_2)$. The two knotoids are one the mirror image of the other.}
\label{fig:4_1rosso}
\end{figure}

In figure \ref{fig:4_1rosso} we show how to obtain two specific representatives of the equivalence classes $\Pi(4_1,\tau_1)$ and $\Pi(4_1,\tau_2)$. It is clear from the picture that these knotoids are one the mirror image of the other.

\subsection{Strongly invertible composite knots}\label{sec:quanteinvo}
As a corollary of Theorem \ref{thm:completo} and Corollary \ref{cor:completo3}, we have the following result dealing with strong involutions of composite knots.

\begin{prop}\label{prop:quanteinvocomposite}
Consider a knot $K$ isotopic to the connected sum of $\#_{i=1}^n K^i_h$, where $n \geq 2$ and every $K_h^i$ is a strongly invertible, hyperbolic knot. Suppose that these hyperbolic knots are pairwise distinct. Then, the number of non-equivalent strong involutions of $K$ is equal to $4^{n-1}(n!)$.
\end{prop}

\begin{proof}
By Theorem \ref{thm:completo}, the number of different involutions of $K = \#_{i=1}^n K^i_h$ is equal to the number of elements of $\mathbb{K}(S^2)/_\approx$ whose image under $\gamma_S$ is a knot isotopic to $K$. Theorem \ref{thm:main} implies that every knotoid $k$ with $\gamma_S(k)=K$ is decomposable as $k=\cdot_{j=1}^n k_j$ where each $k_j $ belong to the equivalence class $\Pi(K^i_h, \tau) \in \mathbb{K}(S^2)/_\approx$ for a unique $(K_h^i, \tau)$. By hypothesis, the summands of $\#_{i=1}^n K^i_h$ are pairwise distinct, thus, none of the $k_j$ is a knot-type knotoid. Thus, Theorem \ref{unicadecomp} implies that this decomposition is unique in $\mathbb{K}(S^2)$. There are $n!$ ways to order the factors of $k=\cdot_{j=1}^n k_j$, and each arrangement corresponds to a different element in $\mathbb{K}(S^2)$. Moreover, by Corollary \ref{cor:completo3}, each $K^i_h$ is associated to exactly $4$ inequivalent oriented knotoids. Depending on whether or not $K^i_h$ has period $2$, these can be either 
\begin{itemize}
\item two inequivalent reversible knotoids $k^1_i$, $k^2_i$ and their rotations $k^1_{i\textbf{rot}}$, $k^2_{i\textbf{rot}}$, contained in the two classes $\Pi(K_h^i, \tau_1)$ and $\Pi(K_h^i, \tau_2)$;

\item two inequivalent knotoids $k^1_i$, $k^2_i$ (each equivalent to the reverse of its rotation) and their reverses $-k^1$, $-k^2$, contained in the two classes $\Pi(K_h^i, \tau_1)$ and $\Pi(K_h^i, \tau_2)$;

 \item a knotoid $k_i$, its reverse $-k_i$, its rotation $k_{i\textbf{rot}}$ and its reverse rotation $-k_{i\textbf{rot}}$, contained in $\Pi(K_h^i, \tau)$.
\end{itemize}
Choosing a different oriented knotoid associated to the same hyperbolic knot in the decomposition $k=\cdot_{j=1}^n k_j$ corresponds to creating a different element in $\mathbb{K}(S^2)$. Thus, there are a total of $4^n(n!)$ different composite knotoids in $\mathbb{K}(S^2)$ whose double branched cover is a knot isotopic to $K$. Since for every knotoid $k' = k'_1 \cdot k'_2 \dots k'_{m-1} \cdot k'_m$ it holds $k'_{\textbf{rot}} = k'_{1\textbf{rot}} \cdot k'_{2\textbf{rot}} \dots k'_{m-1\textbf{rot}} \cdot k'_{m\textbf{rot}}$ and $-k' = -k'_m \cdot -k'_{m-1} \dots -k'_2\cdot -k'_1$, by considering reversion and reflection on the composite knotoids, the claim follows.
\end{proof}

\section{On the map $\gamma_T$: an example}\label{sec:dimos}
It is often hard to distinguish non-equivalent planar knotoids which represent the same class in $\mathbb{K}(S^2)$. Important developments in this direction have been carried on in \cite{knotoID}, where polynomial invariants are used to detect the planar knotoid types associated to open polymers. In what follows, we show how we can efficiently use the map $\gamma_T$ to this end. Consider the pair of knotoids $k_1$ and $k_2$ in $\mathbb{K}(\mathbb{R}^2)$ of Figure \ref{fig:dimos}, on the top. They both represent the trivial knotoid in $\mathbb{K}(S^2)$.  

\begin{figure}[h]
\includegraphics[width=6cm]{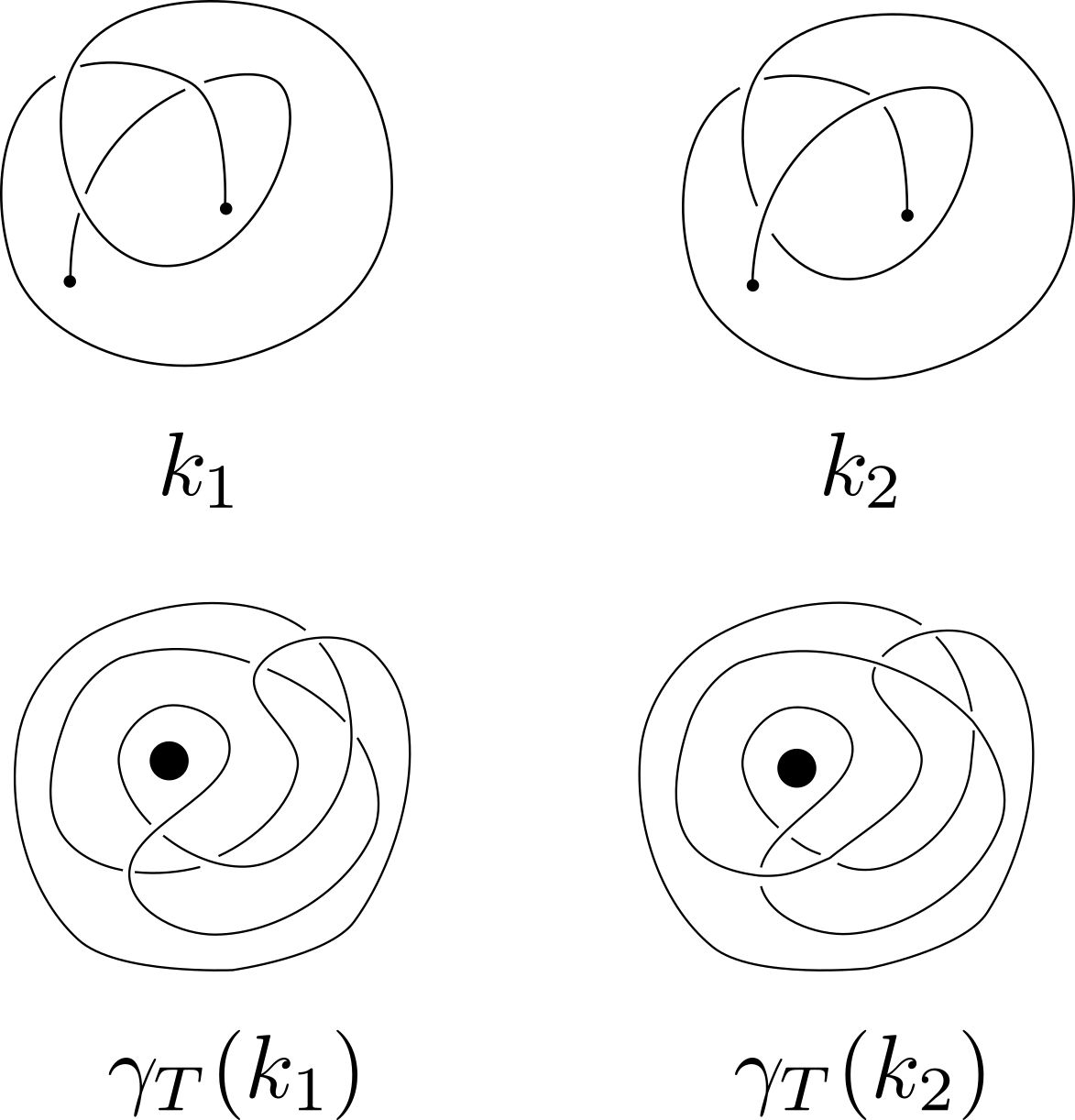}
\caption{On the top, the knotoids $k_1$ and $k_2$ in $\mathbb{K}(\mathbb{R}^2)$. On the bottom, their images under $\gamma_T$.}
\label{fig:dimos}
\end{figure}

The images of knotoids $k_1$ and $k_2$ under $\gamma_T$ are two knots in the solid torus. To distinguish them, we can consider the following construction. We can embed the solid torus in $S^3$ as done in Section \ref{miserve}, but this time after giving a full twist along the meridian of $S^1 \times D^2$. We then obtain two knots in $S^3$, shown in Figure \ref{fig:dimostoro}, that can be easily shown to be the knots $9_{46}$ and $5_2$ by computing the Alexander and Jones polynomials. These invariants are in fact enough to distinguish them since, according to knotinfo \cite{knotinfo}, the knots $5_2$ and $9_{46}$ are uniquely determined by their Alexander and Jones polynomials among all knots up to $12$ crossings. Note that this procedure may be applied to several similar cases, highlighting the power of the map $\gamma_T$. We emphasise that the authors are not aware of any other method other than using $\gamma_T$ that is capable of distinguishing $k_1$ and $k_2$. This example was kindly suggested by Dimos Goundaroulis.

\begin{figure}[h]
\includegraphics[width=6cm]{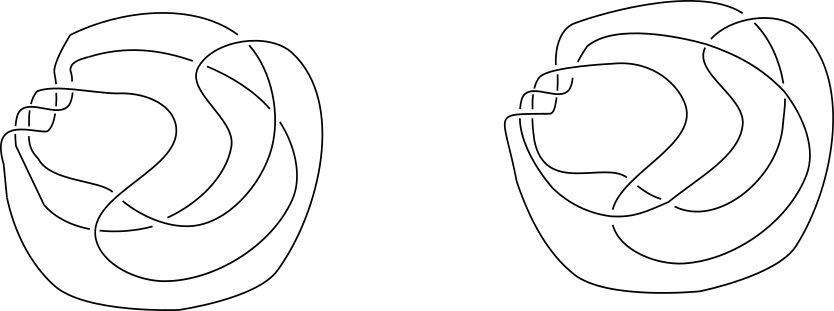}
\caption{By embedding the solid torus in $S^3$ as in Section \ref{miserve} after giving a full twist along the meridian, we obtain a pair of knots in $S^3$.}
\label{fig:dimostoro}
\end{figure}

\begin{rmk}
As mentioned in Section \ref{miserve}, after applying the map $\gamma_T$ one could directly compare the resulting knots in the solid torus by using invariants for knots in the solid torus, see \emph{e.g.}~\cite{sophia}, \cite{maciei} and \cite{hoste}.
Alternatively, one could also consider the two-component link $L$ in $S^3$ obtained as the union of $\gamma_T(k)$ with the meridian of the solid torus $S^1 \times D^2$. 
 
\end{rmk}

\section{Gauss Code and Computations}\label{sec:gauss}
The oriented Gauss code $GC(D)$ for a knot diagram $D$ is a pair $(C,S)$, where $C$ is a $2n$-tuple and $S$ an $n$-tuple, $n$ being the number of crossings of the diagram. Given a diagram $D$, $GC(D)$ is constructed as follows: assign a number between $1$ and $n$ to each crossing, and pick a point $a$ in the diagram, which is not a double point. Start \emph{walking} along the diagram from $a$, following the orientation, and record every crossing encountered (in order) by adding an entry to $C$ consisting of the corresponding number, together with a sign $+$ for overpassing and $-$ for underpassing, until you reach the starting point $a$ again. Note that each crossing is encountered twice. $S$ is the $n$-tuple whose $i$th entry is equal to $1$ if the $i$th crossing is positive and $-1$ otherwise. As an example, the Gauss code associated to the diagram in Figure \ref{fig:trifogliocode} is equal to: $$ GC(D)= ((1,-2,3,-1,2,-3),(  1,1,1)) $$ 

\begin{figure}[h]
\includegraphics[width=3cm]{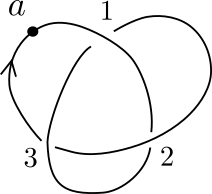}
\caption{Computing the Gauss code for a knot diagram.}
\label{fig:trifogliocode}
\end{figure}

 Gauss codes can be easily extended to knotoid diagrams, see \cite{knotoids}. The procedure is basically the same, but in this case the starting point $a$ coincides with the tail of the diagram. As an example, the Gauss code for the knotoid in Figure \ref{fig:knotoidcode} is equal to:  $$GC(D) = ((-1, 1,2,-3,-2,3), ( 1,1,1)) $$

\begin{figure}[h]
\includegraphics[width=3cm]{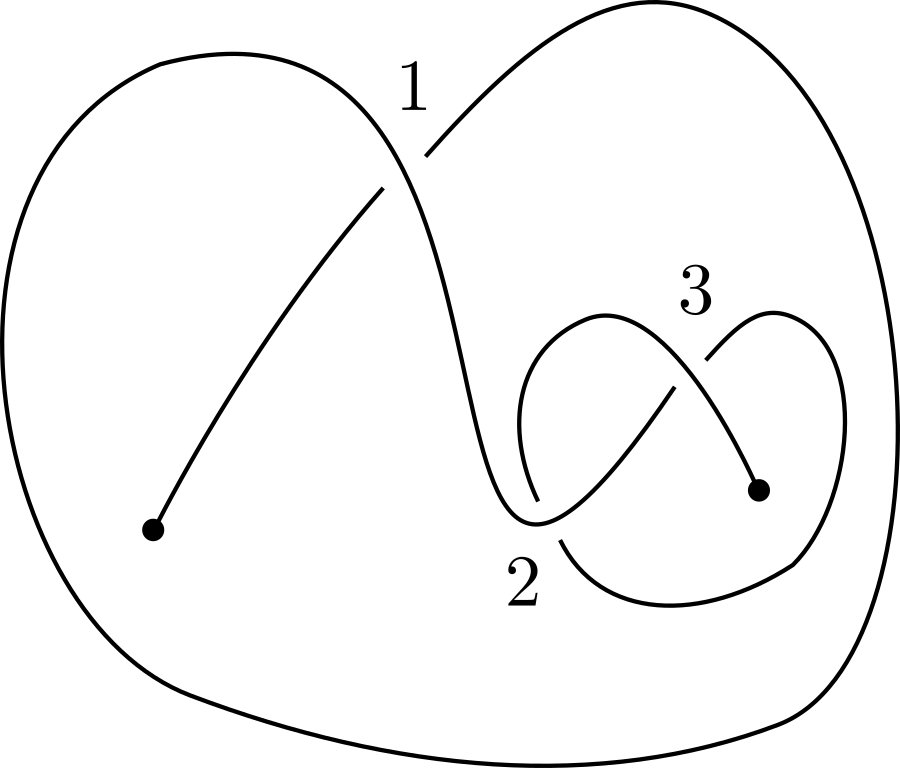}
\caption{Computing the Gauss code for a knotoid.}
\label{fig:knotoidcode}
\end{figure}

The information encoded in $GC(D)$ is enough to reconstruct $D$, both in the case of knot and knotoid diagrams.

\subsection{Generalised Gauss code for knotoids}
Gauss code for knotoid diagrams may be generalised to contain also the information about the intersection with the branching set. We will call the \emph{generalised Gauss code} (indicated as $gGC(D)$) the pair $(C,S)$ where $S$ is the same as in $GC(D)$, while $C$ contains also entries equal to $b$ every time every time the diagram intersects with the arcs that connect the branched points (\emph{i.e.} the endpoints) with the boundary of the disk containing the diagram. For instance, the Gauss code for the knotoid in Figure \ref{fig:knotoidcodebr} is: $$-1,b,b, 1,2,-3,b,-2,3 / 1,1,1 $$

\begin{figure}[h]
\includegraphics[width=3cm]{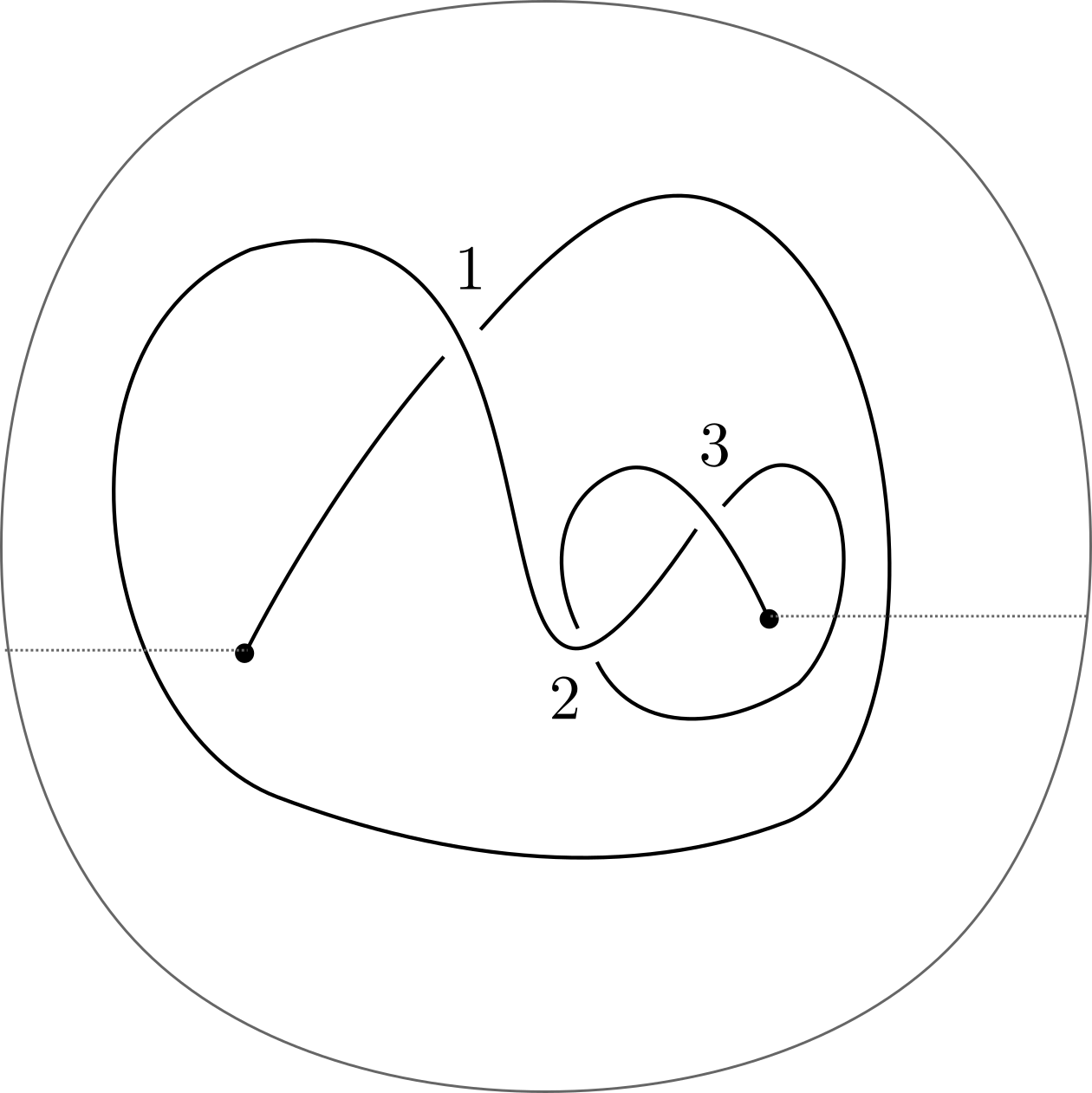}
\caption{Computing the generalised Gauss code for a knotoid.}
\label{fig:knotoidcodebr}
\end{figure}

\subsection{Gauss code for the lifts} \label{gclift}
Given a diagram $D$ representing a knotoid $k$ with $gGC(D) = (C,S)$ it is possible to compute $GC(\gamma_S(D))$, where $\gamma_S(D)$ is the diagram representing $\gamma_S(k)$ obtained with the ``cuts'' technique, as in Figure \ref{rivestimento}. 

\begin{figure}[h]
\includegraphics[width=7cm]{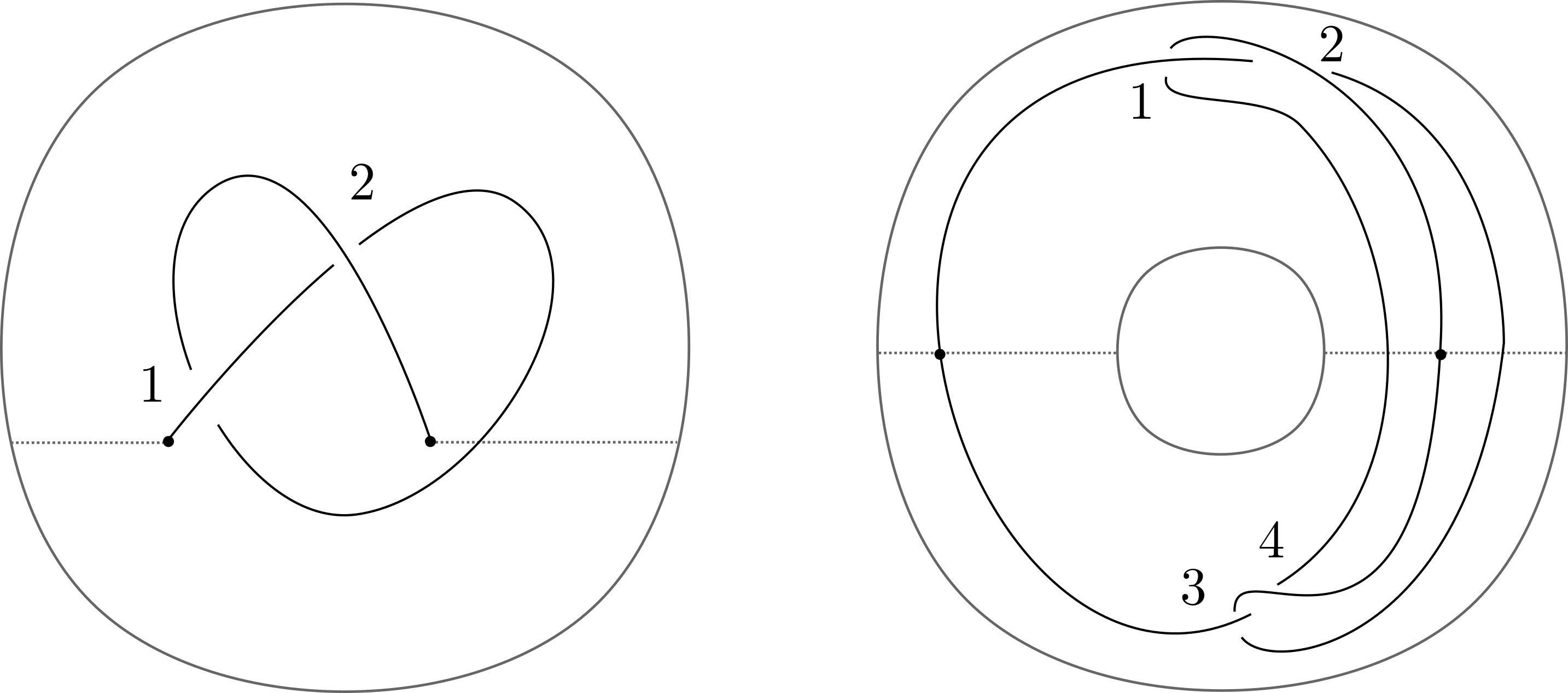}
\caption{Computing the Gauss code for the lift: an example.}
\label{fig:esempiogc}
\end{figure}

Consider the knotoid diagram $D$ on the left-side of Figure \ref{fig:esempiogc} with $gGC(D)=(( 1,-2,b,-1,2),(1,1) $. Label the crossings in $\gamma_S(D)$ as shown on the right-side of Figure \ref{fig:esempiogc}: every half of the annulus is a copy of the disk in which $D$ lies, keep the same enumeration on the top-half and increase by $2$ the labels in the bottom one.
Now, start computing $GC(\gamma_S(D))$: notice that until we reach an intersection between the diagram and one of the arcs splitting the annulus, the entries added in $GC(\gamma_S(D))$ are equal to the first entries in $gGC(D)$. After an intersection point, the path continues on the bottom half of the annulus, and the next entries added in $GC(\gamma_S(D))$ are equal to the corresponding ones in $gGC(D)$, but with every label increased by $2$. Once we reach the lift of the head, the path along the knot continues, and it is the same path we have just done, but in the opposite direction and on opposite halves of the annulus. Thus, the last entries added are a copy of the entries written so far, added in the opposite order and with labels corresponding to opposite halves of the annulus and thus: $$GC(\gamma_S(D)) = ((1,-2,-3,4,2,-1,-4,3), S)$$

To compute $S$, note that the sign of a crossing in the top-half is the same as its corresponding crossing in the bottom-half. Moreover, since the labels corresponding to each crossings in $gGC(D)$ appear once before the entry $b$ and once after, the signs of the first two crossings in the knot diagram are changed, and the complete Gauss code is $$GC(\gamma_S(D)) = ((1,-2,-3,4,2,-1,-4,3),(-1,-1,-1))$$

\begin{figure}[h]
\includegraphics[width=4cm]{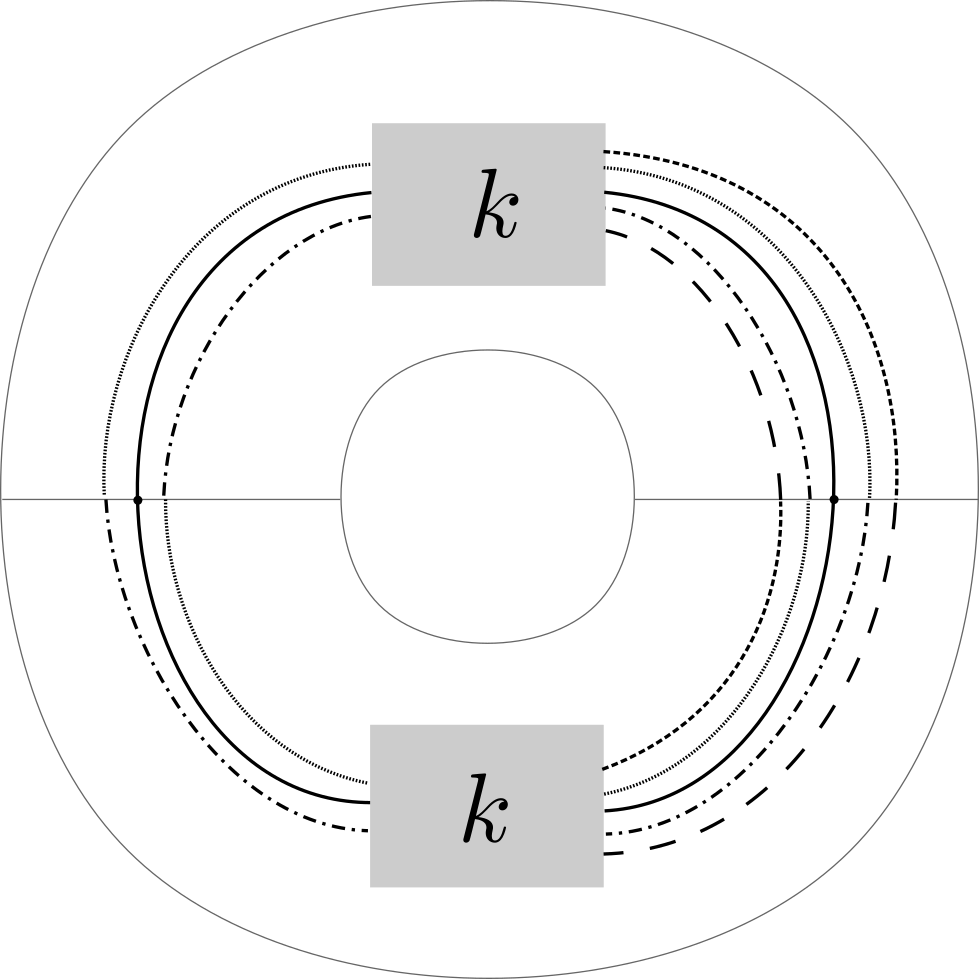}
\caption{Computing the Gauss code of the lift $\gamma_S(D)$.}
\label{fig:exgclift}
\end{figure}

The previous procedure can be generalised to produce an algorithm. Thus, consider the diagram in Figure \ref{fig:exgclift}, and start walking along the knot from the lift of the tail. Every time we pass from one half of the annulus to the other, the path on the diagram follows as in the knotoid diagram, but on a different half. Moreover, as before, once we reach the lift of the head of the knotoid, the path proceeds as the one just traced, in the opposite direction and on different halves as before. Now, suppose that on $gGC(D)$ the two appearances of the same label happen without the occurrence of a $b$ entry between them. This means that in $\gamma_S(D)$ we are going to reach the top-lift of the crossing twice without passing to the other half (thus, without swapping the orientation), and the same holds for the bottom-lift of the crossings. In this case the signs of both the lifted crossings in $\gamma_S(D)$ are equal to the sign of the corresponding one in $D$.
Putting everything together, we obtain the following algorithm, that can be easily implemented.\\

\textbf{Input}: generalised Gauss code of the knotoid, \textbf{n} = number of crossings in the knotoid diagram;

\begin{itemize}
 \item  go through the knotoid code: copy the entries until you find a $b$;
 \item  until you reach the point corresponding to the head of the knotoid: after reaching a $b$
\begin{itemize}
  \item if the number of $b$-entries encountered is odd, add entries equals to the knotoid ones, but changing the labels by adding $n$ to them. Do that until you reach another $b$; 
  \item if the number of $b$-entries encountered is even, add entries equals to the knotoid ones. Do that until you reach another $b$.
\end{itemize}

 \item After reaching the head: copy the entries added so far, starting from the last one, and changing the labels by subtracting $n$ if they are greater than $n$, and adding $n$ otherwise;
 \item Consider the $k$ crossing in the knotoid diagram:
 \begin{itemize}
  \item if the corresponding labels in the knotoid code appear twice with an even (or zero) number of $b$-entries between them, then the sign of the $k$ and $k+n$ crossings in the knot diagram are equal to the sign of the starting crossing;
  \item if the corresponding labels in the knotoid code appear twice with an odd number of $b$-entries between them, then the sign of the $k$ and $k+n$ crossings in the knot diagram are opposite to the sign of the starting crossing.
 \end{itemize}

\end{itemize}

\textbf{Output}: Gauss code for the lifted knot diagram.

\bibliographystyle{amsplain}

\end{document}